\documentclass[a4,10pt]{amsart}
\usepackage{standalone}  % 加载 standalone 宏包
\usepackage{tikz}
\usepackage{booktabs,tabularx,array}
\usepackage{placeins}
 \newcolumntype{Y}{>{\raggedright\arraybackslash}X}
\usepackage{float}
\usetikzlibrary{calc, intersections, arrows.meta, decorations.markings}
\usepackage{amssymb,amstext,amscd,amsfonts,mathdots,mathrsfs}
\usepackage[hidelinks]{hyperref}
\usepackage{pdflscape}
\usepackage{setspace}% for spacing
\usepackage[all]{xy}% for quiver
\usepackage{enumerate}% for [(1)] or [(a)]
\usepackage{indentfirst}%for indent
\usepackage{ytableau}
\usepackage{color}% for color
\usepackage{colortbl}
\usepackage{dsfont}
\usepackage{multicol}
\usepackage{amsmath}% for the equation label
\usepackage{tikz}% for tikz quiver
\usetikzlibrary{shapes.geometric, arrows}
\usetikzlibrary{calc}
\def\arraystretch{1.2} %for the stretch of array
\numberwithin{equation}{section}
 \definecolor{darkred}{HTML}{993333}
\definecolor{darkblue}{RGB}{20,40,120}
\newcommand{\clr}{darkblue}

\def\dim{\text{dim}}
\def\Ob{\text{Ob}}
\def\Vect{\text{Vect}}

\def\op{\text{op}}
\def\Id{\text{Id}}
\def\Z{\mathbb{Z}}
\def\N{\mathbb{N}}
\def\id{\text{id}}
\def\Obj{\text{Obj}}
\def\R{\mathbb{R}}
\def\k{\Bbbk}
\def\P{\mathcal{P}}
\def\pcl{P^+_{cl,k}}
\def\ssum{\textstyle\sum\limits}
\def\dim{\operatorname{dim}}
\def\Hom{\operatorname{Hom}}
\def\End{\operatorname{End}}
\def\max{\operatorname{max}}
\def\ev{\operatorname{ev}}
\def\deg{\operatorname{deg}}
\def\res{\operatorname{res}}
\def\ch{\operatorname{char}}

\newtheorem{Theorem}{Theorem}[section] %[chapter] theorem number will %continue
\newtheorem{theorem*}{Theorem}[section]
\newtheorem{Lemma}[Theorem]{Lemma}
\newtheorem{Cor}[Theorem]{Corollary}

\newtheorem{Prop}[Theorem]{Proposition}
\newtheorem{Defn}[Theorem]{Definition}

\newtheorem{rem}[Theorem]{Remark}
\newtheorem{Assumption}[Theorem]{Assumption}
\theoremstyle{definition}
\newtheorem{THEOREM}{Theorem}
\def\theTHEOREM{\Alph{THEOREM}}

\newcommand{\rclr}{rgb:red,4}
\newcommand{\wdot}{ node[circle, draw, fill=white, thick, inner sep=0pt, minimum width=4pt]{}}
\newcommand{\bdot}{ node[circle, draw, fill=\clr, thick, inner sep=0pt, minimum width=4pt]{}}
\newcommand{\rdot}{ node[circle, draw=darkred, fill=darkred, thick, inner sep=0pt, minimum width=4pt]{}}
\newcommand{\bd}[1]{\mathbf{#1}}
\newcommand{\ob}[1]{\mathsf{#1}}
\newcommand{\up}{\uparrow}
\newcommand{\down}{\downarrow}
\newcommand{\sdim}{\text{sdim}}
\newcommand{\C}{\mathcal{C}}
\newcommand{\OBC}{\mathcal{OBC}}
\newcommand{\AOB}{\mathcal{AOB}}
\newcommand{\COB}{\mathcal{COB}}
\newcommand{\AOBC}{\mathcal{AOBC}}
\newcommand{\unit}{\mathds{1}}
\newcommand{\OB}{\mathcal{OB}}
\newcommand{\B}{\mathcal{B}}
\newcommand{\CB}{\mathcal{CB}}
\newcommand{\AB}{\mathcal{AB}}
\newcommand{\fgl}{\mathfrak{gl}}
\newcommand{\fq}{\mathfrak{q}}
\newcommand{\fh}{\mathfrak{h}}
\newcommand{\fn}{\mathfrak{n}}
\newcommand{\fb}{\mathfrak{b}}
\newcommand{\wrd}{\langle\up,\down\rangle}
\definecolor{darkblue}{HTML}{111199}
\newcommand{\fulldot}{
    \begin{tikzpicture}[color=\clr]
    \draw (0,0) \bdot;
    \end{tikzpicture}
}
\newcommand{\emptydot}{
    \begin{tikzpicture}[color=\clr]
    \draw (0,0) \wdot;
    \end{tikzpicture}
}
\newcommand{\undot}[1]{\operatorname{undot}({#1})}
\newcommand{\p}[1]{|{#1}|}

\def\clock{\begin{tikzpicture}[baseline=-.9mm]
\filldraw[white] (0,0) circle (1.72mm);
\draw[-] (0,-0.18) to[out=180,in=-102] (-.178,0.02);
\draw[-] (-0.18,0) to[out=90,in=180] (0,0.18);
\draw[-] (0.18,0) to[out=-90,in=0] (0,-0.18);
\draw[-] (0,0.18) to[out=0,in=90] (0.18,0);
\end{tikzpicture}\,}
\newcommand{\cE}{\mathsf{E}}
\newcommand{\cF}{\mathsf{F}}
\newcommand{\excise}[1]{}
\newcommand{\thal}{{}^{\theta}\!\alpha}
\newcommand{\So}{\mathsf{S}}
\def\iLam{\Lambda^{\!\jmath}}
\def\iH{\H^{\jmath}}

\newcommand{\lcap}{
\begin{tikzpicture}[baseline = 3pt, scale=0.5]
        \draw[-,thick] (1,0) to[out=up, in=right] (0.53,0.5) to[out=left, in=right] (0.47,0.5);
        \draw[-,thick] (0.49,0.5) to[out=left,in=up] (0,0);
\end{tikzpicture}
}
\newcommand{\lcup}{
\begin{tikzpicture}[baseline = 6pt, scale=0.5]
        \draw[-,thick] (1,1) to[out=down, in=right] (0.53,0.5) to[out=left, in=right] (0.47,0.5);
        \draw[-,thick] (0.49,0.5) to[out=left,in=down] (0,1);
\end{tikzpicture}
}
\newcommand{\lcapl}{
\begin{tikzpicture}[baseline = 3pt, scale=0.5]
        \draw[<-,thick] (1,0) to[out=up, in=right] (0.53,0.5);
         \draw[-,thick]  (0.53,0.5) to[out=left, in=right] (0.47,0.5);
        \draw[-,thick] (0.49,0.5) to[out=left,in=up] (0,0);
\end{tikzpicture}
}
\newcommand{\lcupl}{
\begin{tikzpicture}[baseline = 6pt, scale=0.5]
        \draw[<-,thick] (1,1) to[out=down, in=right] (0.53,0.5);
        \draw[-,thick](0.53,0.5) to[out=left, in=right] (0.47,0.5);
        \draw[-,thick] (0.49,0.5) to[out=left,in=down] (0,1);
\end{tikzpicture}
}
\newcommand{\lcuplL}{
\begin{tikzpicture}[baseline = 6pt, scale=0.5]
        \draw[<-,thick] (1,1) to[out=down, in=right] (0.53,0.5);
        \draw[-,thick](0.53,0.5) to[out=left, in=right] (0.47,0.5);
        \node at (-0.15, 0.85) {$\color{darkblue}\scriptstyle\bullet$};
        \draw[-,thick] (0.49,0.5) to[out=left,in=down] (0,1);
\end{tikzpicture}
}
\newcommand{\lcuplr}{
\begin{tikzpicture}[baseline = 6pt, scale=0.5]
        \draw[<-,thick] (1,1) to[out=down, in=right] (0.53,0.5);
        \draw[-,thick](0.53,0.5) to[out=left, in=right] (0.47,0.5);
        \node at (1.15, 0.85) {$\color{darkblue}\scriptstyle\bullet$};
        \draw[-,thick] (0.49,0.5) to[out=left,in=down] (0,1);
\end{tikzpicture}
}
\newcommand{\lcapo}{
\begin{tikzpicture}[baseline = 3pt, scale=0.5]
        \draw[->,thick] (1,0) to[out=up, in=right] (0.53,0.5) to[out=left, in=right] (0.47,0.5);
        \draw[-,thick] (0.49,0.5) to[out=left,in=up] (0,0);
\end{tikzpicture}
}
\newcommand{\lcupo}{
\begin{tikzpicture}[baseline = 6pt, scale=0.5]
        \draw[-,thick] (1,1) to[out=down, in=right] (0.53,0.5) to[out=left, in=right] (0.47,0.5);
        \draw[->,thick] (0.49,0.5) to[out=left,in=down] (0,1);
\end{tikzpicture}
}
\newcommand{\rcap}{
\begin{tikzpicture}[baseline = 3pt, scale=0.5]
        \draw[<-,thick] (1,0) to[out=up, in=right] (0.53,0.5) to[out=left, in=right] (0.47,0.5);
        \draw[-,thick] (0.49,0.5) to[out=left,in=up] (0,0);
\end{tikzpicture}
}
\newcommand{\rcup}{
\begin{tikzpicture}[baseline = 6pt, scale=0.5]
        \draw[<-,thick] (1,1) to[out=down, in=right] (0.53,0.5) to[out=left, in=right] (0.47,0.5);
        \draw[-,thick] (0.49,0.5) to[out=left,in=down] (0,1);
\end{tikzpicture}
}
\newcommand{\swap}{
\begin{tikzpicture}[baseline = 3pt, scale=0.4]
        \draw[-,thick] (0,0) to[out=up, in=down] (1,1);
        \draw[-,thick] (1,0) to[out=up, in=down] (0,1);
\end{tikzpicture}
}
\newcommand{\swapo}{
\begin{tikzpicture}[baseline = 3pt, scale=0.5]
        \draw[->,thick] (0,0) to[out=up, in=down] (1,1);
        \draw[->,thick] (1,0) to[out=up, in=down] (0,1);
\end{tikzpicture}
}

\newcommand{\swapr}{
\begin{tikzpicture}[baseline = 3pt, scale=0.5]
        \draw[<-,thick] (0,0) to[out=up, in=down] (1,1);
        \draw[<-,thick] (1,0) to[out=up, in=down] (0,1);
\end{tikzpicture}
}
\newcommand{\rswap}{
\begin{tikzpicture}[baseline = 3pt, scale=0.5]
        \draw[->,thick] (0,0) to[out=up, in=down] (1,1);
        \draw[<-,thick] (1,0) to[out=up, in=down] (0,1);
\end{tikzpicture}
}

\newcommand{\lswap}{
\begin{tikzpicture}[baseline = 3pt, scale=0.5]
        \draw[<-,thick] (0,0) to[out=up, in=down] (1,1);
        \draw[->,thick] (1,0) to[out=up, in=down] (0,1);
\end{tikzpicture}
}
\newcommand{\dswap}{
\begin{tikzpicture}[baseline = 3pt, scale=0.5]
        \draw[<-,thick] (0,0) to[out=up, in=down] (1,1);
        \draw[<-,thick] (1,0) to[out=up, in=down] (0,1);
\end{tikzpicture}
}
\newcommand{\caup}{
\begin{tikzpicture}[baseline = 5pt, scale=0.75]
                 \draw[<-,thick] (0.75,1) to[out=down,in=left] (1,0.6) to[out=right,in=down] (1.25,1);
          \draw[<-,thick] (1.25,0) to[out=up,in=right] (1,0.4) to[out=left,in=up] (0.75,0);\end{tikzpicture}}

\newcommand{\xd}{
\begin{tikzpicture}[baseline = 3pt, scale=0.5, color=\clr]
\draw[->,thick] (0,0) to[out=up, in=down] (0,1);
\draw(0,0.5) \bdot;
 \draw (0.35,0.4) node{\tiny $2$};
\end{tikzpicture}
}
\newcommand{\xdx}{
\begin{tikzpicture}[baseline = 3pt, scale=0.5, color=\clr]
\draw[<-,thick] (0,0) to[out=up, in=down] (0,1);
\draw(0,0.5) \bdot;
 \draw (0.35,0.4) node{ \tiny$ 2$};\end{tikzpicture}
}
\newcommand{\xdz}{
\begin{tikzpicture}[baseline = 3pt, scale=0.5, color=\clr]
\draw[->,thick] (0,0) to[out=up, in=down] (0,1);
\draw(0,0.5) \bdot;
\end{tikzpicture}
}
\newcommand{\xli}{
\begin{tikzpicture}[baseline = 3pt, scale=0.5, color=\clr]
        \draw[<-,thick] (0,0) to[out=up, in=down] (0,1);
\end{tikzpicture}
}
\newcommand{\sli}{
\begin{tikzpicture}[baseline = 3pt, scale=0.5, color=\clr]
        \draw[->,thick] (0,0) to[out=up, in=down] (0,1);
\end{tikzpicture}
}
\newcommand{\xdxz}{
\begin{tikzpicture}[baseline = 3pt, scale=0.5, color=\clr]
\draw[<-,thick] (0,0) to[out=up, in=down] (0,1);
\draw(0,0.5) \bdot;
 \end{tikzpicture}
}

\newcommand{\obg}{\begin{tikzpicture}[baseline = 10pt, scale=0.5, color=\clr] \draw[-,thick] (0,0.5)to[out=up,in=down](0,1.2);
    \end{tikzpicture}
    }
\def\bigdt{{\color{white}\bullet}\hspace{.1mm}\!\!\!\pmb\circ}
\def\dt{{\color{white}\bullet}\!\!\!\circ}
\def\bull{{\scriptstyle\bullet}}

\newcommand{\xdotk}{
\begin{tikzpicture}[baseline = 3pt, scale=0.5, color=\clr]
        \draw[-,thick] (0,0) to[out=up, in=down] (0,1);
            \draw (0,0.4)\bdot;
        \draw (0.5,0.4) node{$k$};
\end{tikzpicture}
}
\newcommand{\xdot}{\begin{tikzpicture}[baseline = -1mm, color=\clr]
	\draw[-] (0.08,-.3) to (0.08,.4);
      \node at (0.08,0.15) {$\dt$};
\end{tikzpicture}}
\newcommand{\cev}[1]{\reflectbox{\ensuremath{\vec{\reflectbox{\ensuremath{#1}}}}}}
\newcommand{\xdota}{
\begin{tikzpicture}[baseline = 3pt, scale=0.5, color=\clr]
        \draw[-,thick] (0,0) to[out=up, in=down] (0,1);
            \draw (0,0.4)\bdot;
        \draw (0.5,0.4) node{$a$};
\end{tikzpicture}
}

\newcommand{\hf}{\diamond}
\newcommand{\darkg}{\color{green!70!black}}
\tikzset{darkg/.style={green!70!black}}

\newcommand{\comment}[1]{{\color[rgb]{0.1,0.7,0.1} #1}}
\newcommand{\old}[1]{{\color{red} #1}}
\newcommand{\new}[1]{{\color{blue} #1}}
\tikzset{
	cplus/.style={
		circle,
		draw=black,
		fill=white,
		line width=0.2pt,
		inner sep=0.025pt,
		minimum size=.2pt,
		font=\fontsize{.5}{.8}\selectfont
	}
}
\newcommand{\midscript}[1]{{\fontsize{6pt}{7pt}\selectfont$#1$}}

\begin{document}
\setlength{\baselineskip}{14pt}

\title[Affine Brauer categorification]
{Affine and cyclotomic Brauer categorification via $\imath$-Kac--Moody $2$-categories: the half-integral type 
 $\operatorname{AIII}$ case
}

\author{Mengmeng  Gao}
\email{g19920119@163.com}
\author{Hebing  Rui}
\email{hbrui@tongji.edu.cn}
\author{Linliang Song}
\email{llsong@tongji.edu.cn}
\address{School of Mathematical Sciences, Tongji University, Shanghai, 200092, China.}
%\date{\today}
\thanks{M. Gao is  partially supported by NSFC (grant No. 12301038), H. Rui is partially supported   by NSFC (grant No.  11971351), and  L. Song is  partially supported   by NSFC (grant No.  12071346). }

\begin{abstract}
 Cyclotomic Brauer or cyclotomic   Nazarov–Wenzl algebras arise in higher Schur–Weyl dualities involving parabolic categories $\mathcal O$ for Lie algebras of types $B,C$ and $D$. 
  Their connection with Kac--Moody-type categorification is substantially less developed than the corresponding type $A$ theory for cyclotomic Hecke algebras.

We construct a categorical bridge between affine Brauer-type representation theory and the half-integral quasi-split type $\operatorname{AIII}$ $\imath$-Kac--Moody $2$-category of Bao–Shan–Wang–Webster. More precisely, an action of the affine Brauer category on a locally Schurian category, with dot spectrum exactly $\frac12+\mathbb Z$, determines a generalized nilpotent 2-representation of the even component $\mathfrak  U^{\imath}_{+}$. Conversely, every nilpotent $2$-subrepresentation of an ambient locally Schurian $2$-representation of $\mathfrak  U^{\imath}_{+}$ carries a compatible affine Brauer action whose dot spectrum is contained in $\frac12+\mathbb Z$. Applying these constructions to cyclotomic quotients, we prove that the locally unital algebra attached to a $\mathbf u$-admissible cyclotomic Brauer category is isomorphic to the locally unital  algebra attached to the corresponding cyclotomic quotient of the principal $2$-representation of $\mathfrak  U^{\imath}_{+}$. Consequently, the associated cyclotomic Brauer (or cyclotomic Nazarov–Wenzl) algebras acquire natural $\mathbb Z$-gradings. To our knowledge, this is the first such categorical realization of 
$\mathbf u$-admissible half-integral cyclotomic Brauer algebras by means of an $\imath$-Kac--Moody $2$-category. It provides the categorical and graded framework toward a Brauer-type extension of the Brundan–Kleshchev–Ariki theory in which coideal algebras and $\imath$-canonical bases are expected to replace  ordinary quantum groups and canonical bases.
\end{abstract}

\maketitle

\tableofcontents
%%%%%%%%%%%%%%%%%%%%%%%%%%%%%%%%%%%%%%%%%%%%%%%%%%%%%%%%
\section{Introduction}
One of the central developments in the representation theory of Hecke algebras is the passage from ungraded algebras to graded diagrammatic categorification. Ariki’s theorem identifies the decomposition numbers of cyclotomic Hecke algebras with coefficients of canonical bases \cite{Ariki-dec}. Brundan and Kleshchev subsequently constructed an explicit isomorphism between blocks of cyclotomic Hecke algebras and cyclotomic quiver Hecke algebras \cite{BK-klr}. Besides producing a natural grading on the Hecke-algebra side, their isomorphism connects the representation theory of cyclotomic Hecke algebras directly with Kac--Moody categorification and leads to graded refinements of decomposition numbers \cite{BK-gr}. This circle of ideas is one of the most effective instances in which a categorical quantum-group action converts canonical-basis information into concrete representation-theoretic invariants.

The purpose of the present paper is to establish a new categorical
bridge for the affine Brauer category and its cyclotomic quotients.
Rui and Song introduced these categories and constructed higher
Schur--Weyl duality functors from cyclotomic Brauer categories to
parabolic BGG categories $\mathcal O$ associated with orthogonal and
symplectic Lie algebras, namely Lie algebras of types $B$, $C$, and
$D$~\cite{RS-cyc}. Affine and cyclotomic Nazarov--Wenzl algebras arise
as endomorphism algebras in these diagrammatic categories.
However, the analogue of the Brundan–Kleshchev–Ariki picture is more intricate in this setting. The relevant quantum symmetry is a quantum symmetric pair, or coideal algebra, rather than an ordinary quantum group; diagrammatically, one must account for cups, caps, unoriented Brauer strands, and boundary phenomena in addition to the quiver-Hecke-type crossings and dots. These features are not formal modifications of the type $A$ theory.

The categorical symmetry used here is the quasi-split type $\operatorname{AIII}$ $\imath$-Kac--Moody $2$-category $\mathfrak  U^{\imath}$  introduced in the half-integral case by Bao–Shan–Wang–Webster \cite{BSWW-icate}. Although the name “type $\operatorname{AIII}$” refers to the underlying symmetric pair and its coideal algebra, the resulting action provides a categorical framework for the
Brauer-type categories arising in the relevant Lie-theoretic
realizations.

This distinction is important: the paper is not another ordinary
type~$A$ categorification. It relates a quantum-symmetric-pair
categorification to the diagrammatic centralizer algebras attached to
orthogonal and symplectic representation theory. The present project
was initiated in 2022 and grows out of our earlier study of affine and
cyclotomic Brauer categories~\cite{RS-cyc} and the generating-series
identities appearing in~\cite{GRS-tri}. While this work was in progress,
Brundan--Wang--Webster noted the general expectation that the
quasi-split type~$\operatorname{AIII}$ $2$-$\imath$-quantum groups
should be related to the affine Brauer category of~\cite{RS-cyc}, and
raised the problem of finding a conceptual explanation analogous to
the type~$A$ Heisenberg--Kac--Moody correspondence
\cite[Introduction, pp.~5--6]{BWW25}. Savage and Webster recorded a
related expectation for the affine Brauer and Kauffman categories
\cite[Introduction]{SW}. Theorem~A makes this general expectation
precise in the half-integral type~$\operatorname{AIII}$ setting by
providing explicit constructions in both directions.

Concrete $2$-representations of $\imath$-Kac--Moody $2$-categories remain much less understood than their ordinary Kac--Moody counterparts. One obstruction is the complexity of the defining relations. In the half-integral type $\operatorname{AIII}$ case, the boundary color $\frac12$ is governed by an inhomogeneous relation with no analogue in the ordinary type $A$ Kac--Moody $2$-category. Consequently, it is difficult to recognize an $\imath$-Kac--Moody $2$-action directly on a naturally occurring representation-theoretic category. Our first main result shows that the affine Brauer category provides precisely such a concrete model.

To formulate this result precisely, we first introduce  the relevant component of the
$\imath$-Kac--Moody $2$-category and fix some terminology for its $2$-representations.

The $2$-category $\mathfrak U^\imath$ is the
disjoint union of its even and odd full sub-$2$-categories, and these two
components are isomorphic. We therefore formulate the theorem for the even
component $\mathfrak U^\imath_+$. 
Throughout, unless otherwise stated, every ambient
$2$-representation of $\mathfrak U^\imath_+$ considered in this
paper is assumed to be either locally finite Abelian or locally
Schurian  following~\cite{BSW-heisenberg}.
We adopt from  \cite{BSW-heisenberg} the notion of a nilpotent $2$-representation
and use the same terminology for a  $2$-subrepresentation of an ambient locally
Schurian $2$-representation when  
the two conditions in Definition~\ref{nil} hold for all of its objects.
The notion of a generalized nilpotent
$2$-representation, introduced in the present paper, is defined in
Section~3. For brevity, we call a generalized nilpotent
$2$-representation of $\mathfrak U^\imath_+$ a generalized nilpotent
$\imath$-Kac--Moody categorification.

Let $\mathcal{AB}$  denote the affine Brauer category of Rui–Song \cite{RS-cyc}. If $\mathcal C$ is a locally Schurian category equipped with an $\mathcal{AB}$-module-category structure, the dot endomorphism of the generating functor decomposes that functor into generalized eigenspace functors. When its spectrum is the full half-integral set
\begin{equation} I=\frac 12+\mathbb Z, \end{equation}
we use these eigenspace functors to construct a $2$-representation  of the even component $\mathfrak U^{\imath}_{+}$, including the four orientations of crossings and the left and right cups and caps. 
Conversely, starting from a nilpotent $2$-subrepresentation of an
ambient locally Schurian $2$-representation of
$\mathfrak U^\imath_+$, we combine the generating functors and
natural transformations to recover an affine Brauer action.
 The precise statement is as follows.

\begin{THEOREM}\label{main111}  Let $\mathcal C$ be a locally Schurian category over $\mathbb C$.
\begin{itemize}
\item [(1)]  Suppose that $\mathcal C$ admits an affine Brauer categorification and that the spectrum of the dot endomorphism on the generating endofunctor is exactly $\frac12+\mathbb Z$. Then $\mathcal C$ carries the structure of a generalized nilpotent $2$-representation of $\mathfrak U^{\imath}_{+}$. Moreover, the full subcategory $\mathcal C^{\mathrm{fg}}$ of finitely generated objects is a nilpotent $2$-subrepresentation. In particular, if $\mathcal C$ is locally finite Abelian, then the resulting $2$-representation is nilpotent.
\item[(2)] Suppose that $\mathcal C$ carries a $2$-representation of $\mathfrak U^{\imath}_{+}$, and let $\mathcal M$ be a nilpotent $2$-subrepresentation. Then $\mathcal M$ admits an affine Brauer categorification for which the dot spectrum is contained in $\frac{1}{2}+\mathbb Z$.
\end{itemize}
\end{THEOREM}

Theorem~\ref{main111} should be compared with the type~$A$ relationship between
Heisenberg and Kac--Moody categorical actions developed by
Brundan--Savage--Webster
\cite[Sections~4 and~5.5]{BSW-heisenberg}. 
The construction and verification in the present setting, however, take place in a different diagrammatic 
category and are not direct adaptations of the type $A$ arguments. Most importantly,  the exceptional relation at the boundary
color $\frac12$ is inhomogeneous and does not yield the inversion isomorphism available  in the  type~$A$ setting.  Consequently,  the remaining oriented generators cannot be recovered by that method. 

We therefore construct all generating $2$-morphisms explicitly,
including the four orientations of crossings and both the left and
right cups and caps, and verify the defining relations directly.
We also introduce the notion of a generalized nilpotent
$2$-representation, which is adapted to locally Schurian categories
and is essential for the application to finitely generated projective
modules in Section~8.

Part (1) is proved in Theorem~\ref{thm:mainthmaffineb}. Part (2) follows from the reverse construction of Section~7. These statements are deliberately formulated with different spectral hypotheses. The full-spectrum assumption in part (1) guarantees that all generalized-eigenspace functors required by $\mathfrak  U^{\imath}_{+}$  are present. The reverse construction only asserts that the reconstructed dot has half-integral spectrum; it need not realize every half-integer. Accordingly, Theorem~\ref{main111} gives two complementary constructions and does not claim that they are mutually inverse for arbitrary $2$-representations.

The principal representation-theoretic consequence is obtained from cyclotomic quotients. Let $\mathbf u$ be a finite multiset of parameters in $\frac12+\mathbb Z$, let $\mathcal{CB}_{\mathbf u}$ be the corresponding $\mathbf u$-admissible cyclotomic Brauer category, and let $B$ be its locally unital path algebra. The cyclotomic parameters determine a dominant $\imath$-weight $\kappa$ and a cyclotomic quotient $\mathcal Q(\kappa)$ of the principal $2$-representation of $\mathfrak U^{\imath}_{+}$. Let $Q$ denote the locally unital endomorphism algebra of its distinguished additive generators. Our second main result is the following.

\begin{THEOREM}\label{iso-local-init} There is an isomorphism
$B\cong  Q$
of locally unital $\mathbb C$-algebras.\end{THEOREM} 

This is Theorem~\ref{cyc-iso}. The  isomorphism  is obtained by applying the two constructions  of Theorem~\ref{main111} to the relevant  cyclotomic module categories and comparing the resulting actions on their canonical projective generators. It identifies the generalized-eigenspace idempotents and the dot, crossing, cup, and cap generators with the corresponding $\imath$-Kac--Moody diagrams. 

Theorem~B should also be compared with the generalized cyclotomic
quotient isomorphism of Brundan--Savage--Webster
\cite[Theorem~5.19]{BSW-heisenberg}. To the best of our knowledge,
Theorem~B is the first isomorphism relating the locally unital algebra
of a cyclotomic Brauer category to the endomorphism algebra arising
from a cyclotomic quotient of the principal $2$-representation of an
$\imath$-Kac--Moody $2$-category.

Since the cyclotomic quotient of $\mathfrak  U^{\imath}_{+}$  is graded, Theorem~B transports its grading to the cyclotomic Brauer category. Passing from the locally unital algebra isomorphism in Theorem~B
to individual endomorphism algebras yields a Brauer--coideal analogue
of the realization of cyclotomic Hecke algebras by cyclotomic quiver
Hecke algebras. 

\begin{THEOREM}\label{thmc}Let $\mathbf u=(u_1, u_2, \ldots, u_a)$ be a finite multiset contained in $\frac12+\mathbb Z$, and set $m(z)=\prod_{i=1}^a (z-u_i)$. For $n\ge 0$, 
let $B_{m,n}$ be the associated  $\mathbf u$-admissible cyclotomic Brauer algebra over $\mathbb C$. Then there is an isomorphism of
$\mathbb C$-algebras
\begin{equation} \label{iso-nwalg}   B_{m,n}
   \cong
   \operatorname{End}_{ Q}(X_n),
   \qquad
   X_n=\bigoplus_{\mathbf i\in I^n}E_{\mathbf i}P',
\end{equation} where $P'=1_{\varnothing}Q$ is the distinguished finitely generated
projective right $Q$-module. For each fixed $n$, only finitely many summands in the definition of $X_n$ are non-zero. Consequently, $B_{m,n}$ admits a  $\mathbb Z$-grading   induced via
Theorem B.
\end{THEOREM}

Gradings on certain level-two cyclotomic Nazarov--Wenzl algebras
were obtained by Ehrig and Stroppel via their realization as
endomorphism algebras in a graded parabolic category
$\mathcal O$~\cite{ES16, ES18}. A combinatorially defined grading on
ordinary Brauer algebras was constructed in~\cite{Li}.
Theorem~C treats $\mathbf u$-admissible half-integral cyclotomic
Brauer algebras at arbitrary level, and the grading obtained here
is induced, through Theorem~B, from the $\imath$-Kac--Moody
$2$-categorical structure on $\mathcal Q(\kappa)$.

The associated weakly triangular representation theory and its
decategorification were developed in~\cite{GRS-tri}. Theorems~B
and~C provide the corresponding categorical identification and
grading arising from an $\imath$-Kac--Moody $2$-category.
They
therefore supply a graded categorical foundation for future
comparisons with $\imath$-canonical bases, without asserting such a
comparison or a decomposition-number formula in the present paper.

Table~\ref{tab:type-A-BCD-comparison}  summarizes the  comparison between the type $A$ paradigm and the Brauer-coideal setting of the present paper.  

\begin{table}[!htbp]
  \centering
  \small
  \setlength{\tabcolsep}{5pt}
  \renewcommand{\arraystretch}{1.18}
  \caption{The type $A$ paradigm and the Brauer-coideal setting of  this paper.}
  \label{tab:type-A-BCD-comparison}
  \begin{tabularx}{\textwidth}{@{}>{\bfseries}p{0.21\textwidth}YY@{}}
    \toprule
    & Type $A$ paradigm
    & Brauer-coideal setting of this paper  \\
    \midrule
    Lie-theoretic realization
      & Type A Schur-Weyl duality and related category $\mathcal O$ realizations
      & Schur-Weyl duality involving parabolic category $\mathcal O$ for  Lie algebras of types $B$, $C$, and $D$ \\
    Algebraic  side
      & Degenerate affine Hecke algebras and their cyclotomic quotients 
      & Affine Brauer category and its cyclotomic quotients \\
    Categorified side
      & Kac--Moody/KLR 2-categories
      & The type AIII $\imath$-Kac--Moody 2-category associated with  a quantum symmetric pair \\
    Spectral operators
      &  Jucys--Murphy or tensor Casimir operators and their generalized eigenspaces
      & The affine Brauer dot, realized by a tensor Casimir operator in Lie--theoretic applications, and its half-integral generalized eigenspaces\\
    Verification problem
      & Construct and verify  the KLR  generators and relations  on generalized eigenspaces
      & Begin with the more direct affine Brauer relations; Theorem~A then
constructs and verifies all generators and relations of the
$\imath$-Kac--Moody $2$-action.
\\
    Cyclotomic comparison
      & Brundan--Kleshchev-type isomorphism
      & The locally unital  algebra isomorphism of Theorem~B and its fixed-rank consequence in Theorem~C.  \\
    \bottomrule
  \end{tabularx}
\end{table}

The comparison in the table is structural rather than literal: the
$\imath$-Kac--Moody $2$-category appearing here is of type $\operatorname{AIII}$,
whereas the Lie-theoretic realizations of the affine Brauer action are
related to parabolic category $\mathcal O$ in types $B$, $C$, and $D$.
A further feature of our construction is that the defining identities
must be verified on arbitrary finitely generated objects.

The separation into half-integral and integral cases is also structural. In the present half-integral case the spectrum is $\frac12+\mathbb Z$, and the local boundary behavior is concentrated at the color $\frac12$. In the integral case the involution has the fixed spectral point $0$, producing a different local model and additional boundary relations \cite{GRSW}. The integral construction is therefore not a routine specialization or repetition of the present argument. Treating the two cases separately makes it possible to isolate their distinct local categorification mechanisms.

We briefly describe the proof of Theorem~\ref{main111}.  Starting from an affine Brauer action, we first decompose the generating functor into generalized eigenspaces. The dot, crossings, cups, and caps of the affine Brauer category are then combined with rational functions of the dot to define the generating $2$-morphisms of $\mathfrak U^{\imath}_{+}$. The ordinary quiver-Hecke relations account for only part of the verification. The three additional crossing orientations and the two adjunction structures require separate formulas, and the inhomogeneous boundary relation at $\frac12$ requires a direct calculation. Conversely, a nilpotent $\mathfrak  U^{\imath}_{+}$-action allows the sum of the positive and negative generating functors to be formed object-wise. Suitable combinations of the $\imath$-Kac--Moody $2$-morphisms then satisfy the defining affine Brauer relations. Nilpotence makes these sums object-wise finite and permits the
required power-series evaluations; the exceptional boundary
expressions are treated separately.

For Theorem~B, we apply the forward construction to locally finite-dimensional modules over the cyclotomic Brauer path algebra. The cyclotomic relations determine a highest $\imath$-weight $\kappa$ and force the action to factor through $\mathcal Q(\kappa)$, yielding a homomorphism $Q\to B$. Applying the reverse construction to finitely generated projective $Q$-modules yields a homomorphism $B\to Q$. By construction, the two homomorphisms 
are inverse on generators and hence are mutually inverse. 
 This comparison is also what identifies the grading on the Brauer side.

The paper is organized as follows. Section 2 recalls the affine Brauer categories and fixes the diagrammatic conventions. Section 3 recalls the half-integral type $\operatorname{AIII}$ $\imath$-Kac--Moody $2$-category and the notions of nilpotent and generalized nilpotent $2$-representations. Section 4 develops the generalized-eigenspace and weight decompositions used in both constructions. Section 5 constructs the generating $2$-morphisms of $\mathfrak  U^{\imath}_{+}$  from an affine Brauer action and verifies the homogeneous relations. Section 6 proves the inhomogeneous boundary relation. Section 7 gives the reverse construction of an affine Brauer action from a nilpotent $\imath$-Kac--Moody $2$-representation. Section 8 defines the cyclotomic quotient, proves local finite-dimensionality, establishes Theorem~B, and derives the grading on cyclotomic Brauer algebras. Section~9  records the rational-function identities used in the boundary calculations.

\textbf{Declaration on the use of artificial intelligence.}
The mathematical results, constructions, and proofs in this paper
were developed by the authors. ChatGPT was used for English-language
editing, improvements to the presentation.
Certain rational-function identities were 
verified  by the authors using exact symbolic computations
in Maple. We
take full responsibility for the entire content of the manuscript.

\section{The affine Brauer category} 
In this section, we recall the definition of the affine Brauer category from
 \cite[Definition~1.2]{RS-cyc} and record several  formulas that will be used  later. Throughout, $\Bbbk$ is an algebraically closed field with characteristic different from $2$. 

The affine Brauer category is the degenerate affine diagrammatic category associated with Brauer centralizer algebras. Its crossing, cup, and cap encode the unoriented Brauer calculus, while the dot records the affine Jucys–Murphy-type operator. Polynomial relations imposed on the dot produce cyclotomic quotients and hence the cyclotomic Brauer (or Nazarov–Wenzl) algebras occurring in higher Schur--Weyl dualities involving parabolic categories $\mathcal O$ for Lie algebras of types $B,C$ and $D$. This combination of affine spectral data with unoriented duality morphisms is precisely what makes the category a natural candidate for an action of a quantum-symmetric-pair 2-category.

 \tikzset{strand/.style={-,thick}}
\begin{Defn} \label{C defn}
Let  $\mathcal {AB}'$ be the $\Bbbk$-linear strict monoidal category generated by a single object, denoted diagrammatically by a vertical strand
$\
$.
The objects of $\mathcal {AB}'$ are identified with $\mathbb N$. Under this identification,  $n\in\mathbb N$ corresponds to $
%
^{\otimes n}$, and $0$ corresponds to the unit object.
The generating morphisms are $%
	: 1 \to 1 $.
These morphisms   are subject to the following defining relations:
%given in  %	subject to the defining relations in 
%\eqref{B relations 1 (symmetric group)}--\eqref{B relations 2 (zigzags and invertibility2)}:
\begin{multicols}{2}
    \item[(1)] 
		\label{B relations 2 (zigzags and invertibility)}
		$%
$,
	%\end{equation}
	
  \item[(5)]  % \begin{equation}
		\label{B relations 2 (zigzags and invertibility2)}
		$%
$.
	%\end{equation}
    \end{multicols}
\end{Defn} 

\begin{Lemma}\label{AB relations-re} In $\mathcal {AB}'$, the following relations hold:  \begin{multicols}{2}
     \item [(1)]
 $ %
$.
                   \end{multicols} \end{Lemma}

\begin{proof}
 The first two assertions  follow from  Definition~\ref{C defn}(1) and (5). For (3),  we compute  as follows: 
 $$\begin{aligned} 
 %
  \quad
                   \text{by Definition~\ref{C defn}(1),(4).} \end{aligned} $$
 %Here the first equality follows from Definition~\ref{C defn}(1), and the second equality follows from by (1),(2) and %Definition~\ref{C defn}(3), and the last one follows from  Definition~\ref{C defn}(1),(4). 
 \end{proof}
Keeping the relations in  Definition~\ref{C defn}(1)--(4) unchanged, one may replace Definition~\ref{C defn}(5) by Lemma~\ref{AB relations-re}(1) and (2). This  gives an equivalent presentation of the category, which will be used in  Section~\ref{ikmcmc}.

\begin{Defn} \label{dot-n}   For all integers $n\ge 1$, let 
$
    is defined to be    the identity morphism. 
\end{Defn}

Later, we also write   $%
.
This convention  allows  us to    
represent a $\Bbbk$-linear combination of monomials by  labeling dots with   polynomials in 
$x$.

\begin{Cor}\label{caup1}  In $\AB' $, the following relations hold:%原来点的位置画错了 
\begin{multicols}{2}
\item [(1)]
$%
	\frac{x^r}{u^r}
$.
\end{Cor} 
\begin{proof} A direct check shows
 that  each  identity in~(1) and (2) is equivalent to the corresponding identity  obtained by replacing
every occurrence of $\ominus$ with $\bullet$. Thus, it suffices to prove
the latter version. Using
Definition~\ref{C defn}(2), (4) and Lemma~\ref{AB relations-re}(3),
we compute \begin{equation}\label{squa-red} 
%
. 
\end{equation} 
 This proves (1). The proof of~(2)
is entirely similar,
so we omit the details. \end{proof}

 \tikzset{strand/.style={-,thick}}
\begin{Defn} \label{AB defn}\cite[Definition~1.2]{RS-cyc}
The \emph{affine Brauer category} $\AB$ is the $\Bbbk$-linear
strict monoidal category obtained as the quotient of   $\mathcal{AB}'$
 by  the tensor ideal generated by
the following relations:
\begin{multicols}{2}
    \item[(1)]\label{B relations 1 (symmetric group)}
    $
	\tikzset{strand/.style={-,thick}}
	%
	$.
\end{multicols}
\end{Defn}

 We remark that the  affine Brauer category $\AB$  considered here is isomorphic to the category defined in \cite[Definition~1.2]{RS-cyc}. The required isomorphism sends the generating morphism   %
  $ and  fixes all other generating morphisms. We adopt the present normalization because  it is compatible with the $\imath$-Kac--Moody $2$--category $\mathfrak U^\imath$ defined in \cite[Definition~3.3]{BSWW-icate}.

\begin{Lemma}\label{anti} 
 There is a $\Bbbk$-linear monoidal isomorphism $\sigma: \AB\rightarrow \AB^{\mathrm{op}}$ that fixes  objects and the  two  generating morphisms $   %
 , 
 \swap$, and swaps generating morphisms  $\lcap$ and $\lcup$  of $\AB$.
 \end{Lemma} 
\begin{proof} Since $\AB$ is a quotient category of $\mathcal {AB}'$, we obtain the relation in Lemma~\ref{AB relations-re}(2).
Consequently,  all defining relations of $\AB$ in Definition~\ref{C defn}(1)--(5), together with  Definition~\ref{AB defn}(1)--(2),
are symmetric with respect to the horizontal axis, and the result follows.
\end{proof}

\begin{Cor} \label{slidecro}If the generating morphisms in $\mathcal{AB}'$ satisfy the relation in Definition~\ref{AB defn}(1), then the following relations hold: 
\begin{multicols}{2}
\item [(1)] $
$.
    \end{multicols}
    \end{Cor}

\begin{proof}
   A direct check shows   
 that, for each identity in~(1)--(3),
the stated identity is equivalent to the corresponding identity obtained by replacing
every occurrence of $\ominus$ with $\bullet$. We omit the details since the identities  can be verified directly using Definition~\ref{C defn}(4), Lemma~\ref{AB relations-re}(3) and Definition~\ref{AB defn}(1). \end{proof}

In $\AB$, all relations in Corollaries~\ref{caup1} and~\ref{slidecro} follow immediately. In Section~\ref{ikmcmc}, however, we will need these corollaries under weaker hypotheses: Corollary~\ref{caup1} will be used without the relations in Definition~\ref{AB defn}(1)--(2), and Corollary~\ref{slidecro} without the relation in Definition~\ref{AB defn}(2).

\begin{Cor}\label{caup}  In $\AB$, the following relations hold: 
\begin{multicols}{2}
\item[(1)]  
	$
	\frac{(-x)^r}{u^r}$. \end{Cor} 
\begin{proof} The two relations in  (1) follow immediately from Lemma~\ref{AB relations-re}(1) and (2). 
	Relation (2) (respectively, (3)) follows  from Lemma~\ref{AB relations-re}(3) (respectively, 
Definition~\ref{C defn}(4)) by pre- and post-composing with the appropriate morphisms labeled by  $(u-x)^{-1}$.  Relation  (4) follows from  Definition~\ref{C defn}(1)--(3), and   (5) is immediate.
  \end{proof}

For any commutative $\Bbbk$-algebra $R$, let  $R((u^{-1}))$ denote  the ring of formal Laurent series  in $u^{-1}$. 
 Thus, each $f(u)\in R((u^{-1}))$ has the form 
$\sum_{i\in \mathbb Z} a_i u^i$ such that $a_i=0$ for $i\gg 0$. 
Throughout the paper, we use the following notation:
\begin{equation} \label{ou1}\clock(u) = u-\frac{1}{2}+\sum_{r\ge 0}
u^{-r} \in  \End_{\AB}(0)((u^{-1})), \end{equation}

%Throughout this paper, we assume that
\begin{Lemma} \cite[Lemmas~4.1--4.2]{GRS-block}  \label{ou}  % 
The following relations hold, the first in   $\End_{\AB}(0)((u^{-1}))$ and the second in  $\End_{\AB}(1)((u^{-1}))$:
\begin{multicols}{2} 
\item[(1)]  $
    \clock (u)\clock (-u)=(\frac{1}{2}-u)(\frac{1}{2}+u)$,
    \item[(2)] $\mathord{
%
}$.\end{multicols}
\end{Lemma} 

\usetikzlibrary{decorations.pathreplacing} % 画 brace 需要这个库

We use the following notation  to simplify the presentation below. It is inspired by the corresponding convention
for the Heisenberg category in \cite{BSW-heisenberg}.
\begin{Defn} Suppose  $f(x_1,x_2)=\sum_{i=0}^m\sum_{j=0}^na_{i,j}x_1^ix_2^j\in R[x_1, x_2]$, and $g(u)=\sum_{i\in \mathbb Z} a_iu^i\in R((u^{-1}))$. Define 
\begin{itemize}
\item [(1)] $%
$.
    \end{itemize} 
Here the variables are assigned to the dots according to the lexicographic order of the
Cartesian coordinates of the dots in the diagram.
\end{Defn}
 %for any $p(u)\in R[u]$, where $R$ is any $\Bbbk$-algebra.
\begin{Lemma}\label{usefuel-equa} 
For any $f(u)\in  \End_{\AB}(0)[u]$,  set   $\Delta_f(x_1, x_2)=\frac{f(x_1)-f(x_2)}{x_1-x_2}\in   \End_{\AB}(0)[x_1, x_2]$. Then the following identities hold:  
\begin{multicols}{2} \item[(1)] $\mathord{
%
}$.
\end{multicols}
\end{Lemma}
\begin{proof}  Identity  (1) coincides with \cite[(3.19)]{BSW-heisenberg}. Identity (2) is obtained from (1) by replacing \(u\) with \(-u\).
Applying Corollary~\ref{caup}(2), we obtain
\begin{equation}\label{kk1} \begin{aligned}
    	%
.
    \end{aligned}\end{equation}
 Then (3) follows immediately from \eqref{ou1}, \eqref{kk1}, and Corollary~\ref{caup}(1),(5).
The proof of (4) is identical to that of (3), using Corollary~\ref{caup}(3) and \eqref{ou1}.  We have
$$\Big[
\mathord{
%
}%\Big]_{u^{-1}}
.$$
 This establishes (5). Similarly, (6) is proved in the same way as (5), and we  omit the details.  To prove 
(7), we compute 
$$\mathord{
%
},$$
and we obtain the desired identity. The proof of  (8) is similar.   
\end{proof}

In Lie-theoretic applications, these affine Brauer relations, with
the dot supplied by a tensor Casimir operator and the remaining
generators arising from the symmetry and duality of tensor products,
provide directly verifiable input from which Sections~5--6 construct
the full $\imath$-Kac--Moody $2$-action.

\section{The $\imath$-Kac--Moody $2$-category}
In this section, we recall  the ungraded version of  the coideal $2$--category $ \mathfrak  U^\imath$ from \cite[Definitions~3.1 and 3.3]{BSWW-icate}, called the $\imath$-Kac--Moody $2$-category. 
From now on, unless otherwise stated, we assume that  \begin{equation}\label{index-I} I=\frac{1}{2}+\mathbb Z.\end{equation}
Let $\mathfrak {sl}_I$ be the simple Kac--Moody Lie algebra % of type $A$ 
associated with the  Cartan matrix $C=(c_{ij})_{i,j\in I}$  such that \begin{equation}\label{cij} c_{i, j}=\begin{cases}2 & \text{if $i=j$,}\\
-1 & \text{if $|i-j|=1$,}\\
0 & \text{otherwise.}\end{cases}
\end{equation}
The corresponding root datum  is $(X, Y, \Pi, \Pi^\vee)$, where $X$,  $Y$, $\Pi $, and  $\Pi^\vee$  are
respectively the weight lattice, the coroot lattice, the set of simple roots $\{\alpha_i\}_{i \in I}$,
and the set of simple coroots $\{\alpha_i^\vee\}_{i \in I}$. 
There is a  perfect pairing 
 \begin{equation}\label{equa-pairing1}
     \langle\  , \ \rangle: ~Y\times X\quad  \longrightarrow   \mathbb Z \end{equation} such that $\langle \alpha^\vee_i,\alpha_j\rangle= c_{i,j}$.  Consider the lattice $\bigoplus_{r=-\infty}^\infty \mathbb Z \varepsilon_r$ such that $\langle \varepsilon_i, \varepsilon_j\rangle=\delta_{i, j}$. For all $i\in I$,  let $\Lambda_i$ be the fundamental weight. Then
\begin{equation}\label{www1} \varepsilon_{i-\frac 12} =\Lambda_i-\Lambda_{i-1}, \text{ and } \alpha_i=\varepsilon_{i-\frac{1}{2}}-\varepsilon_{i+\frac{1}{2}}.\end{equation}     
     Let $\theta: X\rightarrow X$ be the involution  such that 
 $\theta(\varepsilon_r)=-\varepsilon_{-r}$ and hence  $\theta(\alpha_i)=\alpha_{-i}$, $i\in I$.
 Let $X^\theta$ be the $\theta$-fixed points in $X$. 
 Following \cite{BSWW-icate}, 
define   \begin{equation}\label{isimple}   I^\imath=1/2+\mathbb N, \quad  \thal^\vee_i:=\alpha_i^\vee -\alpha_{-i}^\vee, \quad X_\imath= X/X^\theta, \quad Y^\imath =\bigoplus_{i\in I^\imath} \mathbb Z \ \thal^\vee_i. \end{equation}  
Then \begin{equation}\label{scalar23}\langle \thal^\vee_i, \alpha_i\rangle=
\begin{cases}
    2 & \text{ if } i>\frac{1}{2}\\
    3 & \text{ if } i= \frac{1}{2}
\end{cases}.\end{equation}
The pairing in \eqref{equa-pairing1} descends to a nondegenerate pairing
 $\langle\ ,\ \rangle : ~ Y^\imath\times X_\imath\longrightarrow \mathbb Z$.
%Fix a commutative ring $R$ containing the identity and the unit  $2$. 
For $i,j\in I^\imath$,   set 
\begin{equation}
\label{tij} 
t_{ij}=\begin{cases} 
-1, &\text{if $ j=i-1$,}\\
1, &\text{otherwise.}
\end{cases}
\end{equation}

\begin{Defn}\cite[Definition~3.1]{BSWW-icate}\label{def-ikmc}
  The $2$-category $\mathfrak U$ 
is the strict   additive $\mathbb C$-linear $2$-category whose objects are the elements   $\lambda\in X_\imath$.
The $1$-morphisms are generated by 
\begin{align*}
E_i={\scriptstyle\substack{\lambda+\alpha_i\\\phantom{-}}}\substack{{\color{darkred}{\displaystyle\uparrow}} \\
  {\scriptscriptstyle i}}{\scriptstyle\substack{\lambda\\\phantom{-}}}
:\lambda \rightarrow \lambda+\alpha_i=\lambda-\alpha_{-i},\quad F_i={\scriptstyle\substack{\lambda-\alpha_i
\\\phantom{-}}}\substack{{\color{darkred}{\displaystyle\downarrow}} \\
  {\scriptscriptstyle
    i}}{\scriptstyle\substack{\lambda\\\phantom{-}}}:\lambda \rightarrow \lambda-\alpha_i=\lambda+\alpha_{-i}, 
\qquad
\text{for }i \in I^\imath.
\end{align*}
Here ``generated" means that we allow  finite direct sums of compositions  of these $1$-morphisms.
  The 2-morphisms are generated by 
\begin{alignat*}{3}
&x 
= 
\mathord{

}:  F_i E_i1_\lambda\to 1_\lambda.
\end{alignat*} 
These $2$-morphisms are subject to the following relations (1)--(7): % \eqref{Adj}--\eqref{Mixed}:
\begin{itemize}
    \item [(1)] (Adjunction) \label{adj-red} For any  $i\in I^\imath$, \begin{multicols}{2}
        \item [(a)] $
\mathord{
%
}$.\end{multicols}

\item [(3)] (Quiver Hecke relations) For any $i, j, k\in I^\imath$,  \begin{itemize}\item[(a)]\label{qha-red}
$\mathord{
%
$.\end{multicols}
 In the sums above, 
the indices not attached to a circle run over  non-negative integers.

\item [(7)] (Mixed relations) For any $i, j\in I^\imath$ with $i\neq j$,
\begin{multicols}{2}\item [(a)] 
\label{mixedup}
$%
$.\end{multicols}
\end{itemize}
\end{Defn}
%Although $t_{ij}^{-1}=t_{ij}$ in our setting

The bubbles with negative labels are determined recursively
by Definition~\ref{def-ikmc}(4e) and are called fake bubbles.

\begin{Defn}\cite[Definition~3.3]{BSWW-icate}\label{def-ikmc1}
   The {\em $\imath$-Kac--Moody $2$--category} $ \mathfrak  U^\imath$ is the 
   $2$--category obtained from $\mathfrak U$ by taking the quotient by  the $2$-ideal generated  by the following  inhomogeneous relation:
   %\begin{itemize}
%\item [(1)]  %(The inhomogeneous  relation)
\label{Pi=11}
$$\begin{aligned}
, 
\end{aligned}$$
where all strands are labeled by $\frac{1}{2}$.
%\end{itemize}
\end{Defn}
This is the only inhomogeneous defining relation and is concentrated at the boundary color $\frac12$. It records the deviation of the coideal categorification from an ordinary Kac--Moody 2-category. Consequently, verifying this relation is the principal new obstruction in constructing a $\mathfrak  U^{\imath}$-action from an affine Brauer action; it is treated separately in Section 6.

Directly constructing a 2-functor out of $\mathfrak  U^{\imath}_+$ requires checking the entire list above, including the exceptional
inhomogeneous relation. Theorem A replaces this task by the construction of an affine Brauer action, whose
generators and relations are substantially more direct and which, in Lie-theoretic examples, can be realized
using tensor Casimir operators and standard adjunction morphisms.

\begin{Defn}\label{parity} For  $\lambda\in X_\imath$,  define  its parity to be  $0$ if 
$\sum_{i\in I^\imath} \langle \thal^\vee_i, \lambda\rangle$ is even, and to be $1$   if this sum  is odd.  
Let $X_\imath^+=\{\lambda\in X_\imath\mid  \lambda \text{ has parity $ 0$ } \}$ and 
$X_\imath^-=\{\lambda\in X_\imath\mid  \lambda \text{ has parity $1$}\}$.\end{Defn}

\begin{Lemma}\label{evenodd}
For every generating $1$-morphism 
$E_i, F_i : \lambda \to \mu$ with $i \in I^\imath$, 
the weights $\lambda$ and $\mu$ have the same parity.
\end{Lemma}
\begin{proof}
   It follows from Definition~\ref{def-ikmc} that $\mu$ is either  $\lambda+\alpha_i$ or  $\lambda-\alpha_{i}$.  Since only finitely many terms are nonzero,    
    \eqref{cij} and \eqref{isimple} give  $\sum_{j\in I^\imath} \langle \thal^\vee_j,  \alpha_k\rangle \equiv 0\pmod 2$ for all $k\in I^\imath$. 
    Thus, $\lambda$ and $\mu$ have the same parity. \end{proof}

 By Lemma~3.4, each of the generating \(1\)-morphisms \(E_i\) and \(F_i\)
preserves parity. It follows that every \(1\)-morphism in
\(\mathfrak U^\imath\) (respectively, $\mathfrak U$) preserves parity, and hence sends objects in
\(X_\imath^+\) to objects in \(X_\imath^+\), and objects in
\(X_\imath^-\) to objects in \(X_\imath^-\). Consequently, there are no
nonzero \(1\)-morphisms between an object of \(X_\imath^+\) and an object
of \(X_\imath^-\). Hence \(\mathfrak U^\imath\) and $\mathfrak U$ decompose as
\begin{equation}\label{decomp1} 
\mathfrak U^\imath=\mathfrak U^\imath_+\amalg \mathfrak U^\imath_-, \quad  \mathfrak U=\mathfrak U_+\amalg \mathfrak U_-
\end{equation}
as  disjoint unions of \(2\)-categories, where
\(\mathfrak U^\imath_+\) and \(\mathfrak U^\imath_-\) (respectively, $\mathfrak U_+$ and $\mathfrak U_-$ ) are the full
sub-\(2\)-categories of \(\mathfrak U^\imath\) (respectively, $\mathfrak U$) on the objects
\(X_\imath^+\) and \(X_\imath^-\), respectively.  Let $\omega$ be the unique element of $X_\imath$
such that
\[
\langle \thal^\vee_i, \omega\rangle=
\begin{cases}
1 & i=\frac{1}{2},\\
0 & i\in I^\imath\setminus \left\{\frac{1}{2}\right\}.
\end{cases}
\]

\begin{Prop}\cite[(3.19)]{BSWW-icate}\label{omegaimath}
    There is a $2$-functor $\omega_\imath: \mathfrak U^\imath_+\rightarrow \mathfrak U^\imath_-$   (respectively, $\mathfrak U_+\rightarrow \mathfrak U_-$) defined by  sending  $\lambda$ to $-\lambda-\omega$. On  generating 1-morphisms, it interchanges  $ E_i$ and $F_i$. On the  generating  2-morphisms, it is given  by 
    \[ x\mapsto x', x'\mapsto x, \tau \mapsto -\tau', \tau'\mapsto -\tau, \eta\mapsto 2^{-\delta_{\frac{1}{2},i}}\eta', \eta'\mapsto \eta, \epsilon \mapsto 2^{\delta_{\frac{1}{2},i}}\epsilon', \epsilon'\mapsto \epsilon, \]
    where $x, x'$, and so on  are defined as in Definition~\ref{def-ikmc}.
  Moreover, the $2$-functor $\omega_\imath$ is an isomorphism of 2-categories. 
\end{Prop}
\begin{proof}
The   $2$-functor $\omega$ on $\mathfrak U^\imath$,  defined in
\cite[(3.19)]{BSWW-icate}, interchanges the subcategories
$\mathfrak U^\imath_+$ and $\mathfrak U^\imath_-$. Hence its restriction to
$\mathfrak U^\imath_+$ gives the desired $2$-functor
$ \omega_\imath:\mathfrak U^\imath_+\longrightarrow \mathfrak U^\imath_-$.
The same formulas on the generators also define a $2$-functor $\mathfrak U_+\longrightarrow \mathfrak U_-$, since they satisfy the defining relations on $\mathfrak U$. Moreover, the arguments in \cite[(3.19)]{BSWW-icate} show that both $2$-functors are isomorphisms. \end{proof}

Recall that a strict $2$-representation of a strict $\Bbbk$-linear $2$-category $\mathcal U$ is  a \(2\)-functor
\[
\Phi:\mathcal U \longrightarrow \operatorname{Cat}_\Bbbk
\]
where $ \operatorname{Cat}_\Bbbk$  is the strict $\Bbbk$-linear $2$-category of $\Bbbk$-linear categories, $\Bbbk$-linear functors, and $\Bbbk$-linear natural transformations. 
Thus, if $X$ denotes the set of objects of $\mathcal U$, then for each object $\lambda\in X$,   there is a $\Bbbk$-linear category $\mathcal R_\lambda$, called the $\lambda$-weight category of $\mathcal R$.  Following \cite{BSW-heisenberg}, we say that  $(\mathcal R_\lambda)_{\lambda\in X} $ is a locally finite Abelian, respectively  locally Schurian, 
$2$-representation if each  $\mathcal R_\lambda$  is a locally finite Abelian,  respectively locally  Schurian.
Unless otherwise  stated, every $2$-representation considered below is assumed to be either locally finite Abelian or locally Schurian.
Let $\mathcal M=(\mathcal M_\lambda)_{\lambda\in X}$ be a
$2$-representation. By a $2$-subrepresentation
$\mathcal N=(\mathcal N_\lambda)_{\lambda\in X}$ of $\mathcal M$,
we mean a family of full replete additive $\Bbbk$-linear
subcategories
\[
\mathcal N_\lambda\subseteq\mathcal M_\lambda
\]
that is stable under all generating $1$-morphisms. The generating
$2$-morphisms then restrict naturally to $\mathcal N$.
Unless explicitly stated, a $2$-subrepresentation is not required
to be locally finite Abelian or locally Schurian. 

Following  \cite{BSW-heisenberg}, {we} define 
\begin{equation}
\label{affm} \mathcal R=\begin{cases}  \bigoplus_{\lambda\in X } \mathcal R_\lambda, & \text{in the locally finite Abelian case} \\
\prod_{\lambda\in X} \mathcal R_\lambda & \text{in the locally Schurian case.}\end{cases} \end{equation} 
Then $\mathcal R$ is also  a locally finite Abelian category if every $\mathcal R_\lambda$ is a locally finite Abelian category, and it is a locally Schurian category if every  $\mathcal R_\lambda $ is a locally Schurian category. 

We now specialize to the $\imath$-Kac--Moody $2$--category $\mathfrak U^\imath$. By \eqref{decomp1}, a $2$--functor defined on $\mathfrak U^\imath$ (respectively, $\mathfrak U$) is  equivalent to a pair of $2$-functors defined on $\mathfrak U^\imath_+$ and $\mathfrak U^\imath_-$ (respectively, $\mathfrak U_+$ and $\mathfrak U_-$). In particular, giving a $2$-representation of $\mathfrak U^\imath$ (respectively, $\mathfrak U$) is equivalent to giving a pair of $2$--representations, one on each component. By Proposition~\ref{omegaimath}, the $2$--functor $\omega_\imath$ provides an isomorphism between $\mathfrak U^\imath_+$ and $\mathfrak U^\imath_-$ (respectively, $\mathfrak U_+$ and $\mathfrak U_-$). Hence, the representation theories of $\mathfrak U^\imath_+$ and $\mathfrak U^\imath_-$ (respectively, $\mathfrak U_+$ and $\mathfrak U_-$) are equivalent. Consequently, any statement proved for one component has  an analogous statement for the other component obtained by transport under $\omega_\imath$. However, a general \(2\)-representation of
\(\mathfrak U^\imath\) (respectively, $\mathfrak U$) is not determined by its restriction to a single
component. Below,  we focus on \(\mathfrak U^\imath_+\) and $\mathfrak U_+$. The corresponding
results for \(\mathfrak U_-^\imath\) and $\mathfrak U_-$ are obtained by transporting the
statements through $\omega_\imath$. 

 In this paper, we focus on two special classes of $2$-representations of $\mathfrak U^\imath_+$, called  nilpotent and generalized nilpotent 2-representations.

\begin{Defn} \label{nil} \cite[\S5.1, p.42]{BSW-heisenberg}  A $2$-representation \(\mathcal M=(\mathcal{M}_\lambda)_{\lambda \in X_\imath^+}\) of $\mathfrak U^\imath_+$ (or $\mathfrak U_+$) 
is nilpotent if the following conditions hold: \begin{itemize}
    \item [(1)]  For each \(\lambda \in X_\imath^+\) and each
    \(V \in \mathcal{M}_\lambda\), 
   \(E_i V = F_i V = 0\) for all but finitely many \(i \in I^\imath\).
    \item [(2)] For all \(i \in I^\imath\), \(\lambda \in  X_\imath^+\), and every finitely generated \(V \in \mathcal{M}_\lambda\), the endomorphism \(\begin{tikzpicture}[baseline = -1mm,darkred, scale=0.7, transform shape]
	\draw[->,thick] (0.08,-.4) to (0.08,.3);
      \node at (.08,0) {$\bullet$};
       \node at (.08,-.6) {$\scriptstyle{i}$};
      \node at (0.4,-.3) {$\color{black}\scriptstyle{\lambda}$};
\end{tikzpicture}
\begin{tikzpicture}[baseline = 0]
  \draw[-,darkg,thick] (0.55,-0.3) to (0.55,.3);
   \node at (0.55,-.4) {$\darkg\scriptstyle{V}$};
\end{tikzpicture} : E_i V \to E_i V\) is nilpotent (equivalently, \(\begin{tikzpicture}[baseline = -1mm,darkred, scale=0.7, transform shape]
	\draw[<-,thick] (0.08,-.4) to (0.08,.3);
      \node at (.08,0) {$\bullet$};
       \node at (.08,-.6) {$\scriptstyle{i}$};
      \node at (0.4,-.3) {$\color{black}\scriptstyle{\lambda}$};
\end{tikzpicture}
\begin{tikzpicture}[baseline = 0]
   \draw[-,darkg,thick] (0.55,-0.3) to (0.55,.3);
   \node at (0.55,-.4) {$\darkg\scriptstyle{V}$};
\end{tikzpicture}\) is nilpotent).
\end{itemize}
 \end{Defn} 
We use the same terminology for a $2$-subrepresentation when
conditions~(1) and~(2) hold for every object in its weight
subcategories.
In this paper, we call $\mathcal M$ generalized nilpotent if condition (2) holds and condition (1) is required
only for finitely generated objects  $V\in \mathcal M_\lambda$, for every $\lambda$ in $X^\imath_+$.
 It follows from the definition that a nilpotent $2$-representation is a generalized nilpotent $2$-representation. The converse is false in general. However, when each $\mathcal M_\lambda$ is locally finite Abelian, the converse is true. 
{Indeed,  every object has finite length, and hence is finitely generated,  so condition (1) in the definition of generalized nilpotence applies to every object}.

\section{Eigenspaces and the exact-spectrum hypothesis}\label{Brauer-cate} 
In this section, we record several  facts  about  module categories  for  the  affine Brauer category over an algebraically closed field $\Bbbk$ of  characteristic different from $2$.

Suppose that the $\Bbbk$-linear locally Schurian  category $\mathcal M$ admits an affine Brauer categorification. 
When  relating  the  $\imath$-Kac--Moody $2$--category $\mathfrak U^\imath$ to the affine Brauer category, we further  assume  that $\Bbbk=\mathbb C$.

\subsection{Subfunctors and natural transformations}  In this subsection,  we prove several identities  for natural transformations that will be needed later. 

We work with strict monoidal categories. If $\mathcal A$ is a strict $\Bbbk$-linear  monoidal category,  then a left $\mathcal A$-module category is a $k$-linear
category $\mathcal M$  together with a strict $\Bbbk$-linear monoidal functor 
$M : A\rightarrow  \End_{\Bbbk}(\mathcal M)$.
Here $\End_{\Bbbk }(\mathcal M)$  denotes the strict monoidal category of $\Bbbk$-linear endofunctors of $\mathcal M$  and natural transformations.
If A is the affine Brauer category, then we say that M admits an affine Brauer categorification. Let \begin{equation}\label{ee}  E=
\mathrm M(\begin{tikzpicture}[baseline = 10pt, scale=0.5, color=\clr] \draw[-,thick] (0,0.5)to[out=up,in=down](0,1.2);
    \end{tikzpicture}).\end{equation} where $\mathrm{M}$ is the functor defined  for the affine Brauer category $\AB$. When drawing natural transformations, we omit $\mathrm{M}$ from the notation and simply write
\begin{tikzpicture}[baseline = 10pt, scale=0.6]
  \draw[-,thick] (0,0.5) to[out=up, in=down] (0,1.2);
\end{tikzpicture}
for $E$.  In other words, we use the same graphical notation for morphisms in $\AB$ and their images under $\mathrm{M}$.
 Furthermore, for any morphism $g$ in $ \mathcal{AB}$ and any object $V$ in $ \mathcal M$, the morphism 
$g_V $ is drawn as \begin{equation}\label{nat-gv}  \begin{tikzpicture}[baseline = 9pt,scale=0.4,inner sep=0pt, minimum width=11pt]
        \draw[-,thick] (0,0) to (0,2);
        \draw (0,1) node[circle,draw,thick,fill=white]{$g$};
        \draw[-,darkg,thick] (1,0) to (1,2);
      \node at (1,-.5) {$\darkg\scriptstyle{V}$};
    \end{tikzpicture}.\end{equation}
   To simplify notation,  we  omit the factor  $
\begin{tikzpicture}[baseline = 0, scale=0.5]
   \draw[-,darkg,thick] (0.65,-0.4) to (0.65,.4);
    \node at (0.65,-.7) {$\darkg\scriptstyle{V}$};
 \end{tikzpicture}
$  and write $
\begin{tikzpicture}[baseline = 9pt,scale=0.4,inner sep=0pt, minimum width=11pt]
\draw[-,thick] (0,0) to (0,2);
\draw (0,1) node[circle,draw,thick,fill=white]{$g$};
\end{tikzpicture}$ 
for
$\begin{tikzpicture}[baseline = 9pt,scale=0.4,inner sep=0pt, minimum width=11pt]
\draw[-,thick] (0,0) to (0,2);
\draw (0,1) node[circle,draw,thick,fill=white]{$g$};
\draw[-,darkg,thick] (1,0) to (1,2);
\node at (1,-.5) {$\darkg\scriptstyle{V}$};
\end{tikzpicture}$; equivalently, we say  that 
   $g$  acts on   $V\in \mathcal  M$. In what follows,  the rightmost factor corresponding to $V$ will be suppressed in the diagrams.

\begin{Defn}\cite[p.7]{BD-cate}  An endofunctor $
E$  of a locally Schurian category is said to be sweet if there exists an endofunctor 
$
F$  that is biadjoint to 
$E$.\end{Defn}

 \begin{Lemma}\label{biadj}   
  Let $E$ be the endofunctor of $\mathcal M$ defined in \eqref{ee} and let $V$ be a finitely generated object of $\mathcal M$.   Then \begin{itemize} \item[(1)]  $E$ is a sweet functor. In particular, $(E,E)$ is a  biadjoint pair. 
  \item [(2)]   $\dim \End_{\mathcal M}(EV)<\infty$.\end{itemize}  \end{Lemma}
 \begin{proof} It follows from Definition~\ref{C defn}(1)
 that the cup and cap give a biadjunction between $E$ and itself; hence $E$ is a sweet functor. It is proved in \cite[p.7]{BD-cate} that %\old{$\End_{\mathcal M}(L)\cong \k$ for any simple object $L$ in $\mathcal M$, and} 
$\dim \Hom_{\mathcal M}(V, W)< \infty$ for any finitely generated objects $V, W$ in $\mathcal M$. Since $E$ is a sweet functor and $V$ is finitely generated, \cite[Theorem~2.11]{BD-cate} implies that  $EV$ is also finitely generated and hence  $\dim \End_{\mathcal M}(EV)<\infty$.\end{proof}

\begin{rem}\label{nateq}
\(\mathcal M\) is a \(\Bbbk\)-linear locally Schurian
category;  equivalently,  \(\mathcal M\) is equivalent to the
category of locally finite-dimensional left modules over a locally unital, locally finite-dimensional $\Bbbk$-algebra. It follows from
\cite[Lemma~2.6]{BS-semi} that every object of \(\mathcal M\) is a filtered
colimit of finitely generated objects. On the other hand,
\cite[Theorem~2.11]{BD-cate} shows that a sweet functor is both continuous
and cocontinuous; in particular, it preserves the filtered colimits used below.
Consequently, if \(F,G:\mathcal M\to\mathcal M\) are sweet functors, then
two natural transformations \(\alpha,\beta:F\Rightarrow G\) are equal if
and only if \(\alpha_V=\beta_V\) for every finitely generated object
\(V\in\mathcal M\). We will use this criterion freely in what follows
without further mention.
\end{rem}

\begin{Defn}\label{subf}
    For any \(i\in\Bbbk\), define an endofunctor  \(E_i\) as follows.
     If \(V\) is finitely generated, \(E_iV\) is the generalized \(i\)-eigenspace of \(EV\) with respect to the dot endomorphism 
        $
        \begin{tikzpicture}[baseline=-0.5ex,scale=1]
            \draw[thick] (0,-0.3) -- (0,0.3);
            \fill (0,0) circle (0.06);
        \end{tikzpicture} : EV \to EV$.
     For arbitrary \(V\), write \(V = \varinjlim_a V_a\) as a filtered colimit of finitely generated subobjects, and set
        $
        E_iV := \varinjlim_a E_iV_a$.
        \end{Defn}

The definition is independent of the chosen presentation,  since any filtered  system of  finitely generated subobjects whose 
colimit is 
 \(V\) is cofinal in the system of all finitely generated subobjects of $V$.  The assignment \(V \mapsto E_iV\) is functorial because for any morphism \(f:V\to W\) and any finitely generated \(V_a\subseteq V\), the image \(f(V_a)\) is a finitely generated subobject of \(W\), and naturality of the dot action guarantees that \(E_iV_a\) maps to \(E_i(f(V_a))\subseteq E_iW\). Passing to the colimit gives a 
 well-defined morphism \(E_iV \to E_iW\). 

Since \(E\) preserves filtered colimits, we have $\varinjlim_a EV_a\cong EV$. Moreover,  filtered colimits are  exact in module categories. Therefore, the inclusions \(E_iV_a\hookrightarrow E V_a\) induce a monomorphism \(E_iV \hookrightarrow EV\). Hence \(E_i\) is indeed a subfunctor of \(E\).

% Note that $E_i=0$ unless $i\in J$.  From this point to the end of the paper, we always deal with non-zero endofunctors %$E_i$, namely, $i\in J$.  
 
 In string calculus, the identity morphism of an object is drawn by using the same diagram as the object. For example, the identity morphism of $E$
is drawn as $ $. 
%Motivated by \cite{BSW-heisenberg}, 
Therefore, we  draw the functor  $E_i$ and its identity morphism by the vertical string colored by $i$. 
This leads to the following definition.

\begin{Defn}\label{subf0} For any $i\in \Bbbk$, define %check the following result 
\begin{multicols}{2}
    \item [(1)]  $%
\!:E_i \Rightarrow E_i$, the identity morphism of $E_i$,

\item [(2)]  $%
}:E \Rightarrow E$, the projection of $E$ onto its summand $E_i$. \end{multicols}\end{Defn}

 \begin{Lemma}\label{ortho-dotcom} For any  $i\in \Bbbk$, let   %
 be the restriction of %
 to %the summand 
$E_i$. 
 As natural transformations, the following relations hold  for all  $i, j\in \Bbbk$: \begin{multicols}{2} \item [(1)] 
$ %
$.
\end{multicols}  \end{Lemma}
\begin{proof} 
By Remark~\ref{nateq}, it is enough to verify these  relations after evaluating them   on any finitely generated object.  For a finitely generated object $V$,  only finitely  many terms on the right-hand side  of (1) are nonzero. Thus the sum in (1) is 
well-defined.  
So, (1) and (2) follow immediately from  Definition~\ref{subf0}.
The naturality of the dot morphism $%
 $ implies that the transition maps \(E(V_a)\to E(V_b)\)
preserve generalized \(i\)-eigenspaces. Hence the dot commutes with the corresponding projection onto $E_i$,  
and (3) and (4) follow.  \end{proof}

\begin{Lemma} \label{cap-cup-move}   Suppose that $i, j\in \Bbbk$ with  $i\neq -j$. Then  the following relations hold: 
\begin{multicols}{2}
    \item[(1)] $%
=0$.

\end{multicols}
    \end{Lemma}

\begin{proof} We verify  (1); the proof of (2) is analogous. 
Let $V$ be a  
 finitely generated object. We  evaluate the left-hand side of (1) on $V$. 
  As before,  we do not draw the extra strand on the far right representing the object $V$ on which the diagram acts. We have  
 %   to denote this diagrammatic action. 
\begin{equation}\label{equa-prime1}
\Big(
%
= 0,
\end{equation} for some $n_1\in\mathbb{N}$. By Lemma~\ref{ortho-dotcom} and Lemma~\ref{AB relations-re}(1), we also have 
\begin{equation}\label{equa-prime2}
\Big(
%
= 0,\
\end{equation} for some $n_2\in\mathbb{N}$. Since $i\neq -j$, the polynomials 
$u-i$ and $u+j$ are coprime.  Hence, there exist polynomials 
$f(u), g(u)\in \Bbbk[u]$ such that
%\begin{equation}\label{fg321}
$$g(u)(u - i)^{n_1} + f(u)(u + j)^{n_2} = 1.$$
%\end{equation}
Combining \eqref{equa-prime1}-\eqref{equa-prime2}, we obtain  (1).%Finally, (3) and(4) follows from (1) and (2) together with %Lemma~\ref{ortho-dotcom}(1). 
\end{proof}

\begin{Lemma}\label{ibiad}For any $i\in \Bbbk$,  
%$(E_i,E_{-i})$ is a biadjoint pair and hence %
the endofunctor $E_i$ is a sweet functor.
\end{Lemma}
\begin{proof}
We have 
$$
.$$ A similar computation shows that $
 %
$.
In other words, $(E_i,E_{-i})$ is a biadjoint pair and hence $E_i$ is a sweet functor. 
\end{proof}

\begin{Defn}\label{III} Let  $ J=\{ i\in \k\mid E_iL\neq 0 \text{ for some simple object } L \in \mathcal M \}$, where  
 $\mathcal M$ is a category admitting  an  affine  Brauer categorification. We call J the spectrum of the dot endomorphism.\end{Defn} 

 \begin{Lemma}
 $E=\bigoplus_{i\in J} E_i$.     
 \end{Lemma}
\begin{proof} First, we have  $E=\bigoplus_{i\in \Bbbk} E_i$ by Lemma \ref{ortho-dotcom}(1). It suffices  to prove   that if  $E_iV\neq  0$ for some object  $V$, then 
$E_iL\neq 0$ for some simple object $L$.   If so, we have $E=\bigoplus_{i\in J} E_i$.

Since we assume that $\mathcal M$ is a locally Schurian category, every object is a filtered colimit of finitely generated objects.
By Lemma~\ref{ibiad}, $E_i$ is a sweet functor, and hence it preserves the filtered colimit. Thus,  if $E_iV\neq  0$ for any object $V$, then  $E_iW\neq 0$ for some finitely generated object $W$. Hence,  we can assume that $V$ is a finitely generated object. Since $E_iV$ is a nonzero  finitely generated object in a locally Schurian category, it has a simple quotient, say L; see 
 \cite[(L6)]{BD-cate}.  Thus, there is a non-zero homomorphism from $E_iV$ to $L$. Since $(E_i, E_{-i})$ is a biadjoint pair, $E_{-i}L\neq 0$. Note that $E_{-i}$ preserves finitely generated objects. Since $E_{-i}L$ is nonzero and finitely generated, it has a simple quotient $L'$. Applying the biadjunction between $E_i$ and $E_{-i}$ to the nonzero map $E_{-i}L\rightarrow L'$ gives
  $E_iL'\neq 0$. This completes the proof of the lemma.
\end{proof}

Suppose   
$i,j,i',j'\in J$. 
For  brevity, we  write \begin{equation}
.
\end{equation}

\begin{Lemma}
\label{dotslidecrossing} Suppose %For any   
$i,j,i',j'\in J$. Then,  the following relations hold:   
  \begin{itemize} \item [(1)] $ 
%\begin{align}
\mathord{
%
=0$ whenever   $\{i,j\}\neq \{i',j'\}$ and at least one of the conditions   $i'\neq -j'$ or   $i\neq -j$ holds.
\end{itemize} 
%\label{i-dot}
%\end{align}
\end{Lemma}
\begin{proof} The relation (2) follows from Lemma~\ref{ortho-dotcom}, Definition~\ref{C defn}(4), and Lemma~\ref{cap-cup-move}(1). 
 Using Lemma~\ref{AB relations-re}(3) to replace  Definition~\ref{C defn}(4), one     
can show (1) in the same way. 

It follows from both (1) and (2) that  
   $$ %
, $$
whenever $\{i,j\}\neq \{i',j'\}$ and 
 either $i'\neq -j'$ or $i\neq -j$.  Fix a finitely generated object $V$. If  $i'\neq j$, 
choose sufficiently large integers $k$ and $l$ so that the corresponding powers of $x-j$ and $x-i'$ annihilate the relevant 
generalized eigenspaces after evaluation on $V$. Since $u-j$ and $u-i'$ are coprime,  
 there are  $f(u), g(u)\in \Bbbk[u]$
such that $$f(u)(u-j)^k+g(u)(u-i')^l=1.$$ 
Set  $ p(u)=g(u)(u-i')^l $. Then after evaluating on $V$,  %这里是逻辑关系， 如果不先取定V， 由于这样的V是无穷个， 我们没有办法取一个公共的V
we obtain 
$$
\mathord{
%
}
=
0.
$$  It remains only to consider the case \(i'=j\). Since
\(\{i,j\}\ne\{i',j'\}\), we must have \(j'\ne i\). This case is treated
by the same argument, and we omit the details. 
\end{proof}

\begin{Lemma}\label{crossing-inverse}
For any $i, j\in J$ with $j \neq \pm i$,  
$
%
+\cdots,
\]
and  $c =
(j-i)^{-1}$, and the omitted terms are higher powers of the nilpotent operators $x_1-j$ and $x_2-i$.
\end{Lemma}
\begin{proof} First, note that  %\begin{equation} %\label{x1x2arr} 
$$(x_1-x_2)^{-1}=c
- c^2(x_1-j)
+ c^2(x_2-i) + \text{(higher order terms)},$$ 
 and  since  $x_1-j$ (respectively $x_2-i$) acts nilpotently  on %
), it follows that     $%
$ contains  only  finitely  many nonzero  terms when evaluated on the relevant generalized eigenspaces. Thus, it  is well-defined. 

Since we assume  that $j\neq - i$,   Lemma~\ref{dotslidecrossing}(1)-(2) implies that 
\begin{align}\label{47}
\mathord{
%
.
\end{align}
Since \eqref{47} is  the same as \cite[(4.7)]{BSW-heisenberg},  the result follows from  the arguments in  \cite[Lemma~4.2]{BSW-heisenberg} by forgetting the arrows there.
\end{proof}

\begin{Lemma} \label{ii}Suppose $0\neq i\in J$. Then 
\begin{equation}\label{iii} 
%
.\end{equation}
\end{Lemma}

\begin{proof}
Set $
\alpha=\text{LHS of \eqref{iii}}-\text{RHS of \eqref{iii}}$. 
Let 
$R$ denote the dot placed on the lower segment of the left strand of $\alpha$, and let  $
L$ denote  the dot placed on the upper segment of the right strand of $\alpha$.
%Let \(L\) (resp. \(R\)) be the dot acting on the left (resp. %right) strand.
Then 
\begin{align*}
R\alpha
&=
%
.
\end{align*}

Subtracting the two displayed identities and simplifying the result using   Lemma~\ref{dotslidecrossing}(1)--(2), we obtain 
$(R-L)\alpha=0$. Evaluating   on a finitely generated object $V$, we have 
$(R_V-L_V)\alpha_V=0$. Since $R_V+i$
and $L_V-i$ are commuting nilpotent endomorphisms on the relevant
generalized eigenspaces, there exists a  sufficiently large integer
$N=N(V)$ such that
\[
\bigl[(R_V+i)-(L_V-i)\bigr]^N\alpha_V=0.
\]
Since $(R_V-L_V)\alpha_V=0$, it follows that
\(
(2i)^N\alpha_V=0
\).
As $2i\neq0$, we obtain $\alpha_V=0$. Since this holds for every
finitely generated object $V$, Remark~4.3 implies that $\alpha=0$.
%我感觉有点写的太清楚了
\iffalse Using Lemma~\ref{dotslidecrossing} yields $(R-L)\alpha=0$.   Since   % \(x_1+i\) and \(x_2-i\) act nilpotently on every object,
both \(R+i\) and \(L-i\) are  commuting  nilpotent natural transformations on the relevant generalized eigenspaces, it follows that 
 $[(R+i)-(L-i)]^N\alpha=0$ for a sufficiently large integer $N$.
 Consequently,  $(2i)^{N}\alpha=0$.  Since $2\neq 0$ and $i\neq 0$ in $\Bbbk$,   $\alpha=0$, as required.\fi
\end{proof}

%休息一会再继续

\subsection{Center and weights}\label{secfour} Every endomorphism of \(\Id_{\mathcal M}\) defines a central element of
\(\mathcal M\). Hence the image under \(\mathrm M\) of any dotted bubble in
\(\mathcal{AB}\) is central in \(\mathcal M\). Evaluating such a bubble on an
object \(V\in\mathcal M\) gives an element of  %$Z_V$,  
the center    of   $\End_{\mathcal M}(V)$.
By abuse of notation, we  use the same symbols for dotted bubbles and for
their images under \(\mathrm M\).

\begin{Lemma} \label{jjj} Let $V$ be a finitely generated object of $\mathcal M$, and suppose that $f(u)\in R[u]$,
where  $R$ is a $\Bbbk$-subalgebra of the center  of $\End_{\mathcal M} (V)$.
Then     the following relation holds:  \begin{equation} \label{bubb-key0}     =[\clock(-u) f(-u)]_{u^0}+\frac{ 1}{2}f(0).\end{equation}\end{Lemma}
\begin{proof} After evaluating  on a finitely generated object $V$,   Lemma~\ref{usefuel-equa}(1) gives 
\begin{align*}
    %
 f(x)  & = \Bigg[ \frac{\clock (u) - u + \frac{1}{2}}{u} f(u) \Bigg]_{u^{-1}} = [\clock (u)f(u)]_{u^0} + \frac{1}{2} f(0) \\ & =[\clock(-u)f(-u)]_{u^0}+\frac{1}{2}f(0).
\end{align*} Therefore, \eqref{bubb-key0} holds.  
\end{proof}

\begin{Lemma} \label{lem:bublleyimeszeropoly}
For any  finitely generated object $V\in \mathcal  M$, define 
  $\mathcal O_V(u):=
\mathord{
%
}$. 
\begin{itemize} \item [(1)] Let   $g(u) = \mathcal O_V(u)\, f(u) $, where
$f(u) \in \Bbbk[u]$ is an annihilating polynomial of %
 acting  on $V$. Then $g(u)$ { is a polynomial whose coefficients lie in the subalgebra of $\operatorname{End}_{\mathcal M} (V)$
generated by the dotted bubbles evaluated on $V$, and hence so is $g(-u)$}. 
Further, 
$g(-u)$ is  an annihilating   polynomial of %
 acting  on $V$.
\item [(2)] If $f(u)=m_V(u)$ denotes the
 minimal polynomial  of  $ %
$  acting on  $V$, then we denote $g(u)$ by $n_V(u)$. In particular, when  
$V=L$ is  simple,  we have  $n_L(u)\in \Bbbk[u]$ such that $m_L(u)\mid n_L(-u)$.
\end{itemize}\end{Lemma}

\begin{proof}
    For $r>0$, Lemma \ref{jjj} gives  
    \[ [\mathcal O_V(u) f(u)]_{u^{-r}} =[\mathcal O_V(u) f(u)u^r]_{u^0}
    =
    %
=0.\]
Hence   $g(u)$    is a polynomial whose coefficients lie in the subalgebra of $\operatorname{End}_{\mathcal M} (V)$
generated by the dotted bubbles evaluated on $V$.   Now, after evaluating  on $V$,    Lemma \ref{usefuel-equa}(3) gives  
\begin{align*}
\mathord{
%
}=0.
\end{align*} Therefore $g(-u)$ is also an annihilating  polynomial for 
$
%
$ 
acting   on $V$. This proves (1). When $V$ is the simple object $L$, $%
$ is a scalar in $\Bbbk$. Hence $ n_L(u)\in \Bbbk[u]$ by  applying (1). Moreover, since $n_L(-u)$ is an annihilating  polynomial for 
 %
 acting  on $L$, we obtain $m_L(u)\mid n_L(-u)$, proving (2).
\end{proof}

Throughout the paper, we use  the notation $m_V(u)$ and $n_V(u)$ introduced  above. In particular,  $m_V(u)\in \Bbbk[u]$ denotes    the minimal polynomial of the action of  $ %
$   on the finitely generated object $V\in \mathcal M$.
 If   $EV=0$, we set $m_V(u)=1$ and $a_V=0$, where  $a_V:=\deg m_V(u)$.

\begin{Lemma}  \label{u-admiss}  
 $\mathcal O_L(u)=   ((-1)^{a_L}u-\frac{1}{2}) \frac{m_L(-u)}{m_L(u)} $, where  $L\in \mathcal M$ is a  simple object. \end{Lemma}
\begin{proof}
If $EL=0$, then for each $i\in \mathbb N$, 
 $%
=0$ and the result follows.  We henceforth assume that $EL\neq 0$. In particular,   $a_L=\deg m_L(u)>0$.
 By Lemma  \ref{lem:bublleyimeszeropoly}(2), $n_L(-u)\in \Bbbk[u]$ and  $m_L(u) | n_L(-u)$. 
 Moreover, 
    Lemma \ref{ou}(1) gives 
    \[ n_L(u)n_L(-u)=(\frac{1}{2}-u)(\frac{1}{2}+u)m_L(-u)m_L(u) \]
   Comparing the leading coefficients on both sides, we obtain    
    \begin{equation}\label{nmkey}
     n_L(u)= ((-1)^{a_L}u+\frac{1}{2}\sigma)m_L(-u)   
    \end{equation}
    for some $\sigma\in \{\pm 1\} $. 
  It remains to show that $\sigma=-1$.  Evaluating on the simple object $L$,  Lemma \ref{usefuel-equa}(3) and \eqref{nmkey} yields  
\begin{align*}
   0=
,
\end{align*}
where $h(u)=[\frac{m_L(u)}{u}]_{u^{\ge 0}}$. 
Since $\deg m_L(u)\ge 1$, we have $h(u)\neq 0$. Together with the fact that $\deg h(u)<\deg m_L(u)$,  this forces   $\sigma=-1$.    
\end{proof}

\begin{Lemma}\label{choosesimple}
Suppose that $L$ is a simple object  in $\mathcal M$ and that $K$ is  an irreducible subquotient of $E_i L$ for some $i \in \k$. 
Then
\begin{equation} \label{choose1}
    \mathcal O_K(u)
    = \mathcal O_L(u)\frac{(u+i)^2-1}{(u-i)^2-1} \frac{ (u-i)^2}{(u+i)^2}  .
\end{equation}
\end{Lemma}

\begin{proof}
Lemma~\ref{ou}(2) gives 
$$
\mathord{
%
}.
$$
Here \(x\) denotes the dot acting on \(E_iL\). Since \(K\) is an irreducible
subquotient of \(E_iL\), the dot \(x\) acts on \(K\) by the scalar \(i\). Passing
to the subquotient \(K\) therefore gives
\[
\mathcal O_K(u)
=
\mathcal O_L(u)
\frac{(u+i)^2-1}{(u-i)^2-1}
\frac{(u-i)^2}{(u+i)^2},
\]
as claimed.
\end{proof}
From now on,  we impose   the following assumption.
\begin{Assumption}\label{iIassu}
Let $J$ be the set defined in Definition~\ref{III}. We assume that
\begin{itemize}
\item[(1)] $J=\mathbb Z$ or $J=\frac12+\mathbb Z$ if $\mathrm{char}\,\Bbbk=0$,
\item[(2)]  $J=\mathbb Z_p$   or $J=\frac12+\mathbb Z_p$ if $\mathrm{char}\,\Bbbk=p$, {where
$\mathbb Z_p$ denotes the image of $\mathbb Z$ in $\Bbbk$ (equivalently, the prime subfield of $\Bbbk$).}
\end{itemize}
\end{Assumption}

This assumption is natural. Indeed, when $\mathcal M$ is an $\mathcal{AB}$-module
category arising from the category of locally finite-dimensional left modules
over a cyclotomic Brauer category, the set $J$  is  among the
index sets appearing in Assumption~\ref{iIassu} under the
$\mathbf u$-admissibility condition; see \cite{GRS-tri} for details.
For the affine Kauffman category, there are two additional choices for
$J$; see \cite{GRS-Kaff}. 

In the present paper, we  assume 
\[
J=\frac12+\mathbb Z= I,\qquad J^\imath=\frac12+\mathbb N= I^\imath,
\]
and assume that $\Bbbk=\mathbb C$.   The remaining  cases will be treated  elsewhere. 
 Thus, $J$ is the set  $I$ defined  in \eqref{index-I}. 
This is the case needed to establish the
relationship between   $2$-representations of  $\mathfrak U_+^\imath$ and module
categories over the affine Brauer category.  Some of the results below, however, do not depend on these
assumptions, and this will be clear from the context.

The equality $J=I$ is not merely a notational convenience. The construction of the $\mathfrak U^{\imath}_{+}$-action requires the generating eigenspace functors for every color in the half-integral orbit, together with the adjacent-color crossings and the boundary morphisms at $\frac12$. If some colors are absent, the same formulas still define transformations on the summands that exist, but they do not automatically assemble to a 2-representation of the full category $\mathcal U^{\imath}_{+}$. This is the source of the asymmetry between the two parts of Theorem A.

\begin{Defn} \label{expre}   For any simple object $L$ in  $ \mathcal M$,   write
$m_L(u)=\prod_{i\in \k}(u-i)^{\epsilon_i(L)}$.
Define \begin{itemize}\item [(1)]
$\text{wt}_L=
 \Lambda_{\frac{1}{2}(-1)^{a_L}}+ \sum_{i\in I}(\epsilon_{-i}(L)-\epsilon_{i}(L))\Lambda_i$, 
 \item [(2)] 
    $ {\text{wt}}^\imath_L= 
	\delta_{a_L}\Lambda_{\frac{1}{2}(-1)^{a_L}}+ \sum_{i\in I}\epsilon_{-i}(L)\Lambda_i \in X_\imath$
    \end{itemize}   where $\delta_{a_L}=1$ if $a_L$ is  odd and $\delta_{a_L}=0$ if $a_L$ is  even.  
    %这里这样改是为了使得权的定义满足 引理 2.18 的(2)-(3) 以及推论 2.23
    \end{Defn} 
  In particular, when $EL=0$, $m_L(u)=1$ and $a_L=0$, and hence $\text{wt}_L=\Lambda_{\frac12}$.
 
\begin{Lemma}\label{BLOCK} Let $L$ and $K$ be  simple objects  of  $\mathcal M$. Then the following statements hold.
     \begin{itemize}\item[(1)] For any  $i\in I^\imath$,   $ \langle{ }^\theta \alpha_i^\vee, {\text{wt}}^\imath_L \rangle=  \langle \alpha_i^\vee, \text{{wt}}_L\rangle-\delta_{i, \frac{1}{2}} $.
         \item [(2)] If $L$ and $K$ are in the same block of $\mathcal M$, then ${\text{wt}}^\imath_L={\text{wt}}^\imath_K $. 
         \item [(3)]  If $K$ is a subquotient of $E_iL$ for some  $i\in I$, then 
         $ {\text{wt}}^\imath_K ={\text{wt}}^\imath_L +\alpha_i={\text{wt}}^\imath_L-\alpha_{-i} $.
         \item [(4)]  There is a block decomposition $ \mathcal M=\prod_{\lambda\in  X_\imath}\mathcal M_{\lambda}$, where $\mathcal M_\lambda$ is  the Serre subcategory of $
\mathcal M$ consisting of objects $
V$ such that every simple subquotient $
L$ of $
V$ satisfies $\text{wt}^\imath_L
 =\lambda$.     In the locally    finite Abelian case, the product is replaced by the direct sum.  
%the Serre subcategory of $\mathcal M$ consisting of  objects $V$ such that $\overline{\text{wt}}_L=\lambda$ for  each  subquotient $L$ of $V$.   
         \item[(5)] For any $i\in I$, the functor  $ E_i$ maps   $\mathcal M_{\lambda}$  into $ \mathcal M_{\lambda+\alpha_i}= \mathcal M_{\lambda-\alpha_{-i}}$.      
  \item[(6)] If ${\text{wt}}^\imath_L= {\text{wt}}^\imath_K$, then $ {\text{wt}}_L= {\text{wt}}_K$ (i.e. $\mathcal O_L(u)=\mathcal O_K(u)$).     
         \end{itemize}
 \end{Lemma}
\begin{proof} Statement  (1) follows immediately from  Definition~\ref{expre}(1) and (2).
If $L$ and $K$ are in the same block, then  $\mathcal O_L(u)=\mathcal O_K(u)$, forcing   $\text{wt}_L=\text{wt}_K$ in $X$, and  
     $  {\text{wt}}^\imath_K= {\text{wt}}^\imath_L$ in $X_\imath$,  proving (2).
     %Since $\alpha_i=2\Lambda_i-\Lambda_{i-1}+\Lambda_{i+1}$.
 If $K$ is a subquotient of $E_iL$, by \eqref{www1} and Lemma~\ref{choosesimple},  we have
    $ \text{wt}_K=\text{wt}_L+\alpha_i-\alpha_{-i} $.
 Suppose  $j\in I^\imath$. By (1)  and \eqref{equa-pairing1},
we compute $$
    \langle {}^\theta\alpha_j^\vee,  {\text{wt}}^\imath_K-{\text{wt}}^\imath_L\rangle 
    %&= (\alpha_j^\vee, \text{wt}_K-\text{wt}_L)
  %  =(\alpha_j^\vee, \alpha_i-\alpha_{-i})
   %= (\alpha_j^\vee, \alpha_i)-(\alpha_j^\vee, \alpha_{-i})\\
   = \langle \alpha_j^\vee, \alpha_i\rangle -\langle \alpha_{-j}^\vee, \alpha_{i}\rangle 
   = \langle {}^\theta\alpha_j^\vee, \alpha_i\rangle .
%\end{aligned}
$$
This proves the first equality in (3). The second equality in (3) follows immediately from the fact that  $\alpha_i+\alpha_{-i}$ is $ 0$  in $X_\imath$. Finally,  (4) follows from (2), and   (5) follows from (3). 
 Note that for any simple object $L$ and any $i\in I^\imath$, Definition~\ref{expre}(1) implies that $$\langle \alpha_{-i}^\vee, \text{{wt}}_L\rangle+\langle \alpha_i^\vee, \text{{wt}}_L\rangle=\delta_{i, 1/2}.$$ Suppose ${\text{wt}}^\imath_L={\text{wt}}^\imath_K$. By  (1) and the preceding equation, we have  $\langle \alpha_i^\vee, \text{{wt}}_L\rangle=\langle \alpha_i^\vee, \text{{wt}}_K\rangle$ for any $i\in I$. Therefore,  $ {\text{wt}}_L= {\text{wt}}_K$.
\end{proof}

\subsection{Central algebras associated with $\mathcal M_\lambda$}
 By the basis theorem given in \cite{RS-cyc},  the endomorphism algebra of the unit object $\End_{\AB}(0)$ in the affine Brauer category is the  unital polynomial algebra  generated by all $\begin{tikzpicture}[baseline = 1.25mm]
  \draw[-] (0,0.4) to[out=180,in=90] (-.2,0.2);
  \draw[-] (0.2,0.2) to[out=90,in=0] (0,.4);
 \draw[-] (-.2,0.2) to[out=-90,in=180] (0,0);
  \draw[-] (0,0) to[out=0,in=-90] (0.2,0.2);
   \node at (0.2,0.2) {$\scriptstyle\bullet$};
   \node at (0.4,0.2) {$\scriptstyle{k}$}; %\draw[-,thick] (.6,.4) to (.6,-.2);
      %\node at (.9,0) {$\scriptstyle{\mathcal M_\lambda}$};
\end{tikzpicture}$, $k\in \mathbb N$. Throughout, when we talk about a finitely generated object $V$, we always assume that $V\neq 0$. Otherwise, there is nothing to prove. 

 \begin{Defn} For  any $\lambda\in X_{\imath}$ with $\mathcal M_\lambda\neq 0$, let $Z_\lambda$ be the image of $\End_{\AB}(0)$ in  $\End(\text{Id}_{\mathcal M_\lambda})$. %generated by all diagrams  
% $\begin{tikzpicture}[baseline = 1.25mm]
%   \draw[-] (0,0.4) to[out=180,in=90] (-.2,0.2);
%   \draw[-] (0.2,0.2) to[out=90,in=0] (0,.4);
%  \draw[-] (-.2,0.2) to[out=-90,in=180] (0,0);
%   \draw[-] (0,0) to[out=0,in=-90] (0.2,0.2);
%    \node at (0.2,0.2) {$\scriptstyle\bullet$};
%    \node at (0.4,0.2) {$\scriptstyle{k}$}; \draw[-,thick] (.6,.4) to (.6,-.2);
%       \node at (.9,0) {$\scriptstyle{\mathcal M_\lambda}$};
% \end{tikzpicture}\in \End(\text{Id}_{\mathcal M_\lambda})$ for all $k\in \mathbb N$, viewed as  natural transformations of the identity functor.
Define
\[
 J_\lambda
 :=
 \left\{
 b\in Z_\lambda
 \ \middle|\
 b\text{ acts as zero on every simple object }
 L\in\mathcal M_\lambda
 \right\}.
 \]
\end{Defn} 

 \begin{Lemma}\label{lem:localring}  For  any $\lambda\in X_{\imath}$ with $\mathcal M_\lambda\neq 0$,   $J_\lambda$ is a maximal ideal of $Z_\lambda$. Moreover, for every
 integer $s\geq 1$, the quotient
 \(
 Z_{\lambda,s}:=Z_\lambda/J_\lambda^s
 \)
 is a local $\Bbbk$-algebra with unique maximal ideal
 $J_\lambda/J_\lambda^s$.
 \end{Lemma}

 \begin{proof}
 The algebra $Z_\lambda$ is commutative. By Lemma~\ref{BLOCK}(6), every
 $b\in Z_\lambda$ acts by a scalar on each simple object of
 $\mathcal M_\lambda$, and this scalar is independent of the choice of
 the simple object. Hence there is a well-defined surjective algebra
 homomorphism
 \[
 \chi_\lambda:Z_\lambda\longrightarrow\Bbbk
 \]
 which sends $b$ to this scalar.  Therefore, its kernel $J_\lambda$ is a maximal ideal of $Z_\lambda$.

 Let $s\geq 1$, and let $\mathfrak m/J_\lambda^s$ be a maximal ideal of
 $Z_\lambda/J_\lambda^s$. Then $\mathfrak m$ is a maximal ideal of
 $Z_\lambda$ containing $J_\lambda^s$. Since $\mathfrak m$ is prime,
 for every $b\in J_\lambda$, the inclusion
 \(
 b^s\in J_\lambda^s\subseteq\mathfrak m
 \)
 implies that $b\in\mathfrak m$. Hence
 $J_\lambda\subseteq\mathfrak m$. Since $J_\lambda$ is maximal, it
 follows that $\mathfrak m=J_\lambda$. Thus $Z_{\lambda,s}$ is local
 with unique maximal ideal $J_\lambda/J_\lambda^s$.
 \end{proof}

 \iffalse 
\begin{Lemma}

    For any $\lambda\in X_\imath$,  $Z_\lambda$ is a local $\Bbbk$-algebra.
\end{Lemma}

\begin{proof}
    The algebra  $Z_\lambda$ is commutative. By Lemma~\ref{BLOCK}(6), 
    \(
    \begin{tikzpicture}[baseline = 1.25mm]
      \draw[-] (0,0.4) to[out=180,in=90] (-.2,0.2);
      \draw[-] (0.2,0.2) to[out=90,in=0] (0,.4);
      \draw[-] (-.2,0.2) to[out=-90,in=180] (0,0);
      \draw[-] (0,0) to[out=0,in=-90] (0.2,0.2);
      \node at (0.2,0.2) {$\scriptstyle\bullet$};
      \node at (0.4,0.2) {$\scriptstyle{k}$};
      \draw[-,darkg,thick] (.6,.4) to (.6,-.2);
      \node at (.7,0) {$\darkg\scriptstyle{L}$};
    \end{tikzpicture}
    \)
    is a   scalar in $\Bbbk$ for any simple object $L$ in $\mathcal{M}_\lambda$, and this scalar is independent of the choice of $L$.
Define
    \[
        J_\lambda = \{ b \in Z_\lambda \mid b \text{ acts as zero on every simple object $L$ in $\mathcal M_\lambda$ } \}.
    \]
    It remains to prove  that $J_\lambda$ is the unique maximal ideal of $Z_\lambda$.

    First, $J_\lambda$ is a proper ideal since $1 \notin J_\lambda$. Now, take any  $b \notin J_\lambda$. Then $b$ acts as a nonzero scalar $c$ on some simple object $L \in \mathcal{M}_\lambda$, and hence on every simple object in $\mathcal{M}_\lambda$ (by the scalar independence mentioned above). By  \cite[Lemma 2.12]{BD-cate}, 
    $b$ is then a natural isomorphism, and therefore invertible in $Z_\lambda$. It follows that $J_\lambda$ consists precisely of the non-invertible elements of $Z_\lambda$, so it is the unique maximal ideal.
\end{proof}
\fi

\begin{Defn}\label{phiV}
    For any finitely generated object $V\in \mathcal M_\lambda$, define $Z_{\lambda,V}=\phi_V(Z_\lambda)$, where  
\begin{equation}\label{evalu-na}\phi_V: \End(\text{Id}_{\mathcal M_\lambda}) \rightarrow \End_{\mathcal M_\lambda}(V) \end{equation}
is  the algebra homomorphism given by evaluating a natural endomorphism $\eta$ of the identity functor $\text{Id}_{\mathcal M_\lambda}$ at the object $V$.
\end{Defn}

\begin{Lemma} \label{zlam-loc} Let \(V\) be a finitely generated object of \(\mathcal M_\lambda\). Then  $Z_{\lambda,V}$ is a local $\Bbbk$-algebra with unique
maximal ideal $J_{\lambda,V}:= \phi_V(J_\lambda)$.\end{Lemma} 
\begin{proof} 
Since $V$ is a nonzero finitely generated object, it has a simple
quotient. If $b\in\ker\phi_V$, then $b$ acts as zero on this simple
quotient. By Lemma~4.20(6), the common scalar by which $b$ acts on all
simple objects of $\mathcal M_\lambda$ is therefore zero. Hence
\(
\ker\phi_V\subseteq J_\lambda,
\)
and consequently
\[
Z_{\lambda,V}/J_{\lambda,V}
\cong Z_\lambda/J_\lambda
\cong\Bbbk.
\]
Thus $J_{\lambda,V}$ is a maximal ideal of $Z_{\lambda,V}$.

Moreover, every element of $J_{\lambda,V}$ is nilpotent. Indeed, let
$a\in J_{\lambda,V}$, and choose $b\in J_\lambda$ such that
$a=\phi_V(b)$. For every $c\in\Bbbk$, the natural transformation
\[
1-cb:\operatorname{Id}_{\mathcal M_\lambda}
\Longrightarrow \operatorname{Id}_{\mathcal M_\lambda}
\]
acts as the identity on every simple object of $\mathcal M_\lambda$.
Evaluating on $V$, we see that
\[
1-ca=\phi_V(1-cb) \in\operatorname{End}_{\mathcal M_\lambda}(V).
\]
Hence, by \cite[Lemma~2.12]{BD-cate}, $1-cb$ is a natural isomorphism, since $b$ acts on every simple object $L$ as zero. This implies that $1-ca$ is invertible for every $c\in\Bbbk$.

Since $\operatorname{End}_{\mathcal M_\lambda}(V)$ is finite-dimensional
and $\Bbbk$ is algebraically closed, this implies that $a$ is
nilpotent. Indeed, consider the linear operator on
$\operatorname{End}_{\mathcal M_\lambda}(V)$ given by left
multiplication by $a$. If $a$ were not nilpotent, this operator would
have a nonzero eigenvalue $r\in\Bbbk$. Then left multiplication by
$1-r^{-1}a$ would not be invertible, contradicting the invertibility
of $1-r^{-1}a$. Thus every element of $J_{\lambda,V}$ is nilpotent.
Since $Z_{\lambda, V}$ is finite-dimensional, the ideal $J_{\lambda, V}$  is nilpotent. 
Hence every
maximal ideal of the commutative algebra $Z_{\lambda,V}$ contains
$J_{\lambda,V}$. Since $J_{\lambda,V}$ is itself maximal, it is the
unique maximal ideal of $Z_{\lambda,V}$. Therefore
$Z_{\lambda,V}$ is local.
\end{proof}

For any $f\in Z_{\lambda}$, let $\bar f\in \Bbbk$ be the image of $f$ in $Z_\lambda/ J_\lambda$. Similarly, for any $g\in Z_{\lambda, V}$,  we write $\bar g$ for its image  in $\Bbbk$. 

\begin{Lemma}
\label{barmathcalO}
Let \(V\) be a finitely generated object of \(\mathcal M_\lambda\), and let
\(L\) be a simple object of \(\mathcal M_\lambda\). Then
\[
 \overline{\mathcal O}_V(u)
 =
 {\mathcal O}_L(u)
 =
 \left((-1)^{a_L}u-\frac{1}{2}\right)
 \frac{m_L(-u)}{m_L(u)}.
\]
\end{Lemma}

\begin{proof}
The dotted bubbles define endomorphisms of the identity functor on
\(\mathcal M\), and hence act as central elements on every object of
\(\mathcal M\). Therefore their actions are compatible with subobjects,
quotients, and subquotients.

Let $z$ be any coefficient of the dotted-bubble series
$\clock(u)$. By Lemma~\ref{BLOCK}(6), $z$ acts on every simple object of
$\mathcal M_\lambda$ by the same scalar, say $\chi_\lambda(z)$.
Hence
\(
z-\chi_\lambda(z)1\in J_\lambda\),
and therefore
\(
z_V-\chi_\lambda(z)1_V\in J_{\lambda,V}\).
It follows that the image of $z_V$ in
$Z_{\lambda,V}/J_{\lambda,V}$ is equal to its scalar action on any
simple object $L\in\mathcal M_\lambda$. Applying this coefficientwise
to $\mathcal O_V(u)$ gives
\(
\overline{\mathcal  O_V(u)}=\mathcal O_L(u)
\).
Combining this with  Lemma~\ref{u-admiss} gives  the desired formula.
\end{proof}

\iffalse 
Since \(V\) is finitely generated and belongs to 
\(\mathcal M_\lambda\), all its simple subquotients lie in
\(\mathcal M_\lambda\). Moreover, by {Lemmas~\ref{BLOCK}(6) and ~\ref{u-admiss}, the  dotted bubbles have the same scalar action on all simple objects of $\mathcal M_\lambda$}. Hence the action of the dotted bubble on \(V\)
is determined, modulo  nilpotent endomorphisms, by its action on any simple constituent \(L\) in $\mathcal M_\lambda$.
Thus $
\overline{\mathcal O}_V(u)=\mathcal O_L(u)$.\fi

Throughout the paper, unless otherwise stated, we use the following notation:
\begin{equation}\label{mini-1}
\begin{cases}
m_V(u)=\displaystyle\prod_{i\in\Bbbk}(u-i)^{\epsilon_i(V)}, 
\qquad 
m_{V,i}(u)=\displaystyle\prod_{j\neq i}(u-j)^{\epsilon_j(V)},\\[0.8em]
\bar n_V(u)=\displaystyle\prod_{i\in\Bbbk}(u-i)^{\phi_i(V)}, 
\qquad 
\bar n_{V,i}(u)=\displaystyle\prod_{j\neq i}(u-j)^{\phi_j(V)}.
\end{cases}
\end{equation}
Here $i,j\in\Bbbk$. Note that $\epsilon_i(V)=0$ for all $i\notin I$, so only the factors indexed by elements of $I$ contribute to the above products.

\begin{Cor} \label{cor:phievi} 
Suppose that $L$ is a simple constituent of a finitely generated object  $V\in\mathcal M_\lambda$.  % $ \phi_i(V)-\epsilon_i(V)=$
For any $i\in I^\imath$,   we  have 
\begin{equation}
\label{equ:phiminusepslon} \epsilon_{-i}(L)-\epsilon_i(L)+\delta_{i,\frac{1}{2}(-1)^{a_L}}
  = \phi_i(V)-\epsilon_i(V)= \begin{cases} \langle ^\theta\alpha_i^\vee,\lambda \rangle, & \text{ if } i\neq  \frac{1}{2},\\
    \langle ^\theta\alpha_i^\vee,\lambda \rangle+1, & \text{ if } i=  \frac{1}{2}.
     \end{cases}
\end{equation}
%这里是因为新的seting 下 如果圈的关系式对应的话 就需要这个式子）
    \end{Cor}
\begin{proof} It follows from Lemma~\ref{lem:bublleyimeszeropoly} and 
 Lemma \ref{barmathcalO} that\begin{equation}\label{eat}
    \frac{\bar n_{V}(u)}{m_V(u)}= \mathcal O_L(u) =  ((-1)^{a_L}u-\frac{1}{2}) \frac{m_L(-u)}{m_L(u)}.  
 \end{equation}   
This implies  the first equality  in~\eqref{equ:phiminusepslon}. The second equality in~\eqref{equ:phiminusepslon}  follows by a direct  computation  from Definition \ref{expre} and the definition of    $\mathcal M_\lambda$  in Lemma \ref{BLOCK}(4). We omit the details. 
 \end{proof}

 \begin{Cor}
 \label{cor:even}
 Every weight $\lambda$ occurring in  $ M$  has  parity $0$. 
 %For each $\mathcal M_\lambda$,  the parity of $\lambda$ is zero.  %   
 %$\bar \lambda$ must  be  even, where  $\bar \lambda =.    
 \end{Cor}
 \begin{proof} For any  simple object  $L\in \mathcal M_\lambda$,  we have 
  $a_L=\sum _{i\in I^\imath}(\epsilon_i(L)+\epsilon_{-i}(L))$. By \eqref{equ:phiminusepslon}, $\sum_{i\in I^\imath} \langle ^\theta\alpha_i^\vee,\lambda \rangle$ is always even, regardless of  whether $a_L$ is even or odd. This forces  the parity of $\lambda$ to be zero. \end{proof}
  
\begin{Lemma}
\label{lem:moniclift}
    \cite[Corollary 2.3]{BSW-heisenberg} Let $Z$ be a  commutative local $\Bbbk$-algebra with unique maximal ideal $J$, and assume that $J^t=0$ for some $t$.  Suppose $f(u)\in Z[u]$ is monic and $\bar f(u)=\bar g(u)\bar h(u)$ for  coprime monic polynomials $\bar g(u)$, $\bar h(u)\in \Bbbk[u]$. Then, there exist unique monic lifts  $ g(u), h(u)\in Z[u]$ of $\bar g(u)$, $\bar h(u)$ such that $g(u)$ and $h(u)$ are coprime and  $f(u)=g(u)h(u)$. 
\end{Lemma}

The proof of \cite[Corollary 2.3]{BSW-heisenberg}  applies verbatim to
this formulation; it uses only the nilpotence of $J$, not the finite-dimensionality of $Z$.
As a standard consequence,  we record the following elementary fact, which will be used frequently throughout the paper.
Let $f_1(u), f_2(u)\in Z_{\lambda,V}[u]$ be  monic polynomials. Assume that $\bar f_1(u)$ and   $\bar f_2(u)$ are coprime. Then, the ideal generated by $f_1(u)$ and $f_2(u)$ is $Z_{\lambda, V}[u]$. %后面用到的f_2的形式不一定属于\Bbbk[u]
Thus, 
there exist polynomials $g(u),h(u)\in Z_{\lambda,V}[u]$ such that
\begin{equation}\label{key-coprim}
g(u)f_1(u)+h(u)f_2(u)=1.
\end{equation}

\begin{Cor}\label{lift-n}
Let $\lambda\in X^\imath$, let $V$ be a finitely generated
object of $\mathcal M_\lambda$, and let $i\in\Bbbk$. There exist
unique monic lifts
\(
n_{V,i}(u),t_{V,i}(u)\in Z_{\lambda,V}[u]
\)
of $\overline n_{V,i}(u)$ and $(u-i)^{\phi_i(V)}$, respectively,
such that $n_{V,i}(u)$ and $t_{V,i}(u)$ are coprime and
\(
n_V(u)=n_{V,i}(u)t_{V,i}(u).
\)
\end{Cor}

\begin{proof}
By Lemma~\ref{zlam-loc},  $Z_{\lambda,V}$ is a finite-dimensional local
$\Bbbk$-algebra whose maximal ideal is nilpotent. Hence the result is
a special case of Lemma~\ref{lem:moniclift}.\end{proof}

\subsection{Natural transformations}

\begin{Lemma}\label{tvi}
Fix \(i\in I\). Define \(\eta:E_i\Rightarrow E_i\) which is determined as follows. For every
finitely generated object \(V\in\mathcal M_\lambda\), let
\(\eta_V:E_iV\to E_iV\) be given diagrammatically by
\[
\eta_V
=
\begin{tikzpicture}[baseline=-1mm, scale=0.8, transform shape]
  \draw[-] (0.08,-0.3) -- (0.08,0.4);
  \node at (0.08,-0.47) {\scriptsize \(i\)};
  \node at (0.08,0.15) {\scriptsize \(\bullet\)};
  \node at (-0.55,0.15) {\scriptsize \(\dfrac{m_{V,i}(x)}{n_{V,i}(x)}\)};
\end{tikzpicture},
\]
where \(m_{V,i}(u)\) is defined as in \eqref{mini-1}, and
\(n_{V,i}(u)\) is defined as in Corollary~\ref{lift-n}. Then \(\eta\) is a
natural transformation.
\end{Lemma}

\begin{proof}
For a finitely generated object \(V\in\mathcal M_\lambda\), the object \(E_iV\)
is the generalized \(i\)-eigenspace of \(EV\).
By the definitions of $m_{V, i}(u)$ and $n_{V, i}(u)$, the reductions of these polynomials modulo the maximal ideal of $Z_{\lambda, V}$ are not divisible by $u-i$. 
Hence 
\(m_{V,i}(x)\) and \(n_{V,i}(x)\) act invertibly on \(E_iV\), and   the
expression defining \(\eta_V\) is well defined.

By Lemma~\ref{ibiad}, the functor \(E_i\) is sweet. Therefore, by
Remark~\ref{nateq}, it suffices to verify the naturality on finitely generated
objects. Let \(V,W\in\mathcal M_\lambda\) be finitely generated objects, and let
\(h:V\to W\) be a morphism. We  show that the family  $(\eta_V)_V$  is natural on finitely generated objects. 
%and $\eta_W$ commutes with
%\(E_i(h)\).

Choose \(s\) large enough so that
$
J_{\lambda,V}^{s}=0$ and 
$
J_{\lambda,W}^{s}=0$.
Set $
Z_{\lambda,s}:=Z_\lambda/J_\lambda^s$,
and let
$
\phi_s:Z_\lambda\to Z_{\lambda,s}
$
be the canonical quotient map. Since
\[
\ker\phi_s\subseteq\ker\phi_V
\qquad\text{and}\qquad
\ker\phi_s\subseteq\ker\phi_W,
\]
the maps \(\phi_V\) and \(\phi_W\) factor through \(Z_{\lambda,s}\). Thus there
are canonical surjections
$
\widetilde\phi_V:Z_{\lambda,s}\twoheadrightarrow Z_{\lambda,V}$,
and $
\widetilde\phi_W:Z_{\lambda,s}\twoheadrightarrow Z_{\lambda,W}$,
such that
$
\phi_V=\widetilde\phi_V\circ\phi_s,
$ and 
$\phi_W=\widetilde\phi_W\circ\phi_s$.

Set $
p(u):=m_V(u)m_W(u)$
and define
\begin{equation}\label{pq-natr}
q(u):=\big[\clock(u)p(u)\big]_{\geq 0},
\qquad
q_V(u):=\mathcal O_V(u)p(u).
\end{equation}
Similarly, we have $q_W(u)$.
This choice of \(p(u)\) allows us to use a common numerator for both \(V\) and
\(W\). By Lemma~\ref{lem:bublleyimeszeropoly}, \(q_V(u)\) and $q_W(u)$ are polynomials and
$\phi_V(q(u))=q_V(u)$, and $\phi_W(q(u))=q_W(u)$. 
Write the image of \(\phi_s(q(u))\) in the residue field of \(Z_{\lambda,s}\)
as
\[
\overline{\phi_s(q(u))}
=
\prod_{j\in\Bbbk}(u-j)^{\phi'_j}.
\]
By Lemma~\ref{lem:moniclift}, there exist unique monic factors
\(b_i(u),q_i(u)\in Z_{\lambda,s}[u]\) such that
$
\phi_s(q(u))=b_i(u)q_i(u)$,
with
\[
\overline{b_i}(u)=(u-i)^{\phi'_i},
\qquad
\overline{q_i}(u)=\prod_{j\neq i}(u-j)^{\phi'_j}.
\]
After applying \(\widetilde\phi_V\), we obtain
$
q_V(u)=b_{V,i}(u)q_{V,i}(u)$,
where
\[
\overline{b_{V,i}}(u)=(u-i)^{\phi'_i},
\qquad
\overline{q_{V,i}}(u)=\prod_{j\neq i}(u-j)^{\phi'_j}.
\]
By the uniqueness of the lifted factorization,
$
\widetilde\phi_V(q_i(u))=q_{V,i}(u)
$.
Choose \(\widehat q_i(u)\in Z_\lambda[u]\) such  that
$
\phi_s(\widehat q_i(u))=q_i(u)$.
Then
\[
\phi_V(\widehat q_i(u))
=
\widetilde\phi_V(\phi_s(\widehat q_i(u)))
=
\widetilde\phi_V(q_i(u))
=
q_{V,i}(u).
\]
Since
$
q_V(u)
=
\mathcal O_V(u)p(u)
=
n_V(u)m_W(u)$, by the uniqueness of the lifted factorization, 
removing the \(i\)-part gives
$
q_{V,i}(u)=n_{V,i}(u)m_{W,i}(u)$.
Since  $
p_i(u):=m_{V,i}(u)m_{W,i}(u)$, when evaluated on  $V$, we have 
\begin{equation}\label{nat1234}
. \end{equation}
All denominators are invertible on \(E_iV\), since their reductions are
products of factors \(x-j\) with \(j\neq i\).

Applying the same  argument to \(W\), we also have
\begin{equation}\label{nat1235}
%
.
%\qquad\text{on }E_iW.
\end{equation}
Moreover, \(\widehat q_i(x)\) is invertible on the generalized
\(i\)-eigenspaces, and its inverse is represented there by a polynomial in
\(x\). Therefore \(p_i(x)\widehat q_i(x)^{-1}\) is a polynomial in the dot
operator. By the naturality of the dot action, this polynomial commutes with
\(E_i(h)\). Hence \eqref{nat1234} and \eqref{nat1235} imply that
\[
\left(%
\right).
\]
Thus the family $(\eta_V)_V $ is natural on finitely generated objects. Since $E_i$ preserves filtered colimits, it extends uniquely to a natural transformation $\eta$.  
\end{proof}

In summary, the preceding results show that  given an affine Brauer action on a locally Schurian category whose dot has half-integral spectrum, the
generalized eigenspace functors carry a canonical weight decomposition; all rational normalizing factors used
in Sections 5–6 are well-defined and natural on the corresponding summands.
%\newpage

\section{From affine Brauer actions to $\imath$-Kac--Moody $2$-representations, I}\label{modue5}
 In this section, we explain how a module category $\mathcal M$ over the affine Brauer category gives rise to a generalized nilpotent $2$-representation of $\mathfrak U_+$. 
 This is the first step in constructing the desired generalized nilpotent $2$-representation of    $\mathfrak U_+^\imath$. The main task is to construct natural transformations on $\mathcal M$ corresponding to the generating $2$-morphisms of $\mathfrak U_+$, and to verify the relations among them. 
 
We begin by studying the local bubble relations. As before, 
 $\mathcal M$ is assumed to be  a locally Schurian category.

 \subsection{The local bubble relations} Throughout this subsection,  $V$ denotes a  finitely generated object of $ \mathcal M_\lambda$, where  $\lambda\in X_\imath$.  The following result  will be used frequently later on.
 %when we verify that a module category $\mathcal M$ over an affine Brauer category is a  generalized nilpotent  $2$-representation of  $\mathfrak U_+$.  

\begin{Lemma} \label{ii21} Let    $V$ be a finitely generated object in $\mathcal M_\lambda$, where $\lambda\in X_\imath$. For any $i, j\in I^\imath$, the following relations hold:
%等式左边线上的黑点都去掉了 等式右边的 竖线V移到了最右边 因为n_{V,i}(u)放在竖线V前面才有意义
 \begin{itemize} \item [(1)]  \label{ii1}    
  $\frac{n_{E_{i}V,i}(u)}{m_{E_{i}V,i}(u)}
 , 
        & \text{ if } i=j-1.
  \end{cases}$   
  \end{itemize}
\end{Lemma}

\begin{proof}
By definition, $m_{V,i}(u)$ is obtained from $m_V(u)$  by removing
the factor corresponding to the eigenvalue $i$; the same convention applies to 
$n_{V,i}(u)$.
Hence, the resulting formulas are determined by
the relations among the relevant eigenvalues.

In the present paper, we assume that
$I=\frac12+\mathbb Z$. Thus, we must also consider the exceptional cases  
$i=\pm\frac12$,
since the factors $(u-i)^2-1$ and $(u+i)^2-1$
may  contain the factor  $u-i$.
Then,  statements (1) and (2) follow from Lemma~\ref{ou}(2) by checking which of the  factors  
$\{(u-i)^2-1,\; (u+i)^2-1,\; (u-i)^2,\; (u+i)^2\}$
contain the factor $u-i$. Statement (3) follows in the same way by checking which factors contain   $u-j$. Statement (4) follows similarly by checking which of the factors  
$\{(u-j)^2-1,\; (u+j)^2-1,\; (u-j)^2,\; (u+j)^2\}$  
contain the factor $u-i$.
\end{proof}

\begin{rem} In the integral case, 
when $\operatorname{char} \Bbbk=0$
and $J=\mathbb Z$,
the value $\frac12$ does not occur.
Since $\operatorname{char}(\Bbbk)\neq 2$,
we have $i=-i$ if and only if $i=0$.
Hence, the only exceptional integral case is $i=0$,
and the formulas in Lemma~\ref{ii21}
remain valid for all $i\neq 0$.

The case of characteristic $p$ is analogous:
the validity of these formulas is determined by
the relations among the relevant eigenvalues.
This indicates why the remaining  cases can be treated by the same local analysis, after
the corresponding exceptional eigenvalues are identified.
\end{rem}

\begin{Defn}\label{scalar-lambda1} For any  $i\in I$ and any $\lambda\in X_{\imath}^+$ with $\mathcal M_\lambda\neq 0$,  define $\operatorname{def}(\lambda)_i=\phi_i(V)-\epsilon_i(V)$, where  $V$ is any non-zero finitely generated object of   $\mathcal M_\lambda$. 
\end{Defn} 
We verify below that
this definition is
independent of the
choice of $V$.
By \eqref{eat}, for any \(j\in I^\imath\), any simple object
\(L\in\mathcal M_\lambda\), and any finitely generated object
\(V\in\mathcal M_\lambda\), we have
\[
\epsilon_{j}(L)-\epsilon_{-j}(L)
+\delta_{-j,\frac{1}{2}(-1)^{a_L}}
=
\phi_{-j}(V)-\epsilon_{-j}(V).
\]
Then \eqref{equ:phiminusepslon} implies that the scalar
\(\phi_i(V)-\epsilon_i(V)\) depends only on \(\lambda\) and \(i\), and not on
the chosen  finitely generated object  \(V\in\mathcal M_\lambda\). Therefore,
\(\operatorname{def}(\lambda)_i\) is well-defined.

\begin{Lemma}\label{lem:bubbler}
Let $V\in \mathcal M_\lambda$ be a finitely generated object.
Let $i\in I$ and $r\in \mathbb Z_{\ge 0}$.
When evaluated    on $V$, the following relation holds: 
\begin{equation}\label{lem:bubbler111}

=
\begin{cases}
0 & \text{if } r<-\operatorname{def}(\lambda)_i-1,\\
1 & \text{if } r=-\operatorname{def}(\lambda)_i-1.
\end{cases}
\end{equation}
\end{Lemma}

\begin{proof}
Since $Z_{\lambda,V}$ is a finite-dimensional local
$\Bbbk$-algebra, applying \eqref{key-coprim} to
$$f_1(u)=n_{V,i}(u) \text{ and $f_2(u)=(u-i)^k$}$$
for $k\gg0$, we obtain
$g(u),h(u)\in Z_{\lambda,V}[u]$
satisfying \eqref{key-coprim}.
Moreover, for a strand labeled by $j\in I$,
$
%
$
is  zero if $j=i$, and is invertible if $j\neq i$.
Hence, 
\begin{equation}\label{kkk1}
\begin{aligned}
\operatorname{LHS\ of\ \eqref{lem:bubbler111}}
\overset{\eqref{key-coprim}}{=}\ &
%
 \\ 
\overset{\text{Lem.~\ref{jjj}}}{=}\ &
\left[
\mathcal O_V(u)\,g(u)\,m_{V,i}(u)\,u^{r+1}
\right]_{u^0} \\
{=}\ &
\left[
\frac{t_{V,i}(u)\bigl(1-h(u)(u-i)^k\bigr)}
     {(u-i)^{\epsilon_i(V)}}
\right]_{u^{-r-1}}\\
= &
\left[
\frac{t_{V,i}(u)}
     {(u-i)^{\epsilon_i(V)}}
\right]_{u^{-r-1}}.
\end{aligned}
\end{equation}
Here, the fourth equality follows from Lemma~\ref{lem:bublleyimeszeropoly} and \eqref{key-coprim}. Since $t_{V,i}(u)$ is the monic lift of
$(u-i)^{\phi_i(V)}$ and has degree $\phi_i(V)$,
the last expression gives $0$ if $r< -\operatorname{def}(\lambda)_i-1$, and gives $1$ if 
$r= -\operatorname{def}(\lambda)_i-1$. This proves the lemma.\end{proof}

%The following result is proved by the same argument as Lemma~\ref{lem:bubbler}; we include a brief sketch.

\begin{Lemma}
\label{lem:bubbler-inv}
  Suppose that $i\in I^\imath$
  and let
  $r\in \mathbb Z_{\ge0}$.  When evaluated  on a finitely generated object
  $V\in \mathcal M_\lambda$, the following relation holds: 
 
    \begin{equation}\label{lem:bubbler-inv111}    \begin{tikzpicture}[baseline = 0, scale=0.7, ]
\draw[-] (0.5,0.5) to[out=-90, in=0] (0,0);
\draw[-] (0,0) to[out = 180, in = -90] (-0.5,0.5);
\draw[-] (0.5,0.5) to[out=90, in=0] (0,1);
\draw[-] (0,1) to[out = 180, in = 90] (-0.5,0.5);
\node at (-0.68,0.55) {$\scriptstyle{i}$};
\node at (-0.4,0.8) {$\scriptstyle\bullet$};
\node at (-1.5,1.2) {$\scriptstyle{\frac{n_{E_{-i}V,i}(x)}{xm_{E_{-i}V,i}(x)}x^{r}}$};
 %\draw[-,darkg,thick] (.8,0) to (.8,.9);
%\node at (.8,-.5) {$\darkg\scriptstyle{V}$};
\end{tikzpicture}
=\begin{cases}
    0 &\text{ if } r<\operatorname{def}(\lambda)_i-2^{\delta_{i, \frac{1}{2}}},\\
  2^{\delta_{i, \frac{1}{2}}} &\text{ if } r=\operatorname{def}(\lambda)_i-2^{\delta_{i, \frac{1}{2}}} . 
\end{cases} \end{equation} 
\end{Lemma}

\begin{proof}
We take
$$
 f_1(u)=m_{V,i}(u)(2u+1)(2u-1)^{1-\delta_{i,\frac12}} \text{ and $
 f_2(u)=t_{V,i}(u)$}$$  in \eqref{key-coprim}.
%$This is valid since these two polynomials are coprime when
%$i\neq -\frac12$. 
Then  there exist polynomials
$g(u),h(u)\in Z_{\lambda,V}[u]$ satisfying \eqref{key-coprim}. 
Recall that, on the summand labelled by $j$, the endomorphism given by
$
\begin{tikzpicture}[baseline = -1mm, scale=0.7]
	\draw[-] (0.08,-.3) to (0.08,.4);
      \node at (0.08,0.15) {$\scriptstyle\bullet$};
      \node at (0.08,-0.5) {$\scriptstyle{j}$};
      \node at (-0.78,0.15) {$\scriptstyle{t_{V,i}(-x)}$};
\end{tikzpicture}
\begin{tikzpicture}[baseline = 9pt,scale=0.5,color=\clr,inner sep=0pt, minimum width=11pt]
\draw[-,darkg,thick] (0.7,0.3) to (0.7,1.4);
\node at (.7,-.0) {$\darkg\scriptstyle{V}$};
\end{tikzpicture}
$
is invertible if $j\neq -i$, and is zero if $j=-i$.
Therefore, when evaluated  on
$V$, the following relation holds:
\begin{equation} \label{lll1} \begin{aligned} \operatorname{LHS}   \text{ of \eqref{lem:bubbler-inv111}}
    =&
 \begin{tikzpicture}[baseline = 1mm, scale=0.7, ]
\draw[-] (0.5,0.5) to[out=-90, in=0] (0,0);
\draw[-] (0,0) to[out = 180, in = -90] (-0.5,0.5);
\draw[-] (0.5,0.5) to[out=90, in=0] (0,1);
\draw[-] (0,1) to[out = 180, in = 90] (-0.5,0.5);
\node at (-0.65,0.55) {$\scriptstyle{i}$};
\node at (-0.4,0.8) {$\scriptstyle\bullet$};
\node at (-0.5,1.3){$\scriptstyle{-4g(x)n_{V,i}(x)x^{r+1}}$};
%\draw[-,darkg,thick] (.9,0) to (0.9,.8);
%\node at (.9,-.5) {$\darkg\scriptstyle{V}$};
\end{tikzpicture}
=
 \begin{tikzpicture}[baseline = 1mm, scale=0.7, ]
\draw[-] (0.5,0.5) to[out=-90, in=0] (0,0);
\draw[-] (0,0) to[out = 180, in = -90] (-0.5,0.5);
\draw[-] (0.5,0.5) to[out=90, in=0] (0,1);
\draw[-] (0,1) to[out = 180, in = 90] (-0.5,0.5);
%\node at (0.6,0.55) {$\scriptstyle{i}$};
%\node at (-0.6,0.55) {$\scriptstyle{i}$};
\node at (-0.4,0.8) {$\scriptstyle\bullet$};
\node at (-0.5,1.5){$\scriptstyle{-4g(x)n_{V,i}(x)x^{r+1}}$};
%\draw[-,darkg,thick] (.9,0) to (.9,.8);
%\node at (0.9,-.5) {$\darkg\scriptstyle{V}$};
\end{tikzpicture}\\
 = & -4\left[
\mathcal O_V(-u)g(u)n_{V,i}(u)u^{r+1}
\right]_{u^0} \\ 
= & \left[(4u^2-1)\frac{m_{V}(u)}{n_V(u)}g(u)n_{V,i}(u)
\right]_{u^{-r-1}}\\   
\overset{\eqref{key-coprim}}=&
\left[
\frac{ (2u-1)^{\delta_{i, \frac{1}{2}}}(u-i)^{\epsilon_i(V)}(1-h(u)t_{V,i}(u))}{t_{V,i}(u)} 
\right]_{u^{-r-1}}\\ 
=&
\left[ \frac{ (2u-1)^{\delta_{i, \frac{1}{2}}}(u-i)^{\epsilon_i(V)}}{t_{V,i}(u)}\right]_{u^{-r-1}}.
 \end{aligned}\end{equation}
We now explain how to obtain the above equalities. First, Lemma~\ref{ii21}(2) gives the relation between % 此时(a_{V,i}还没引入
 $\frac{n_{E_{-i}V,i}(x)}  {xm_{{E_{-i}V,i}(x)}}  \begin{tikzpicture}[baseline=-0.5ex,scale=0.5]
   \node at (0.9, -0.7) {$\scriptstyle{i}$};
   \draw[thick, black, ] (0.9,-0.5) -- (0.9,0.5);
        \node at (0.9, 0) {$\scriptstyle\bullet$}; 
        \end{tikzpicture}   $ and $\frac{n_{V,i}(x)}  {xm_{{V,i}(x)}} \begin{tikzpicture}[baseline=-0.5ex,scale=0.5]
   \node at (0.9, -0.7) {$\scriptstyle{i}$};
   \draw[thick, black, ] (0.9,-0.5) -- (0.9,0.5);
        \node at (0.9, 0) {$\scriptstyle\bullet$}; 
        \end{tikzpicture}   $. 
Then, using Lemma~\ref{AB relations-re}(1), we move the dots from the right-hand side to the left-hand side of the bubble. Together with \eqref{key-coprim}, this gives the first equality above.

Next, we again use Lemma~\ref{AB relations-re}(1) to move the bullets from the left-hand side to the right-hand side of the bubble, and then apply Lemma~\ref{jjj} to obtain the third equality. The fourth equality follows from Lemma~\ref{lem:bublleyimeszeropoly} and Lemma~\ref{ou}(1).

Finally, since \(t_{V,i}(u)\) is the monic lift of \((u-i)^{\phi_i(V)}\) and
\(\deg t_{V,i}(u)=\phi_i(V)\), the last expression in \eqref{lll1} gives 
$0$ if $r< \operatorname{def}(\lambda)_i-2^{\delta_{i, \frac{1}{2}}}$, and gives 
$2^{\delta_{i, \frac{1}{2}}}$  if $ r=\operatorname{def}(\lambda)_i \new{-}2^{\delta_{i, \frac{1}{2}}}$. 
This proves the lemma.
\end{proof}

\subsection{Generating natural transformations}
We are now in a position to define natural transformations corresponding to the generating 2-morphisms of $\mathfrak U$.
In contrast to the usual isomorphism-based approaches, our method proceeds by a direct construction, which is necessitated by the presence of inhomogeneous relations in $\mathfrak U^\imath$.

To simplify notation, we introduce the following  rational functions in the variables $x_1, x_2, x_3$ and $x$. 
This notation will be used consistently from this section through Section~7.
\begin{Defn}\label{ration111}  Let $a,b\in \mathbb{Z}_{>0}$, $\sigma,\tau\in\{+,-\}$, and $i\in I^\imath$. We set 
\begin{multicols}{2} \item [(1)]
$\ell^{ab}_{\sigma\tau}:=1+\sigma x_a+\tau x_b$, \item [(2)] $d^{ab}_{\sigma\tau}:=\sigma x_a+\tau x_b$.
 \item [(3)]  $a_{V,i}(x):=\frac{n_{V,i}(x)}  {xm_{{V,i}(x)}}$, 
 \item [(4)] $b_{V,i}(x):=\frac{x m_{V,i}(x)}{n_{V,i}(x)}$.\end{multicols}
%where $x$ is a variable. 
\end{Defn}
Here the superscripts $a,b$ indicate the variables involved, while the subscripts $\sigma,\tau$ record the signs of $x_a$ and $x_b$, respectively. We will use this  notation in diagrams and write
\(
(\ell^{ab}_{\sigma\tau})^{-1}
\)
and
\(
(d^{ab}_{\sigma\tau})^{-1}
\)
for the corresponding  inverse operators on the relevant generalized eigenspaces. For example,
$
\ell^{12}_{+-}=1+x_1-x_2$ and $
d^{13}_{-+}=-x_1+x_3$.

\begin{Defn}\label{natrual-basic}
Let   $i\in I^\imath$, and let $V$ be   a finitely generated object of $ \mathcal M_\lambda$ with  $\lambda\in X_\imath^+$. 
The following formulas define the components at $V$ of the corresponding natural
transformations:
%We define the following natural transformations  on $V$:
\begin{itemize}
\item[(1)]
    $
$.
% where 
 % $g(x)=\frac{n_{E_{-i}V, i}(x)}%{xm_{E_{-i}V, i}(x)}$.
%and 
%$h(x)=\frac{xm_{V, i}(x)}{n_{V, i}%(x)}$.
\end{itemize}
\end{Defn}
Lemma~\ref{tvi} ensures that the natural transformations in Definition~\ref{natrual-basic}(3) are well-defined.
\begin{Lemma}
\label{caup3}
Suppose  $i\in I^\imath$.  When evaluated on   the  finitely generated object $V\in \mathcal M_\lambda$, the following relations hold:  
   \begin{multicols}{2} \item [(1)]  
       $%
 & \text{ if } i=\frac{1}{2}.
\end{cases}
$
   \end{multicols} 
\end{Lemma}
\begin{proof}We use Definition~\ref{natrual-basic}(3) and  Lemma~\ref{ii21}(1)--(2)  to establish the relation  between $a_{E_{-i}V, i}(x) %
 $, and similarly the relations between   $b_{V,i}(x) %
 $. Then,  we  apply Lemma~\ref{AB relations-re}(1)--(2) to obtain the required formulas.
\end{proof}
 
\begin{Defn} \label{natrual-basic1} Let   $i, j\in I^\imath$, and let $V$ be   a finitely generated object of $ \mathcal M_\lambda$ with  $\lambda\in  X_\imath^+$. We define the following natural transformations   on $V$:
\begin{itemize}
\item [(1)] 
%\begin{align*}
$
%
}$,  where $t_{ji}$ is defined in~\eqref{tij}.
\item [(3)] $\mathord{
%
. 
\end{itemize} 
\end{Defn}

It follows from \eqref{kkk1} (respectively, \eqref{lll1}) that the first (respectively, second)  relation in (4) is well-defined.
In particular, the two relations in (4)  define the corresponding  fake bubbles
$
%
$ for all $r<0$ in the indicated ranges of $\operatorname{def}(\lambda)_i$.

The constructions in Definitions~\ref{natrual-basic} and \ref{natrual-basic1} are the key step in passing from an
affine Brauer action to an \(\imath\)-Kac--Moody $2$-representation.  Once the
red natural transformations have been defined, it remains to check that
they satisfy the defining relations of the \(\imath\)-Kac--Moody $2$-category.
These verifications will be carried out from this point to the end of Section~6.

\subsection{$2$-representations of  $\mathfrak U_+$}
\begin{Lemma}\label{adj}(Adjunction) Suppose that $i\in I^\imath$.  When evaluated on any finitely
generated object \(V\in\mathcal M_\lambda\), the natural transformations
defined in Definitions~\ref{natrual-basic} and~\ref{natrual-basic1}
satisfy  Definition~\ref{def-ikmc}(1a)–(1b). \end{Lemma}

\begin{proof}
The first identity in each of Definition~\ref{def-ikmc}(1a) and Definition~\ref{def-ikmc}(1b) follows from Definition~\ref{C defn}(1) together with Definition~\ref{natrual-basic}(1)--(3). By the interchange law in the endomorphism category,  we have  
%这个等式没有，是在图上才有
\[\mathord{
}.\]
Evaluating at $V$ and using Lemma~\ref{AB relations-re}(1)--(2) and Definition~\ref{C defn}(1), we obtain:
$$\mathord{
%
.
}$$
This proves the second relation in Definition~\ref{def-ikmc}(1a). The second relation in Definition~\ref{def-ikmc}(1b) is  proved in the same way, so we omit the details.
\end{proof}

{For $i, j\in I^\imath$}, we define the following  rational functions in the  variables $x_1$ and $ x_2$: 
\begin{equation}\label{sx1x2} s(x_1, x_2)=\begin{cases}h(x_1,x_2) & \text{ if $ j=i$,}
\\\frac{(x_2-x_1)^2}{((1-(x_2-x_1)^2)(x_1+x_2)}  &\text{ if $ j=i+1$,}
\\\frac{(x_2-x_1)^2(1+x_1+x_2)}{(1-(x_2-x_1)^2)(x_1+x_2)}   &\text{ otherwise,}
\end{cases}\ \ t(x_1,x_2)=\begin{cases}\frac{1}{1+x_1+x_2} & \text{ if $ i=j$, }
\\x_1+x_2  &\text{ if $ i=j+1$,}
\\\frac{x_1+x_2}{1-x_1-x_2}  &\text{ otherwise,}
\end{cases}
\end{equation}
where \begin{equation}\label{h12}
h(x_1,x_2)=
      \frac{(x_1+x_2-1)(x_1-x_2)^2}{(x_2-x_1+1) (x_2-x_1-1)^{1-\delta_{i,1/2}}}.\end{equation}
      
\begin{Lemma}\label{lrcrossing}Suppose that \(i,j\in I^\imath\). Then, when evaluated on any finitely
generated object \(V\in\mathcal M_\lambda\), the natural transformations
defined in Definitions~\ref{natrual-basic} and~\ref{natrual-basic1}
satisfy the following relations:

	 \begin{itemize}\item[(1)]    $

=
\left\{
\begin{array}{ll}
\begin{tikzpicture}[baseline = 0, scale=0.7,  ]
	\draw[-] (0.38,-.4) to (-0.38,.5);
	\draw[-] (-0.38,-.4) to (0.38,.5);
      \node at (0,0.05) {$\diamond$};
	\draw[-,densely dotted] (0.2,-.24) to (-0.19,-.24);
   \node at (-.4,-.6) {$\scriptstyle{i}$};
   \node at (.4,-.6) {$\scriptstyle{-i}$};
   \node at (-.4,.65) {$\scriptstyle{-i}$};
   \node at (.4,.7) {$\scriptstyle{i}$};
     % \node at (-0.26,-0.25) {$\dt$};
     % \node at (0.26,-0.25) {$\dt$};
      \node at (1.3,-0.2) {$\scriptstyle{(\ell^{12}_{--})^{-1}}$};
      % \draw[-,darkg,thick] (2.3,-0.5) to (2.3,.5);
%\node at (2.3,-.7) {$\darkg\scriptstyle{V}$};
\end{tikzpicture}=
\begin{tikzpicture}[baseline = 0, scale=0.7, ]
	\draw[-] (0.38,-.4) to (-0.38,.5);
	\draw[-] (-0.38,-.4) to (0.38,.5);
      \node at (0,0.05) {$\diamond$};
	\draw[-,densely dotted] (0.2,.24) to (-0.19,.24);
   \node at (-.4,-.6) {$\scriptstyle{i}$};
   \node at (.4,-.55) {$\scriptstyle{-i}$};
   \node at (-.4,.65) {$\scriptstyle{-i}$};
   \node at (.4,.7) {$\scriptstyle{i}$};
     % \node at (-0.26,0.25) {$\dt$};
      % \node at (0.26,0.25) {$\dt$};
      \node at (1.3,0.2) {$\scriptstyle{(\ell^{12}_{--})^{-1}}$};
     %  \draw[-,darkg,thick] (2.3,-0.5) to (2.3,.5);
%\node at (2.3,-.7) {$\darkg\scriptstyle{V}$};
\end{tikzpicture}
&\text{if $j=i$,}\\
\end{array}%
\right.$\end{itemize}
where  $c_V(x)=\frac{b_{V,j}(-x)(2x-1)(-1-2x)^{1-\delta_{j, 1/2}}}{4x^2}$, 
% c_V(x)要用\frac{b_{V,j}(-x)(2x-1)(-1-2x)^{1-\delta_{j, 1/2}}}{4x^2}这个形式 后面算两个crossing复合等公式的时候都要的\frac{b_{V,j}(-x)(2x-1)(-1-2x)^{1-\delta_{j, 1/2}}}{4x^2} 本来这个引理就是为后面服务的
$d_V(x)=a_{V, j}(x)$, and  $s(x_1, x_2)$ is defined as in \eqref{sx1x2}. \end{Lemma}

%\comment{Song: In the following, the another form of the notation is as follows:
%\[ a(x)=\frac{n_{E_{-j}V,j}(x)}{xm_{E_{-j}V,j}(x)}, \quad b(x)=\frac{x m_{V,j}(x)}{n_{V,j}(x)}\] and 
%\[ c_V(x)=a(-x).\]}

\begin{proof} We use Lemma~\ref{dotslidecrossing} to move the bullets from the bottom to the  top, and obtain 
the second equality for each of the three cases in (1).
For the  first equality in each case, we use  
Definition~\ref{natrual-basic}(2), Definition~\ref{natrual-basic1}(1), and
the first relation in Definition~\ref{natrual-basic1}(3) to convert the red diagrams into the black diagrams,  and apply Lemma~\ref{dotslidecrossing} to move the bullets appropriately. 
 This proves (1).

Evaluating at  $V$ and  using Definition~\ref{natrual-basic1}(1),  Lemma \ref{AB relations-re}, and  Definition~\ref{C defn}(3), we have  
\begin{equation}\label{lcross}
,\end{equation}
where $t(x_1, x_2)$ is defined as in \eqref{sx1x2}. 
Using  Definition \ref{natrual-basic1}(2),  the second relation in  Definition \ref{natrual-basic1}(3),  equation \eqref{lcross}, and Lemma~\ref{adj}, we obtain:  \begin{equation}
  \label{redijleft}  
\begin{aligned}
   %
,
\end{aligned}
\end{equation}
where   $$b(x)=\begin{cases} \frac{1-4x^2}{4x^2}b_{E_iV,j}(x) & \text{if $j\neq 1/2$,}\\
\frac{(1+2x)}{-4x^2}  b_{E_iV, j}(x)  &\text{if $j=1/2$.}\end{cases} $$
To obtain the third equality, we use Lemma~\ref{caup3} to convert the red cup (respectively,  cap) into the black cup (respectively, cap). We explain how to obtain the last equality in~\eqref{redijleft},  and prove the formula in each case of (2).

If $i\neq j$, then Lemma~\ref{ii21}(3) gives the last equality in
\eqref{redijleft}. Together with \eqref{redijleft} and
Lemma~\ref{dotslidecrossing}, this yields the desired formula in each
case of (2) under the assumption $i\neq j$.

If $i=j$, then the last equality in \eqref{redijleft} follows from
Lemma~\ref{ii21}(1). This proves the first equality in (2) in this case.
It remains to prove the second equality in (2) when $i=j$. We prove this
under the additional assumption that $i=\frac{1}{2}$; the case
$i\neq \frac{1}{2}$ is analogous. By Lemma~\ref{dotslidecrossing} and a
straightforward computation, we obtain
$$ \begin{aligned}

\end{aligned}
$$ as required.
Note that the second equality follows from  the identities:  $$c_{V}(x) =
      \frac{(1-2x) }{4x}\frac{m_{V,1/2}(-x)}{n_{V,1/2}(-x)},\quad  d_V(x)=\frac{n_{V,1/2}(x)}{xm_{V,1/2}(x)}. $$
    This completes the proof of (2). 
\end{proof}

\begin{Lemma}\label{cyc-x}(Cyclicity of $x$ and $\tau$) The natural transformations defined in Definitions~\ref{natrual-basic} and ~\ref{natrual-basic1} satisfy the relations in Definition~\ref{def-ikmc}(2a)-(2b). 
\end{Lemma}
\begin{proof} All relations in Definition~\ref{def-ikmc}(2a) follow immediately  from Corollary~\ref{caup}(1) and Lemma~\ref{adj}. 
Suppose  $i, j\in I^\imath$.
Acting on $V$, we use  Definition~\ref{natrual-basic}(2) and Definition~\ref{natrual-basic1}(1)
to  convert the red diagram into a  black diagram, and  obtain 
\begin{equation}\label{downarrow} t_{ij}^{-1}\mathord{

}=\left\{
\begin{array}{ll}
\begin{tikzpicture}[baseline = 0, scale=0.8,  ]
	\draw[-] (0.38,-.4) to (-0.38,.5);
	\draw[-] (-0.38,-.4) to (0.38,.5);
      \node at (0,0.05) {$\diamond$};
	\draw[-,densely dotted] (0.2,.24) to (-0.19,.24);
   \node at (-.4,-.55) {$\scriptstyle{-i}$};
   \node at (.4,-.55) {$\scriptstyle{-i}$};
   \node at (-.4,.65) {$\scriptstyle{-i}$};
   \node at (.4,.65) {$\scriptstyle{-i}$};
      \node at (1.3,-0.2) {$\scriptstyle{(\ell_{+-}^{12})^{-1}}$};
      % \draw[-,darkg,thick] (2.3,-0.5) to (2.3,.5);
%\node at (2.3,-.7) {$\darkg\scriptstyle{V}$};
\end{tikzpicture}
+\begin{tikzpicture}[baseline = 0mm, scale=0.8,  ]
	\draw[-,densely dotted] (0.14,.06) to (0.5,.06);
	\draw[-] (0.58,-.4) to (0.58,.5);
	\draw[-] (0.08,-.4) to (0.08,.5);
      \node at (0.08,-0.55) {$\scriptstyle{-i}$};
      \node at (0.58,-0.55) {$\scriptstyle{-i}$};
      \node at (1.6,0.1) {$\scriptstyle{ (\ell_{+-}^{12})^{-1} }   $};
     % \draw[-,darkg,thick] (2.6,-0.5) to (2.6,.5);
%\node at (2.6,-.7) {$\darkg\scriptstyle{V}$};
\end{tikzpicture}
&\text{if $j=i$,}\\
\begin{tikzpicture}[baseline = 0, scale=0.8, ]
	\draw[-] (0.38,-.4) to (-0.38,.5);
	\draw[-] (-0.38,-.4) to (0.38,.5);
	\draw[-,densely dotted] (0.2,.24) to (-0.19,.24);
    
      \node at (0,0.05) {$\diamond$};
   \node at (-.4,-.55) {$\scriptstyle{-i}$};
   \node at (.4,-.55) {$\scriptstyle{-j}$};
   \node at (-.4,.65) {$\scriptstyle{-j}$};
   \node at (.4,.65) {$\scriptstyle{-i}$};
      \node at (1.1,0.2) {$\scriptstyle{d^{12}_{-+}}$};
     % \draw[-,darkg,thick] (2,-0.5) to (2,.5);
%\node at (2.6,-.7) {$\darkg\scriptstyle{V}$};
\end{tikzpicture}
&\text{if $j=i-1$,}\\
\end{array}\right.\end{equation} where $t_{ij}$ is defined in ~\eqref{tij}.
Recall the rational function  $s(x_1, x_2)$ from  \eqref{sx1x2}, and define 
$$f_V(u)=\frac{m_{E_{-j}E_{-i}V, i}(u)}{n_{E_{-j}E_{-i}V, i}(u)}\frac{n_{E_{-i}V, i}(u)}{m_{E_{-i}V, i}(u)}.$$ Acting on the finitely generated object $V$ yields
\begin{equation}
  \label{redijdown}  
   %
 
=\text{RHS of \eqref{downarrow}}. 
\end{equation}
Here the first equality follows from Definition~\ref{natrual-basic1}(2),
the second relation in Definition~\ref{natrual-basic1}(3), and
Lemma~\ref{adj}.  To obtain 
the second equality above, we use the second relation in
Definition~\ref{natrual-basic}(3) and Lemma~\ref{lrcrossing} to convert the red diagram into a black diagram; then   
Lemma~\ref{dotslidecrossing}  moves dots appropriately. This gives the case $i\neq j$. 
For the last equality in \eqref{redijdown}, the case $i=j$ follows from
Lemma~\ref{ii21}(2) and Lemma~\ref{dotslidecrossing}, while the case
$i\neq j$ follows from Lemma~\ref{dotslidecrossing}  and Lemma~\ref{ii21}(4).
Thus, the first equality in Definition~\ref{def-ikmc}(2b) follows from
Definition~\ref{natrual-basic1}(2), and the second equality in
Definition~\ref{def-ikmc}(2b) follows immediately from
\eqref{downarrow}--\eqref{redijdown}.
\end{proof}

\begin{Lemma}~\label{qha}(Quiver Hecke relations) The natural transformations defined in Definitions~\ref{natrual-basic} and \ref{natrual-basic1} satisfy the relations in Definition~\ref{def-ikmc}(3a)-(3c). 
 \end{Lemma}
\begin{proof}  For the finitely generated object $V$, only finitely many summands  $E_iE_jV$ are nonzero. Let $V^+:= \bigoplus_{i,j\in I^\imath} E_iE_jV$ denote  this finite direct sum. 
Since $i,j>0$ for all $i,j\in I^\imath$, it follows from  Lemma \ref{dotslidecrossing} that, when evaluated  on $V^+$, 
\[ \begin{tikzpicture}[baseline = 7.5pt, scale=0.5]
			\draw[-,thick] (0,0) to[out=up, in=down] (1,2);
			\draw[-,thick] (0,2) to[out=up, in=down] (0,2.2);
			\draw[-,thick] (1,0) to[out=up, in=down] (0,2);
			\draw[-,thick] (1,2) to[out=up, in=down] (1,2.2);
			\node at (0,1.9) {$\scriptstyle\bullet$};
		\end{tikzpicture}
		~-~
		\begin{tikzpicture}[baseline = 7.5pt, scale=0.5]
			\draw[-,thick] (0,0) to[out=up, in=down] (1,2);
			\draw[-,thick] (0,0) to[out=up, in=down] (0,-0.2);
			\draw[-,thick] (1,0) to[out=up, in=down] (0,2);
			\draw[-,thick] (1,0) to[out=up, in=down] (1,-0.2);
			\node at (1,0.1){$\scriptstyle\bullet$};
		\end{tikzpicture}
		~=~		\begin{tikzpicture}[baseline = 7.5pt, scale=0.5]
			\draw[-,thick] (0,0) to[out=up, in=down] (0,2);
			\draw[-,thick] (1,0) to[out=up, in=down] (1,2);
		\end{tikzpicture}.
        \]
 This coincides with the corresponding relation in  the degenerate Heisenberg category. Moreover, 
since Definition~\ref{natrual-basic}(1) and Definition~\ref{natrual-basic1}(1) agree with  the corresponding  definitions  in \cite[\new{Theorem 4.1}]{BSW-heisenberg},  all  relations in Definition~\ref{def-ikmc}(3a)-(3c) have already been verified  in the proof of \cite[Theorem 4.11]{BSW-heisenberg}; see also \cite[Theorem 1.1]{BK-klr} for the algebraic formulation. 
\end{proof}

 \begin{Lemma}\label{bubble1} (Bubble relations) %For any $i\in I^\imath$, 
 The natural transformations defined in Definitions~\ref{natrual-basic} and ~\ref{natrual-basic1} satisfy the relations  in  Definition~\ref{def-ikmc}(4a)--(4e). 
 %\eqref{bubble111}, \eqref{bubble2} and \eqref{bubble3} 
   \end{Lemma}
\begin{proof}
By \eqref{equ:phiminusepslon}, all relations in Definition~\ref{def-ikmc}(4a)--(4d) follow from Lemmas~\ref{lem:bubbler}--\ref{lem:bubbler-inv}. 
 Moreover,  it follows from Definition~\ref{natrual-basic1}(4)  that 
    \[
  \left( \sum_{r\in \mathbb Z}
   \begin{tikzpicture}[baseline = 0]
  \draw[->,thick,darkred] (0.2,0.2) to[out=90,in=0] (0,.4);
  \draw[-,thick,darkred] (0,0.4) to[out=180,in=90] (-.2,0.2);
\draw[-,thick,darkred] (-.2,0.2) to[out=-90,in=180] (0,0);
  \draw[-,thick,darkred] (0,0) to[out=0,in=-90] (0.2,0.2);
 \node at (0,-.13) {$\scriptstyle{i}$};
   %\node at (-0.3,0.2) {$\scriptstyle{\lambda}$};
   \node at (0.2,0.2) {$\color{darkred}\scriptstyle\bullet$};
   \node at (0.4,0.2) {$\color{darkred}\scriptstyle{r}$};
     \draw[-,darkg,thick] (.6,-0.5) to (.6,.5);
 \node at (0.6,-.7) {$\darkg\scriptstyle{V}$}; 
\end{tikzpicture} 
u^{-r-1}\right) 
\left(\sum_{s \in \mathbb{Z}} 
\begin{tikzpicture}[baseline = 0]
  \draw[<-,thick,darkred] (0,0.4) to[out=180,in=90] (-.2,0.2);
  \draw[-,thick,darkred] (0.2,0.2) to[out=90,in=0] (0,.4);
 \draw[-,thick,darkred] (-.2,0.2) to[out=-90,in=180] (0,0);
  \draw[-,thick,darkred] (0,0) to[out=0,in=-90] (0.2,0.2);
 \node at (0,-.13) {$\scriptstyle{i}$};
   %\node at (0.3,0.2) {$\scriptstyle{\lambda}$};
   \node at (-0.2,0.2) {$\color{darkred}\scriptstyle\bullet$};
   \node at (-0.4,0.2) {$\color{darkred}\scriptstyle{s}$};
%\node at (0.8,0.2){$\scriptstyle{
%u^{-r-1}}$};  
\draw[-,darkg,thick] (.5,-0.5) to (.5,.5);
\node at (0.5,-.7) {$\darkg\scriptstyle{V}$}; \end{tikzpicture}
u^{-s-1}\right)= 2^{\delta_{i,\frac{1}{2}}} u^{\delta_{i,\frac{1}{2}}}. \]
Together with \eqref{scalar23}, this  implies the relation in Definition~\ref{def-ikmc}(4e).
\end{proof}

\begin{Lemma}\label{sigrel} (Nodal relations)  The natural transformations defined in Definitions~\ref{natrual-basic} and ~\ref{natrual-basic1} satisfy the relations  in  Definition~\ref{def-ikmc}(5a)--(5b).
 \end{Lemma}
\begin{proof} Suppose $i\in I^\imath$.  Using   Definition~\ref{natrual-basic}(2),
we prove  the relation  in  Definition~\ref{def-ikmc}(5a) as follows,  uniformly in the two cases     $i=1/2$ and  $i\neq 1/2$. 

We  take  $f_1(u)=n_{V, i}(u)$ and $f_2(u)=(u-i)^k$ for a sufficiently large integer $k$ in \eqref{key-coprim}. Therefore, 
there are  $g(u), h(u)\in Z_{\lambda, V}[u]$ satisfying \eqref{key-coprim}. 
Define $p(u)=um_{V,i}(u)g(u)$. When evaluated  on   $V\in \mathcal M_\lambda$, the following relations hold: 
\begin{equation}\label{capcross-red}\begin{aligned} & 
}. 
\end{aligned}
\end{equation}
                                             Here, the last equation follows from the first relation in  Definition \ref{natrual-basic1}(4) after replacing $u$ by $u-i$. 
Using~Definition~\ref{natrual-basic}(3) and Lemma~\ref{lrcrossing},
we convert the  red diagram   $
  %
$ into the black diagram on the left-hand side  of \eqref{capcross-red}.
This shows that the natural transformations satisfy the relation in 
 Definition~\ref{def-ikmc}(5a).

We take  $f_1(u)=m_{V,i}(u)b(u)$ and $ f_2(u)=t_{V, i}(u) (u-i)^k$ for a sufficiently large integer $k$, where  $b(u)=-(2u+1)(2u-1)^{1-\delta_{i,\frac12}}$. 
Then 
 there exist polynomials  $g(u), h(u)\in  Z_{\lambda,V}[u]$ satisfying \eqref{key-coprim}.
Define $p(u)=4un_{V,i}(u)g(u)$.  Acting on  $V\in \mathcal M_\lambda$, we obtain: 
\begin{equation}\label{capcross-red1}
\begin{aligned}
&  
%
}. \end{aligned}\end{equation}
Here,  the last equation follows from the second relation in Definition \ref{natrual-basic1}(4) after replacing $u$ by $u-i$.
Using~ Lemma~\ref{caup3}(1) and Lemma~\ref{lrcrossing},
we convert the  red diagram  $%
$
 into the black diagram on the left-hand side of \eqref{capcross-red1}.
 This shows that the natural transformations satisfy the relation in 
 Definition~\ref{def-ikmc}(5b),  whether or not  $i=1/2$. 
\end{proof}

\begin{Lemma} \label{bicross1 downred} (Bicross relations) %Suppose $1/2\neq i\in I^\imath$. 
 The natural transformations defined in Definitions~\ref{natrual-basic} and ~\ref{natrual-basic1} satisfy the relations  in Definition~\ref{def-ikmc}(6a)-(6d). 
\end{Lemma}
\begin{proof} {Fix $i\in I^\imath$}. First, define the  rational function
\begin{equation}\label{ration-h123}
h(x_1, x_2)=\begin{cases} \frac{(x_1-x_2)^2}{1-(x_1-x_2)^2} & \text{if $i>\frac{1}{2}$,}\\ \frac{(x_1-x_2)^2}{x_1-x_2-1} 
& \text{if $i=\frac{1}{2}$, }\end{cases}
\qquad h'(x_1,x_2)=h(x_2, x_1).\end{equation}

We first assume that  $i>\frac{1}{2}$. Recall $c_V(x)$ and $d_V(x)$ in Lemma~\ref{lrcrossing}. We take 
$f_1(u)= m_{V, i}(-u) (4u^2-1)$ and $f_2(u)=  (u+i)^kt_{V,i}(-u)  $ in \eqref{key-coprim} for a sufficiently large integer $k$. Then there exist 
 polynomials $g(u), h(u)\in  Z_{\lambda, V}[u]  $ satisfying \eqref{key-coprim}. Similarly, we take $f_1(u)=   n_{V, i}(u)$ and $f_2(u)=(u-i)^k$
for a sufficiently large integer $k$. Then there exist two further  polynomials  $g'(u), h'(u)\in  Z_{\lambda, V}[u]  $   satisfying \eqref{key-coprim}.
Applying \eqref{key-coprim} yields 
$$c_V(x)^{-1} ,$$  where $p(u)=4un_{V, i}(-u) g(u)$ and $  q(u)=um_{V,i}(u)g'(u)$. To simplify the presentation of \eqref{redd1}, introduce the following notation. 
%so as to simplify  the presentation of   .
We simply write $h(x_1, x_2)$  in
\eqref{ration-h123} as $h$.
$$\begin{aligned} A_1& =\Bigg[\mathord{
 %
\end{aligned}
\end{equation}
where the fourth equality follows from applying  Lemma~\ref{usefuel-equa}(5) to $D_1$, together with the fact that  $p(u)$ annihilates
%
when evaluated on $V$. Applying Lemma~\ref{usefuel-equa}(3) to the lower part of \(A_1\), we obtain the fifth equality. Finally, applying Lemma~\ref{usefuel-equa}(5) together with Corollary~\ref{caup}(5) to \(C_1\), we obtain the sixth equality.
Lemma~\ref{lrcrossing} gives \begin{equation} \label{lhs111} \operatorname{LHS} \text{of  \eqref{redd1}} =\mathord{
%
}.\end{equation}%More precisely, the fifth equality 
%uses Lemma~\ref{usefuel-equa}(5), while the sixth equality uses Lemma~\ref{usefuel-equa}(5) together with Corollary~\ref{caup}(5).
We  compute $F_1$ on the $\operatorname{RHS} $ of \eqref{redd1} as follows. {Choose the integer $k$ above  sufficiently large so that   $k=k_1 + k_2$ for some sufficiently large positive integers   $k_1$ and $k_2$.}  
  Therefore, for any $a(u)\in Z_{\lambda, V}[u]$, we obtain the desired conclusion in \eqref{eq1}--\eqref{eq0}:
  \tikzset{
line_noarrow_solid/.style={-, black},line_arrow_solid/.style={->, black},
    line_noarrow_dashed/.style={-, densely dotted},
    line_arrow_dashed/.style={->, densely dotted},
    cross/.style={black, -},
    arc_right/.style={-, blue, bend right=80,looseness=1.5},
    arc_left/.style={-, blue, bend left=80, ,looseness=1.5},
    mi_near_start/.style={
        decoration={
            markings,
            mark=at position 0.2 with {
                \node[cplus] at (0,0) {$-$};
            }
        },
        postaction=decorate
    },
    mi_near_end/.style={
        decoration={
            markings,
            mark=at position 0.8 with {
                \node[cplus] at (0,0) {$-$};
            }
        },
        postaction=decorate
    },
    pl_near_start/.style={
        decoration={
            markings,
            mark=at position 0.2 with {
                \node[cplus] at (0,0) {$+$};
            }
        },
        postaction=decorate
    },
    dot_near_start/.style={
        decoration={
            markings,
            mark=at position 0.25 with {
                \fill[black] (0,0) circle (1.5pt);
                \coordinate (dot-start) at (0,0);
            }
        },
        postaction=decorate
    },
    dot_near_end/.style={
        decoration={
            markings,
            mark=at position 0.75 with {
                \fill[black] (0,0) circle (1.5pt);
                \coordinate (dot-end) at (0,0);
            }
        },
        postaction=decorate
    },
    a_near_start/.style={
        decoration={
            markings,
            mark=at position 0.15 with {
                \coordinate (a-start) at (0,0);
            }
        },
        postaction=decorate
    },
    a_near_end/.style={
        decoration={
            markings,
            mark=at position 0.85 with {
                \coordinate (a-end) at (0,0);
            }
        },
        postaction=decorate
    },
    line_near_start/.style={
        decoration={
            markings,
            mark=at position 0.1 with {
                \draw[thick, black] (1.5pt,-1.5pt) -- (1.5pt,1.5pt);
                \coordinate (line-start) at (0,0);
            }
        },
        postaction=decorate
    },
    line_near_end/.style={
        decoration={
            markings,
            mark=at position 0.8 with {
                \draw[thick, black] (1.5pt,-1.5pt) -- (1.5pt,1.5pt);
                \coordinate (line-end) at (0,0);
            }
        },
        postaction=decorate
    }	
}
\begin{equation}
\label{eq1}
 \mathord{\Bigg[
, 
\text{ by Definition~\ref{natrual-basic1}(4).}\\
\end{aligned},\end{equation}
Here, the second equality follows from \eqref{key-coprim} and Lemma~\ref{ou}(1), and the third equality follows from \eqref{eq0} and Definition~\ref{C defn}(1). Combining \eqref{redd1} with the computation of $F_1$ in \eqref{sss1}, we obtain 

\begin{equation}\label{rhs111} \operatorname{RHS} \text{of \eqref{redd1}} =-
%
.\end{equation} 
Now the relation in Definition~\ref{def-ikmc}(6b) follows from \eqref{lhs111} and \eqref{rhs111}.

Next, we prove the relation in Definition~\ref{def-ikmc}(6a). Introduce the following notation: 
$$\begin{aligned}  A_2& = \left[
 %
,
}\\ 
\end{aligned}
\end{equation}
where  the fourth equality of \eqref{redd2} follows by applying Lemma~\ref{usefuel-equa}(5) to $D_2$, together with 
the fact that $q(u)$ annihilates
$%
$
when acting on $V$. The fifth and sixth equalities in \eqref{redd2} are obtained by arguments similar to those for the fifth and sixth equalities in \eqref{redd1}.  More precisely, the fifth equality uses Lemma~\ref{usefuel-equa}(3), while  the sixth equality uses Lemma~\ref{usefuel-equa}(5) and  Corollary~\ref{caup}(5). To obtain the last equality, we need the following relation:

\begin{equation}
    \label{sss11}
\Bigg[\mathord{%
,
\end{equation}
which follows from arguments similar to those for \eqref{sss1}.  The only difference is that we  use \eqref{eq1} (respectively, the first relation in Definition~\ref{natrual-basic1}(4)) to replace  
\eqref{eq0} (respectively, the second relation in Definition~\ref{natrual-basic1}(4)).  By Lemma~\ref{lrcrossing}, we have  $$
    %
= \operatorname{LHS} \text{ of \eqref{redd2}},$$
and the relation in
Definition~\ref{def-ikmc}(6a) follows when  $i>\frac{1}{2}$.

Now suppose that $i=\frac{1}{2}$. In this case, we use the alternative
rational functions $h(x_1,x_2)$ and $h'(x_1,x_2)$ defined in \eqref{ration-h123}. We take
$
f_1(u)=m_{V,i}(-u)(1-2u)$ and $
f_2(u)=(u+i)^k t_{V,i}(-u)$,
for a sufficiently large integer $k$ in \eqref{key-coprim}. Then there
exist polynomials $g(u),h(u)\in Z_{\lambda,V}[u]$ satisfying
\eqref{key-coprim}.   Similarly, we take $f_1(u)=   n_{V, \frac{1}{2}}(u)$ and $f_2(u)=(u-\frac{1}{2})^k$
for a sufficiently large integer $k$. Then, there  are two further polynomials  $g'(u), h'(u)\in  Z_{\lambda, V}[u]  $   satisfying \eqref{key-coprim}. We remark that  $h(x_1, x_2)$ in \eqref{ration-h123} and $h(u)$ above are unrelated; the former  is a rational function in the variables $x_1$ and $x_2$, whereas  the latter  is a polynomial in $Z_{\lambda, V}[u]$.
Note that with the new choice of $h(x_1,x_2)$, $h'(x_1,x_2)$, the left-hand side of the following equation, which occurs in \eqref{redd1} and \eqref{redd2}, becomes
$$\begin{aligned}
   &
.
\end{aligned}$$
We point out that \eqref{redd1} still holds if we replace $-%
$ in the last equality of \eqref{redd1} by 
$- %
$,
and that  \eqref{redd2} remains valid   if we replace both $-%
$  in the penultimate equality of \eqref{redd2} by
$
- %
$.
 With these replacements, 
 \eqref{sss1} and \eqref{sss11} remain valid for
$i=\frac{1}{2}$  after  $h(x_1,x_2)$, $g(u)$ and the related auxiliary polynomials are replaced by their   newly
defined counterparts. The relations in Definition~\ref{def-ikmc}(6c)--(6d)
then follow by the same computations as those   for
Definition~\ref{def-ikmc}(6a)--(6b), with these replacements. We omit
the details, since they are entirely analogous to the previous case.
\end{proof}

\begin{Lemma} \label{mixedup1} (Mixed relations) The natural transformations defined in Definitions~\ref{natrual-basic} and ~\ref{natrual-basic1} satisfy the relations  in Definition~\ref{def-ikmc}(7a)-(7b).\end{Lemma} 
\begin{proof}
We use Lemma~\ref{lrcrossing} to convert the red diagrams on the left-hand sides of  Definition~\ref{def-ikmc}(7a) and (7b) into  black diagrams, and then apply Lemma~\ref{crossing-inverse} to obtain the desired equalities.%  in Definition~\ref{def-ikmc}(7a),(7b).
\end{proof}
The following theorem summarizes the preceding results.
\begin{Theorem}\label{thm:mainthmaffinebu}
Suppose that a locally Schurian category $\mathcal M$ admits an affine Brauer categorification.
 If the spectrum of the dot endomorphism on the generating endofunctor is exactly  $\frac12+\mathbb Z$, then the following results hold: \begin{itemize} \item [(1)]   $\mathcal M$ gives rise
    to a generalized nilpotent $2$-representation of
    $\mathfrak U_+$.
\item[(2)] The full subcategories  of finitely generated objects form  a nilpotent $2$-subrepresentation of $\mathfrak U_+$.
\item [(3)] When $\mathcal M$ is locally finite Abelian, $\mathcal M$ carries the structure of a nilpotent   $2$-representation of $\mathfrak U_+$. \end{itemize} 
\end{Theorem}

\begin{proof}By Lemma~\ref{biadj}, $EV$ is finitely generated and  
 $\End(EV)$ is finite-dimensional whenever $V$ is finitely generated. Hence  
the dot on $EV$ has only
finitely many generalized eigenvalues. Hence $E_iV = 0$ for all but finitely many $i$, and the dot on $E_iV$ is
nilpotent after subtracting the eigenvalue $i$. Since $F_i$ is realized as $E_{-i}$, Definition~\ref{nil}(1) and (2) hold for
every finitely generated object.
  Thus  (1) follows from Lemma~\ref{adj} and Lemmas~\ref{cyc-x}--\ref{mixedup1}. Moreover, by Lemma~\ref{ibiad},   the
functors $E_i$ and $F_i$  preserve finitely generated objects, so the finitely generated objects form a
nilpotent 2-subrepresentation, proving (2).

If $\mathcal M$ is locally finite Abelian, then every object has finite length and hence is finitely generated.  
Therefore  the conditions in Definition~\ref{nil}(1) and (2), already verified above for finitely generated objects, hold for every object $V\in \mathcal M$. Thus the resulting $2$-representation
is nilpotent.
\end{proof}

\begin{table}[!htbp]
  \centering
  \small
  \setlength{\tabcolsep}{4pt}
  \renewcommand{\arraystretch}{1.16}
  \caption{Verification of the defining relations of the $\imath$-Kac--Moody
  2-category from the affine Brauer action.}
  \label{tab:relation-verification}
  \begin{tabularx}{\textwidth}{@{}>{\bfseries}p{0.23\textwidth}p{0.20\textwidth}Y@{}}
    \toprule
    Target relation & Result & Main affine Brauer input \\
    \midrule
    Adjunction
      & Lemma~5.10
      & Cup--cap zig-zag relations and the chosen normalizations \\
    Cyclicity
      & Lemma~5.12
      & Dot sliding, reflection symmetry, and adjunction \\
    Quiver Hecke relations
      & Lemma~5.13
      & The affine crossing--dot relation on generalized eigenspaces \\
    Bubble relations
      & Lemma~5.14
      & The bubble generating series and the weight formula \\
    Nodal relations
      & Lemma~5.15
      & Formal Laurent-series extraction and generalized-eigenspace nilpotence \\
    Bicross relations
      & Lemma~5.16
      & Localized dot calculus, including the exceptional color $i=\tfrac12$ \\
    Mixed relations
      & Lemma~5.17
      & Compatibility of crossings with cups and caps \\
    Inhomogeneous relation
      & Theorem~6.15
      & The two boundary-braiding calculations and the rational identities
        collected in Section~9 \\
    \bottomrule
  \end{tabularx}
\end{table}
\FloatBarrier

%\newpage
\section{From affine Brauer actions to $\imath$-Kac--Moody $2$-representations, II:   the inhomogeneous relation} \label{diliu}
In  Section~5, we verified  the relations in Definition~\ref{def-ikmc} for  the natural transformations defined in
Definitions~\ref{natrual-basic} and~\ref{natrual-basic1}. In this section, we verify the remaining  inhomogeneous relation in Definition~\ref{def-ikmc1}. Together, these verifications  prove 
 Theorem~\ref{main111}(1):
any locally Schurian category \(\mathcal M\) admitting an affine Brauer
categorification naturally carries the structure of a generalized nilpotent
\(2\)-representation of \(\mathfrak U_+^\imath\), provided that the spectrum of the  dot endomorphism on the  
 generating endofunctor is exactly 
\(1/2+\mathbb Z\).

  Throughout, let  $\lambda\in X_\imath^+$ and $V$ be a finitely generated object
of $\mathcal M_\lambda$.
We first define several rational functions.
Since these functions will appear
as labels on diagrams,
we  often suppress their arguments. 
For example, we write
$h(x_1,x_2)$ simply as $h$,
and similarly for the other rational functions below. However, for rational functions defined in the single
variable $x$, such as $c(x)$, we make the argument explicit after substitution.
For instance, we write $c(x_1)$ to denote the value of $c(x)$ at $x=x_1$. 

%In this section, we apply \eqref{key-coprim} twice.  
{Choose a single integer $k$  sufficiently large so that both applications of \eqref{key-coprim} below are valid. First,  take}
\[
f_1(u)=m_{E_{-1/2}V, 1/2}(-u)(1-2u),
\qquad
f_2(u)=(u+1/2)^k t_{E_{-1/2}V, 1/2}(-u).
\]
Then there are $g(u), h(u)\in Z_{\lambda-\alpha_{1/2}, E_{-1/2}V}[u]$ satisfying \eqref{key-coprim}. Second,  {take}
\[
f_1(u)=n_{V,1/2}(u),\qquad
f_2(u)=(u-1/2)^k.
\]
Then there are $g'(u), h'(u)\in Z_{\lambda,V}[u]$ satisfying \eqref{key-coprim} for this second pair of polynomials.
\begin{Defn}\label{rationalff} Suppose that $x_1, x_2, x_3$ and $x$  are indeterminates. Define the following  rational functions:  
 \begin{multicols}{2}
%\item [(1)] $h_1^{-1}=\ell_{--}^{12} $, 
%\item[(2)] $h_2^{-1}=\ell_{-+}^{12}$,
%\item[(3)] $h_3^{-1}=\ell_{-+}^{13}$, 
%\item [(4)] $h_5=\ell_{+-}^{13}$.
%\item[(5)] $h_6=\ell_{-+}^{13}$
%\item [(6)] $(h_2')^{-1}=\ell_{+-}^{12}$, 
%\item [(7)] $(h_3')^{-1}=\ell_{+-}^{13}$, 
%\item[(8)] $(f_{ij})^{-1}=\ell_{+-}^{ij}$, $1\le i, j\le 3$, and $i\neq j$, 
%\item [(9)]  $(g_{ij})^{-1}=\ell_{++}^{ij}$, $1\le i, j\le 3$, and $i\neq j$, 
 \item [(1)] $h(x_1,x_2)=\frac{(x_1+x_2-1)(x_1-x_2)^2}{1+x_1-x_2}$, 
\item [(2)] $h'(x_1, x_2)=h(x_2,x_1)$, 
\item [(3)]$h_4(x_1,x_2)=\frac{(x_1-x_2)^2} {1-x_1+x_2}$.
 \item [(4)] $h_4'(x_1, x_2)=h_4(x_2,x_1)$,
\item[(5)]   $c(x)=\frac{(1-2x)m_{E_{-1/2}V,1/2}(-x)}{4xn_{E_{-1/2}V,1/2}(-x)}$, \item [(6)] $d(x)=\frac{n_{V,1/2}(x)}{xm_{V,1/2}(x)}$, 
 \item [(7)]  $p(x)=4xn_{E_{-1/2}V,1/2}(-x)g(x)$,
\item[(8)]  $q(x)=xm_{V,1/2}(x)g'(x)$.
 \end{multicols}
\end{Defn}  

We  use the notation of Definition~\ref{rationalff} until the end of Subsection~6.3.
  
 \begin{Lemma}\label{zero1} Let $p(x)$ and $q(x)$ be defined in Definition~\ref{rationalff}(7) and (8). Then 
 $$\begin{tikzpicture}[baseline=-0.5ex,scale=0.7, ]
        \draw[thick, black, ] (0,-0.5) -- (0,0.5);
        \node at (0, 0) {$\scriptstyle\bullet$}; 
        \node at (0.6,0) {$\scriptstyle{p(x)}$};        
        \node at (0, -0.7) {$\scriptstyle 1/2$};
  \draw[-,darkg,thick] (1.2,-0.5) to (1.2,.5);
\node at (1.2,-.7) {$\darkg \scriptscriptstyle{E_{-1/2}V}$};
        \end{tikzpicture}
    =0, \quad \text{and} \quad  
 \begin{tikzpicture}[baseline=-0.5ex,scale=0.7, ]
        \draw[thick, black, ] (0,-0.5) -- (0,0.5);
        \node at (0, 0) {$\scriptstyle\bullet$}; 
        \node at (0.6,0) {$\scriptstyle{q(x)}$};        
        \node at (0, -0.7) {$\scriptstyle -1/2$};
  \draw[-,darkg,thick] (1.2,-0.5) to (1.2,.5);
\node at (1.2,-.7) {$\darkg \scriptscriptstyle{V}$};
        \end{tikzpicture}
    =0.$$
\end{Lemma} 
\begin{proof}
    By the definition of $p(x)$ and $q(x)$, together with the two applications of  \eqref{key-coprim} made above, the labels 
    $p(x)$ and $q(x)$ contain the  factors that vanish on the corresponding generalized eigenspaces. Hence the two displayed dotted morphisms are zero.
\end{proof}
In all subsequent diagrams, unless otherwise specified,    the black strands are labeled at both the top and bottom, from left to right,  according to the sequence   
$(-\frac{1}{2}, \frac{1}{2},-\frac{1}{2})$.

\tikzset{diagram styles3/.style={line/.style={-},
line_arr/.style={->},
line_arrop/.style={<-},
cross/.style={-},
cross_arr/.style={->},
cross_arrop/.style={<-},
 arc_right/.style={
        ->, % 添加箭头样式
        blue,
        bend right=85,
        looseness=1.8
    },
arc_left/.style={blue, bend left=85, ,looseness=1.8,->},
mi_near_start/.style={
    decoration={
        markings,
        mark=at position 0.2 with {
            \node[cplus,scale=0.7] at (0,0) {$-$};
        }
    },
    postaction=decorate
},
mi_near_end/.style={
    decoration={
        markings,
        mark=at position 0.2 with {
            \node[cplus,scale=0.7] at (0,0) {$-$};
        }
    },
    postaction=decorate
},
pl_near_start/.style={
    decoration={
        markings,
        mark=at position 0.2 with {
            \node[cplus,scale=0.7] at (0,0) {$+$};
        }
    },
    postaction=decorate
},
pl_near_end/.style={
    decoration={
        markings,
        mark=at position 0.8 with {
            \node[cplus,scale=0.7] at (0,0) {$+$};
        }
    },
    postaction=decorate
},
dot_near_start/.style={
    decoration={
        markings,
        mark=at position 0.25 with {
            \fill[black] (0,0) circle (1.5pt);
            \coordinate (dot-start) at (0,0);
        }
    },
    postaction=decorate
},
dot_near_end/.style={
    decoration={
        markings,
        mark=at position 0.75 with {
            \fill[black] (0,0) circle (1.5pt);
            \coordinate (dot-end) at (0,0);
        }
    },
    postaction=decorate
},
a_near_start/.style={
    decoration={
        markings,
        mark=at position 0.05 with {
            \coordinate (a-start) at (0,0);
        }
    },
    postaction=decorate
},
a_near_end/.style={
    decoration={
        markings,
        mark=at position 0.95 with {
            \coordinate (a-end) at (0,0);
        }
    },
    postaction=decorate
},
b_near_start/.style={
    decoration={
        markings,
        mark=at position 0.15 with {
            \coordinate (b-start) at (0,0);
        }
    },
    postaction=decorate
},
b_near_end/.style={
    decoration={
        markings,
        mark=at position 0.85 with {
            \coordinate (b-end) at (0,0);
        }
    },
    postaction=decorate
},
line_near_start/.style={
    decoration={
        markings,
        mark=at position 0.1 with {
            \draw[thick, black] (1.5pt,-1.5pt) -- (1.5pt,1.5pt);
            \coordinate (line-start) at (0,0);
        }
    },
    postaction=decorate
},
line_near_end/.style={
    decoration={
        markings,
        mark=at position 0.8 with {
            \draw[thick, black] (1.5pt,-1.5pt) -- (1.5pt,1.5pt);
            \coordinate (line-end) at (0,0);
        }
    },
    postaction=decorate
},}}

\tikzset{diagram styles2/.style={ line/.style={-},
line_arr/.style={->},
line_arrop/.style={<-},
cross/.style={-},
cross_arr/.style={->},
cross_arrop/.style={<-},
 arc_right/.style={
        ->, % 添加箭头样式
        blue,
        bend left=85,
        looseness=1.8
    },
arc_left/.style={blue, bend right=85, ,looseness=2.5,-},
mi_near_start/.style={
    decoration={
        markings,
        mark=at position 0.15 with {
            \node[cplus,scale=0.7] at (0,0) {$-$};
        }
    },
    postaction=decorate
},
mi_near_end/.style={
    decoration={
        markings,
        mark=at position 0.8 with {
            \node[cplus,scale=0.7] at (0,0) {$-$};
        }
    },
    postaction=decorate
},
pl_near_start/.style={
    decoration={
        markings,
        mark=at position 0.2 with {
            \node[cplus,scale=0.7] at (0,0) {$+$};
        }
    },
    postaction=decorate
},
pl_near_end/.style={
    decoration={
        markings,
        mark=at position 0.8 with {
            \node[cplus,scale=0.7] at (0,0) {$+$};
        }
    },
    postaction=decorate
},
dot_near_start/.style={
    decoration={
        markings,
        mark=at position 0.3 with {
            \fill[black] (0,0) circle (1.5pt);
            \coordinate (dot-start) at (0,0);
        }
    },
    postaction=decorate
},
dot_near_end/.style={
    decoration={
        markings,
        mark=at position 0.6 with {
            \fill[black] (0,0) circle (1.5pt);
            \coordinate (dot-end) at (0,0);
        }
    },
    postaction=decorate
},
a_near_start/.style={
    decoration={
        markings,
        mark=at position 0.15 with {
            \coordinate (a-start) at (0,0);
        }
    },
    postaction=decorate
},
a_near_end/.style={
    decoration={
        markings,
        mark=at position 0.85 with {
            \coordinate (a-end) at (0,0);
        }
    },
    postaction=decorate
},
b_near_start/.style={
    decoration={
        markings,
        mark=at position 0.15 with {
            \coordinate (b-start) at (0,0);
        }
    },
    postaction=decorate
},
b_near_end/.style={
    decoration={
        markings,
        mark=at position 0.85 with {
            \coordinate (b-end) at (0,0);
        }
    },
    postaction=decorate
},
line_near_start/.style={
    decoration={
        markings,
        mark=at position 0.1 with {
            \draw[thick, black] (1.5pt,-1.5pt) -- (1.5pt,1.5pt);
            \coordinate (line-start) at (0,0);
        }
    },
    postaction=decorate
},
line_near_end/.style={
    decoration={
        markings,
        mark=at position 0.8 with {
            \draw[thick, black] (1.5pt,-1.5pt) -- (1.5pt,1.5pt);
            \coordinate (line-end) at (0,0);
        }
    },
    postaction=decorate
}, }}
\tikzset{diagram styles1/.style={line/.style={-},
line_arr/.style={->},
line_arrop/.style={<-}, 
cross/.style={-},
cross_arr/.style={->},
cross_arrop/.style={<-},
 arc_right/.style={
        ->, % 添加箭头样式
        blue,
        bend right=85,
        looseness=1.8
    },
arc_left/.style={blue, bend left=85, ,looseness=2.5,-},
mi_near_start/.style={
    decoration={
        markings,
        mark=at position 0.12 with {
            \node[cplus,scale=0.7] at (0,0) {$-$};
        }
    },
    postaction=decorate
},
mi_near_end/.style={
    decoration={
        markings,
        mark=at position 0.8 with {
            \node[cplus,scale=0.7] at (0,0) {$-$};
        }
    },
    postaction=decorate
},
pl_near_start/.style={
    decoration={
        markings,
        mark=at position 0.2 with {
            \node[cplus,scale=0.7] at (0,0) {$+$};
        }
    },
    postaction=decorate
},
pl_near_end/.style={
    decoration={
        markings,
        mark=at position 0.8 with {
            \node[cplus,scale=0.7] at (0,0) {$+$};
        }
    },
    postaction=decorate
},
dot_near_start/.style={
    decoration={
        markings,
        mark=at position 0.35 with {
            \fill[black] (0,0) circle (1.5pt);
            \coordinate (dot-start) at (0,0);
        }
    },
    postaction=decorate
},
dot_near_end/.style={
    decoration={
        markings,
        mark=at position 0.75 with {
            \fill[black] (0,0) circle (1.5pt);
            \coordinate (dot-end) at (0,0);
        }
    },
    postaction=decorate
},
a_near_start/.style={
    decoration={
        markings,
        mark=at position 0.15 with {
            \coordinate (a-start) at (0,0);
        }
    },
    postaction=decorate
},
a_near_end/.style={
    decoration={
        markings,
        mark=at position 0.85 with {
            \coordinate (a-end) at (0,0);
        }
    },
    postaction=decorate
},
b_near_start/.style={
    decoration={
        markings,
        mark=at position 0.15 with {
            \coordinate (b-start) at (0,0);
        }
    },
    postaction=decorate
},
b_near_end/.style={
    decoration={
        markings,
        mark=at position 0.85 with {
            \coordinate (b-end) at (0,0);
        }
    },
    postaction=decorate
},
line_near_start/.style={
    decoration={
        markings,
        mark=at position 0.1 with {
            \draw[thick, black] (1.5pt,-1.5pt) -- (1.5pt,1.5pt);
            \coordinate (line-start) at (0,0);
        }
    },
    postaction=decorate
},
line_near_end/.style={
    decoration={
        markings,
        mark=at position 0.8 with {
            \draw[thick, black] (1.5pt,-1.5pt) -- (1.5pt,1.5pt);
            \coordinate (line-end) at (0,0);
        }
    },
    postaction=decorate
},}}
\tikzset{
    diagram styles/.style={
               line_noarrow_solid/.style={-, black},
        line_arrow_solid/.style={->, black},
        line_noarrow_dashed/.style={-, densely dotted},
        line_arrow_dashed/.style={->, densely dotted},
        cross/.style={black, -},
        arc_right/.style={-, blue, bend right=80, looseness=1.5},
        arc_left/.style={-, blue, bend left=80, looseness=1.5},
        mi_near_start/.style={
            decoration={
                markings,
                mark=at position 0.2 with {\node[cplus,scale=0.7] at (0,0) {$-$};}
            },
            postaction=decorate
        },
        mi_near_end/.style={
            decoration={
                markings,
                mark=at position 0.8 with {\node[cplus,scale=0.7] at (0,0) {$-$};}
            },
            postaction=decorate
        },
        pl_near_start/.style={
            decoration={
                markings,
                mark=at position 0.2 with {\node[cplus,scale=0.7] at (0,0) {$+$};}
            },
            postaction=decorate
        },
        dot_near_start/.style={
            decoration={
                markings,
                mark=at position 0.25 with {
                    \fill[black] (0,0) circle (1.5pt);
                    \coordinate (dot-start) at (0,0);
                }
            },
            postaction=decorate
        },
        dot_near_end/.style={
            decoration={
                markings,
                mark=at position 0.75 with {
                    \fill[black] (0,0) circle (1.5pt);
                    \coordinate (dot-end) at (0,0);
                }
            },
            postaction=decorate
        },
        a_near_start/.style={
            decoration={
                markings,
                mark=at position 0.15 with {\coordinate (a-start) at (0,0);}
            },
            postaction=decorate
        },
        a_near_end/.style={
            decoration={
                markings,
                mark=at position 0.85 with {\coordinate (a-end) at (0,0);}
            },
            postaction=decorate
        },
    }
}

\subsection{Step 1: compute the first braiding}In this  subsection, we   compute the first braiding appearing 
on the right-hand side of the inhomogeneous relation in
Definition~\ref{def-ikmc1}.

Let $V\in \mathcal M_\lambda$ be a finitely generated object.
Evaluating the relevant natural transformations on $V$, we obtain
\begin{equation}\label{first}

}= Z_1+Z_2- Z_3.
\end{equation}
The first equality in \eqref{first} follows by applying
Lemma~\ref{lrcrossing}(1) and (2) together with \eqref{redijdown}, which
allows us to rewrite the three red crossings as black crossings.
For the second equality, we use Lemma~\ref{dotslidecrossing}(1)--(2)
to move the dots labeled by the rational function \(h_2'\) in
the first term as far as possible toward the top.  This gives the
three terms displayed on the right-hand side of \eqref{first}, which
we denote by \(Z_1+Z_2-Z_3\).
Here
\begin{equation}\label{notation100}
 c=c(x_1), \qquad p=p(x), \qquad h_1^{-1}=\ell_{--}^{12},
 \qquad h_5=\ell_{+-}^{13}, \qquad (h_2')^{-1}=\ell_{+-}^{12}.
\end{equation}
We denote by \(p(u)\) the polynomial obtained from \(p(x)\) by
replacing $x$ with \(u\).  Moreover,
\begin{equation} \label{Z1000} Z_{1} = .\end{equation}
We now prove  the following formula:   \begin{equation}
 \label{first1}\begin{aligned} 
  \operatorname{RHS} \text{of \eqref{first}} &= 
%
\Bigg]_{u^{-1}}}, \quad N=E_{-\frac{1}{2}}V.
\end{aligned}
\end{equation}  
%\comment{The $O(u)$ above should replaced by $\mathcal O_V(u)$? (such that to %compatible with later notation in Proposition 6.6)}
Furthermore, the terms  $Z_4, Z_5$, and $Z_6$ satisfy the following identities:
\begin{equation} \label{first333}\begin{aligned} 
Z_{4}& =
\mathord{
%
,\\
\end{aligned}
\end{equation}
where \begin{equation}\label{notationgij}
g_{ij}=(x_i+x_j)^{-1}, \text{ $1\le i, j\le 3$, and $i\neq j$.}\end{equation}
%\comment{Better to define $g_{ij}$ is Definition \ref{rationalff}?}
We will use the notations in \eqref{notation100}, \eqref{notationfij}, and \eqref{notationgij} until the end of Subsection~6.3. 

Assuming \eqref{first1} and \eqref{first333}, we can now complete the
calculation of the red diagram on the left-hand side of \eqref{first}.
Indeed, \eqref{first} first rewrites this red diagram as
\(Z_1+Z_2-Z_3\).  The identity \eqref{first1} then expresses
\(Z_1+Z_2-Z_3\) as a sum of black diagrams together with the terms
\(Z_4,Z_5,Z_6\).  Finally, applying the identities in \eqref{first333}
to \(Z_4,Z_5,Z_6\), we obtain the desired formula for the first
braiding appearing on the right-hand side of the inhomogeneous
relation in Definition~\ref{def-ikmc1}.
We emphasize that we do not know whether the black terms appearing
in \eqref{first1} are linearly independent over \(\mathbb C\).
Consequently, choosing a convenient collection of black terms in
\eqref{first1} is one of the main technical difficulties in this
subsection.

\begin{Lemma} \label{one} When evaluated  on $V$, the natural transformation $Z_1$ satisfies
\begin{equation}\label{lem47}
   \begin{aligned}
   Z_{1}
=& Z_4+

\Bigg]_{u^{-1}}.
  \end{equation}
Here we first use Lemma~\ref{usefuel-equa}(1) to rewrite the dot labeled by \(p\)
as a \(\ominus\) labeled by \(p(u)\). We then apply
Corollary~\ref{caup}(2) to move this \(\ominus\)  to the
bottom of the relevant local diagram, which gives the second equality.
Next, we compute the second summand on the right-hand side of \eqref{I1eq}.
Denote this summand by \(Z\). Then
  $$    \begin{aligned}
Z
=&-\mathord{\left[%
}\right]_{u^{-1}}.
    \end{aligned}$$
Here, for the first equality, we apply Corollary~\ref{caup}(2) to
move the upper \(\ominus\) on the cup of the diagram defining \(Z\)
through the crossing.  This produces three terms: the first and the
third terms on the right-hand side of the first equality, together with
one additional term whose resulting local diagram is of the form
appearing in Lemma~\ref{usefuel-equa}(4).  Applying
Lemma~\ref{usefuel-equa}(4), we replace this local diagram by the sum of
the two local diagrams on the right-hand side of that relation.  In the
present notation, these two terms contribute exactly the second and the
fourth terms on the right-hand side of the first equality.

For the second equality, we apply Lemma~\ref{ii} to the local crossing
in the first summand on the right-hand side of the first equality, whose
two strands are labeled, from left to right, by \(1/2\) and \(-1/2\).
This replaces that local crossing by the corresponding expansion in
Lemma~\ref{ii}.  The correction terms arising from this replacement are
precisely the \(f_{12}\)- and \(f_{21}\)-terms displayed above.
Combining this with \eqref{I1eq}, we obtain \eqref{lem47}, as required. \end{proof}

\begin{Lemma}\label{guoch}  When evaluated  on $V$, the natural transformation   $Z_3$  satisfies 
\begin{equation}\label{lem48}
      \begin{aligned}
    - Z_3 
&=  
.
\end{aligned}
  \end{equation}  
%where $c=c(x_1)$, $p=p(x)$, and $p(u)$ is the polynomial obtained from $p(x)$ by substituting 
%$u$ for $x$. 
% where $c=c(x_1)$ and $p=p(x)$.  %  where   $h_4(x_1,x_2)=(x_1-x_2)^2(1+x_2-x_1)^{-1}$.
\end{Lemma}
\begin{proof}
The following computation yields \eqref{lem48}: 
\begin{equation}
\label{I2eq}
\begin{aligned}
    -Z_3 &=
%
 .
\end{aligned}
\end{equation}
For the fourth equality, we use Corollary~\ref{caup}(2) to move the
\(\ominus\) in the diagram on the right-hand side of the third equality
through the crossing to the bottom.  This produces the second, third,
and fourth terms displayed after the fourth equality, together with one
additional term.  Applying Lemma~\ref{usefuel-equa}(1) to the local
diagram in this additional term gives the first diagram displayed after
the fourth equality. For the fifth equality, we apply Lemma~\ref{ii} to the local crossings in
the second and fourth summands displayed after the fourth equality, and
replace them by the corresponding expansions.
\end{proof}

\begin{Lemma}\label{guoc} When  evaluated  on $V$, the natural transformation $Z_2$ satisfies
$$
    Z_{2}=
.
$$
%where $Z_2$ is defined in \eqref{Z100}, and  $c=c(x_1)$, $p=p(x)$, and $p(u)$ is the polynomial obtained from $p(x)$ by substituting 
%$u$ for $x$. 
\end{Lemma}
\begin{proof}For the second term in the above formula, we first apply
Lemma~\ref{lrcrossing}(1)--(2) to convert the two black crossings into red
crossings.  Next, we use Lemma~\ref{bicross1 downred}, namely the image of the relation in
Definition~\ref{def-ikmc}(6d), to  expand the product of these two red crossings.  By
\eqref{sss1}, one of the red terms in the resulting expression is precisely the
term \(F_1\)  in \eqref{sss1} for $i=1/2$.  This gives  the desired formula for
\(Z_2\).\end{proof}

\begin{Prop}\label{emm3}   When  evaluated  on a finitely generated object  $V\in \mathcal M_\lambda$, the natural transformation $
\Bigg]_{u^{-1}}},\end{aligned}\end{equation}
 where $Z_4, Z_5$ and $Z_6$ are given  as in \eqref{first333}. 
\end{Prop}
\begin{proof}  Lemmas~\ref{one}--\ref{guoc} give formulas for $Z_1, Z_2$ and $ Z_3$. 
Substituting these formulas
into \eqref{first} gives the identity \eqref{first1}, where
$Z_4$, $Z_5$, and $Z_6$ are defined as in \eqref{Z100}.
It remains to prove that these three terms satisfy the identities in \eqref{first333}.

There are three terms in the expansion of \(Z_4\) in \eqref{Z100}. 
The first term contains a local diagram of the form appearing on the left-hand side of 
Lemma~\ref{usefuel-equa}(5). Applying Lemma~\ref{usefuel-equa}(5) to this local diagram and then using Lemma~\ref{zero1}, we obtain the first term on the right-hand side of \eqref{ZZZZ}.

For the second term in the expansion of \(Z_4\) in \eqref{Z100}, we apply Lemma~\ref{usefuel-equa}(8) to the relevant local diagram. The resulting diagram, together with the third term in the same expansion, gives the second term on the right-hand side of \eqref{ZZZZ}. Hence we obtain the identity:

\begin{equation}\label{ZZZZ}
 Z_{4} =-
 \begin{tikzpicture}
 [baseline=0, scale=0.5]
\draw[-,thick] (-1,1) -- (1, -1);
\draw[-,thick] (0, 1) -- (-1, -1);
\draw[-,thick] (-.2, -1) -- (.8,1);
\draw[-,densely dotted] (-.9,.9) to (0,.9);
\node at (-1.5,.8) {$\scriptstyle{h_1}$};
\draw[-,densely dotted] (-.9,-0.9) to (-.2,-.9);
%\node [cplus,scale=0.7] at (-0.6,-.2) {$-$};
%\node [cplus,scale=0.7] at (0.1,-.5) {$-$};
 \node at (-1.8,-.9) {$\scriptscriptstyle{h'f_{12}}$};
% \node at (1,0) {$\scriptscriptstyle{p(u)}$};
% \draw[-,densely dotted] (-1,-.9) to (0,-.9);
% \node at (-1.2,-.5) {$\scriptstyle{h'}$};
% \draw[-,densely dotted] (-1,.9) to (1,.9);
% \node at (1.7,.9) {$\scriptstyle{h_3'}$};
%\node at (1.4, 0.8) {$\scriptstyle\lambda$};
%\node[rectangle, fill=blue!20, minimum width=1.5cm, draw=black] at (0, -1.6) {$\kappa$}; 
\end{tikzpicture}
-
\Bigg[\mathord{\begin{tikzpicture}
[
    baseline =1cm, 
    scale=0.9, 
    % 居中基线
   diagram styles
]

% 定义边界框

% 定义顶点位置
\coordinate (v1) at (0.1,1);
\coordinate (v2) at (0.8,1);
\coordinate (w1) at (0.1,1.7);
\coordinate (w2) at (0.8,1.7);
\coordinate (u1) at (0.1,0.7);
\coordinate (u2) at (0.8,0.7);
\draw[cross, a_near_end,pl_near_start] (w1) to[arc_right] (w2);
\draw[cross] (v2) to[arc_right] (v1);
\coordinate (t1) at (2.8,1);
\coordinate (t2) at (2.8,1.7);
\draw[line_noarrow_solid,  a_near_start] (t2) -- (t1);
\draw[-, black, -, densely dotted] (a-end) -- (a-start) node[midway, above=0.01cm] {$\scriptstyle f_{12}$};
\draw[line_noarrow_solid, mi_near_start] (v1) -- (u1);
\draw[line_noarrow_solid] (v2) -- (u2);
\draw[cross,  a_near_end,  a_near_start] (v2) to[arc_right] (v1);
\draw[-, black, -, densely dotted] (a-end) -- (a-start) node[midway, below] {$\scriptstyle ch'$};
\node[right, yshift=-0.2cm] at (a-start)  {$\scriptstyle \frac{p(u)}{2u}$};
\end{tikzpicture}}\Bigg]_{u^{-1}}.
\end{equation}
We then apply Lemma~\ref{ii} to two local crossings of the first summand, and  use   Lemma~\ref{usefuel-equa}(6) and
Lemma~\ref{zero1} to rewrite   the second summand above. This gives
the required formula for $Z_4$ in \eqref{first333}.

The argument  for $Z_5$ is similar:  applying
Lemma~\ref{usefuel-equa}(5)--(6) and Lemma~\ref{zero1} to the
expression for $Z_5$ in \eqref{Z100} gives  the required
formula for $Z_5$ in \eqref{first333}.  We omit the details.

It remains to treat $Z_6$.  Recall that
$p(u)\in Z_{\lambda-\alpha_{1/2}, E_{-1/2}V}[u]$ is determined by the polynomial
$g(u)\in Z_{\lambda-\alpha_{1/2}, E_{-1/2}V}[u]$ satisfying \eqref{key-coprim}.
By Lemma~\ref{ou}(2), Lemma~\ref{lem:bublleyimeszeropoly},
and \eqref{key-coprim}, we have
  \begin{equation} \label{OU1} \mathcal O_{E_{-1/2}V}(-u)
\frac{p(-u)}{u}= 2(u-1/2)^{1+\epsilon_{1/2}(E_{-1/2}V)} \big[\frac{1}{t_{E_{-1/2}V,1/2}(u)}-h(-u)(-u+1/2)^k\big],
 \end{equation}
 for a sufficiently large integer $k$. 
  Then 
  \begin{equation*} 
  \begin{aligned}      
Z_{6}&=h_5\Bigg[\mathord{
}\right]_{u^{-1}}.
    \end{aligned} 
    \end{equation*}
Here, to obtain the first equality, we observe that the initial diagrams in the expansion of $Z_6$ in \eqref{Z100} contain, after composing the cup at the top with a cap,   a local diagram of the form appearing on the left-hand side of Lemma~\ref{usefuel-equa}(7).
Applying  
Lemma~\ref{usefuel-equa}(7) to rewrite this local diagram gives  the first equality. The third equality follows from
\eqref{OU1}.  In the last step, the term involving
$h(-u)(-u+1/2)^k$ vanishes. Indeed, we have   \begin{equation}
    \label{eat1}
%
    =0\end{equation}  for
sufficiently large $k$ and
$M\in\{E_{1/2}E_{-1/2}V,V\}$.
Finally, using the second relation in
Definition~\ref{natrual-basic1}(4), and then replacing $u$ by
$u+1/2$, we obtain the required identity for $Z_6$ in
\eqref{first333}.  This completes the proof.
 \end{proof} 

\subsection{Step 2: compute the second braiding} The aim of this subsection is to compute the second braiding
on the right-hand side of the inhomogeneous relation in
Definition~\ref{def-ikmc1}.  The method for computing this braiding is
somewhat similar to that used for the first braiding, although there are
also several important differences.

Throughout this subsection,  let $V\in \mathcal M_\lambda$ be  a finitely generated object.
Evaluating  the relevant natural transformations on $V$, we obtain
\begin{equation}\label{Second}
        \begin{aligned}

}=Z_1'-Z_2'+Z_3'. 
 \end{aligned}
    \end{equation}
The first equality in \eqref{Second} follows from
Lemma~\ref{lrcrossing}(1)--(2) and \eqref{redijdown}: the former converts the two red crossings into the corresponding black crossings, while the latter rewrites the remaining red crossing as a dotted black expression.

For the second equality, we apply Lemma~\ref{dotslidecrossing}(1)--(2) to slide the dot labeled by \(h_2\) toward the bottom in the first term. This turns the first term into  two contributions, denoted by \(Z_1'\) and \(Z_3'\); we also denote the second term by $Z_2'$. Explicitly,

\begin{equation}
Z_1'=\mathord{
,
} 
\end{equation}
where  \begin{equation} \label{notation200} {\tilde h_1}^{-1}=\ell_{--}^{23},  \ h_2^{-1}=\ell_{-+}^{12},\     d=d(x_1),\   h_6=\ell_{-+}^{13}, \  q=q(x).   \end{equation} 
We will use  the notation in \eqref{notation200} until the end of    Subsection~6.3. We are going to prove the following formula: 
\begin{equation}
 \label{second12}
  \begin{aligned}
\text{RHS of \eqref{Second}} 
=& 
%
}\Bigg]_{u^{-1}}. 
\end{aligned}
 \end{equation}
Furthermore, the terms $Z_7, Z_8$, and $ Z_9$ satisfy the following identities
  \begin{equation} 
    \begin{aligned}
     \label{second678}
   Z_{7}&=
-\mathord{
%
.
    \end{aligned}
    \end{equation}
Combining \eqref{second12} and \eqref{second678} with \eqref{Second}, we obtain  the desired formula for the second braiding on the right-hand side of the inhomogeneous relation in Definition~\ref{def-ikmc1}. It remains to verify \eqref{second12} and \eqref{second678}; this will be done in 
the rest of this subsection. 

\begin{Lemma}\label{guoch1}When evaluated on any finitely generated object  $V\in \mathcal M_\lambda$,   the natural transformation  $
  Z_1'$ satisfies 
\begin{equation}\label{III1eq}
    \begin{aligned}
   Z_1'
 = &\quad Z_7+
\mathord{
\Bigg[
    \end{aligned}
\end{equation}
where % $d=d(x_1)$, $q=q(x)$, and 
$q(u)$ is the polynomial obtained from $q(x)$ by substituting 
$u$ for $x$. 
\end{Lemma}
\begin{proof}
  We have
   \begin{equation}
       \begin{aligned}
      \label{III123eq}
     Z_1'
\overset{Lem~\ref{dotslidecrossing}(3)} =%
\Bigg]_{u^{-1}} .
       \end{aligned}
   \end{equation}
Here, we use  Lemma~\ref{usefuel-equa}(1) to rewrite the dot labeled by $q$ as the $\ominus$ labeled by $q(u)$, and then apply   Corollary~\ref{caup}(3) to move the $\ominus$ in this local diagram toward the top, which gives the second equality.  We next compute  the second summand after the second equality   of \eqref{III123eq}, which we denote by $Z$. We have    
%\begin{equation}\label{III12eq}
 $$\begin{aligned}
  Z =&-
\mathord{
\Bigg[%
}\Bigg]_{u^{-1}}
.
 \end{aligned}   $$
For the first equality, we first apply Corollary~\ref{caup}(1) to move the \(\ominus\) to the other endpoint of the cap, where it becomes an \(\oplus\). We then use Corollary~\ref{caup}(2) to move this \(\oplus\) through the crossing. This gives the first two terms on the right-hand side and one extra term. The local diagram in the extra term is evaluated by Lemma~\ref{usefuel-equa}(3) and Definition~\ref{C defn}(1), and is replaced by the resulting sum of two local diagrams. These give the third and fourth terms.

For the second equality, we apply Lemma~\ref{ii} to the local crossing in the first term on the left-hand side of the first equality, where the strands are labeled \(-1/2\) and \(1/2\) from left to right. This expands the crossing, with correction terms exactly given by the displayed \(f_{12}\)- and \(f_{21}\)-terms. Combining this with \eqref{III123eq} gives \eqref{III1eq}.\end{proof}

 \begin{Lemma}\label{guoch2}When evaluated  on any finitely generated object  $V\in \mathcal M_\lambda$, the natural transformation  $
Z_3'$ satisfies 
\begin{equation}\label{III2}
    \begin{aligned}
        Z_3'
=&
-
\Bigg]_{u^{-1}}
}\\
=&
\text{RHS of \eqref{III2}}.
\end{aligned}
 $$
For the third equality, we keep the second term of the second equality unchanged and simplify the first term. By Lemma~\ref{usefuel-equa}(1), the dot labeled \(q\) is rewritten as an \(\ominus\) labeled \(q(u)\). Then Corollary~\ref{caup}(3) moves this \(\ominus\) upward through the crossing, producing the first, second, and fourth terms in the third equality.

Next, applying Lemma~\ref{usefuel-equa}(4) and Definition~\ref{C defn}(1) to the second term of the third equality yields the second and third terms on the right-hand side of the next equality. Finally, Lemma~\ref{ii}, with strands labeled \(-1/2\) and \(1/2\), expands the local crossings in the fourth summand of the third equality; the correction terms are precisely the displayed \(f_{12}\)- and \(f_{21}\)-terms.\end{proof}

\begin{Lemma}\label{lem415}  When 
    acting on any finitely generated object $V\in \mathcal M_\lambda$, the natural transformation  
    $ Z_2'$ satisfies
    $$
           Z_2'=

}.$$ %where $d=d(x_1)$ and $q=q(x)$.where $d=d(x_1)$ and $q=q(x)$. 
\end{Lemma}
\begin{proof}
The proof is identical to that of Lemma~\ref{guoc}, with  Definition~\ref{def-ikmc}(6d)  replaced by Definition~\ref{def-ikmc}(6c), and \eqref{sss1}  replaced by \eqref{sss11}.
\end{proof}

\begin{Prop}\label{emm6}
 When acting on any finitely generated object  $V\in \mathcal M_\lambda$, the natural transformation $%
\Bigg]_{u^{-1}}
},
\end{aligned}
 \end{equation} where 
$Z_7, Z_8$ and $Z_9$ are defined as in \eqref{Z200}.  
\end{Prop}
\begin{proof}Lemmas~\ref{guoch1}--\ref{lem415} give formulas for \(Z_1'\), \(Z_2'\), and \(Z_3'\). Substituting these formulas into \eqref{Second} gives the identity \eqref{second12}, where \(Z_7\), \(Z_8\), and \(Z_9\) are defined as in \eqref{Z200}. It remains to prove that these three terms satisfy the identities in \eqref{second678}.

We first consider \(Z_7\). From the three-term expansion of \(Z_7\) in \eqref{Z200}, applying Lemma~\ref{usefuel-equa}(5) to the two \(\ominus\)'s in the first term and then using Lemma~\ref{zero1} gives the first term on the right-hand side of \eqref{ZZZZ2}. Similarly, applying Lemma~\ref{usefuel-equa}(7) to the two \(\oplus\)'s in the third term gives two diagrams which, together with the second term in the expansion of \(Z_7\), give the second term on the right-hand side of \eqref{ZZZZ2}. Hence
\begin{equation}
    \label{ZZZZ2}
 Z_{7}=-
 \begin{tikzpicture}
 [baseline=0, scale=0.5]
\draw[-,thick] (1,-1) -- (-1, 1);
\draw[-,thick] (0, -1) -- (1, 1);
\draw[-,thick] (.2, 1) -- (-.8,-1);
\draw[-,densely dotted] (.9,-.9) to (0,-.9);
\node at (1.5,-.8) {$\scriptstyle{h_1}$};
\draw[-,densely dotted] (.9,0.9) to (.2,0.9);
 \node at (1.8,0.9) {$\scriptscriptstyle{hf_{12}}$};
\end{tikzpicture}
-
\Bigg[\begin{tikzpicture}
[ baseline =-12mm,
scale=0.7, diagram styles]
% 定义边界框
% 定义顶点位置
\coordinate (t1) at (-0.6,-1);
\coordinate (t2) at (-0.6,-2);
\coordinate (v1) at (0.1,-1);
\coordinate (v2) at (0.8,-1);
\coordinate (w1) at (0.1,-2);
\coordinate (w2) at (0.8,-2);
\coordinate (u1) at (0.1,-0.7);
\coordinate (u2) at (0.8,-0.7);
\draw[line_noarrow_solid,, a_near_end] (t1) -- (t2);
\draw[cross,  a_near_start, pl_near_start] (w1) to[arc_left] (w2);
\draw[-, black, -, densely dotted] (a-end) -- (a-start) node[midway, below] {$\scriptstyle f_{12}$};
\draw[cross,    a_near_end] (v2) to[arc_left] (v1);
\draw[line_noarrow_solid, mi_near_start] (v1) -- (u1);
\draw[line_noarrow_solid] (v2) -- (u2);
\draw[cross,  a_near_end,  a_near_start] (v2) to[arc_left] (v1);
\draw[-, black, -, densely dotted] (a-end) -- (a-start) node[midway, above] {$\scriptstyle dh$};
\node[right, yshift=-0.4cm] at (a-start)  {$\scriptstyle \frac{q(u)}{2u} $};
\end{tikzpicture}\Bigg]_{u^{-1}}. 
\end{equation}

We then apply Lemma~\ref{ii} to the two local crossings in the first summand of \eqref{ZZZZ2}. The second summand is rewritten using Lemma~\ref{usefuel-equa}(6) and Lemma~\ref{zero1}. This gives the required formula for \(Z_7\) in \eqref{second678}.

The proof for \(Z_8\) is similar. More precisely, applying Lemma~\ref{usefuel-equa}(5) to the two \(\ominus\)'s in each term of \(Z_8\), and then using Lemma~\ref{zero1}, gives the formula for \(Z_8\) in \eqref{second678}. We omit the details.

It remains to compute $Z_9$.  By Lemma~\ref{lem:bublleyimeszeropoly} and \eqref{key-coprim}, we have 
  \begin{equation}  
 \label{OU2}
\mathcal O_V(u)q(u)u^{-1}=  \frac{t_{V,1/2}(u)(1-(u-\frac{1}{2})^k h'(u))}{(u-1/2)^{\epsilon_{1/2}(V)}}\end{equation}
  Then
    \begin{equation*}
    \begin{aligned}
       Z_{9} \overset{Lem~\ref{usefuel-equa}(8)}  =  h_6&
\Bigg[\mathord{
.
    \end{aligned}
\end{equation*}
The first equality is obtained by  applying Lemma~\ref{usefuel-equa}(8) to the \(\oplus\) and lower \(\ominus\) on the right-hand side  of the first equality, as well as  to the corresponding $\oplus$ and $\ominus$  in the second term on the left-hand side.
 The second equality follows from \eqref{OU2}.  The third equality follows from the vanishing relation in \eqref{eat1}. 
 Finally, for the last equality in the computation of $Z_9$,  we apply   the first relation in Definition~\ref{natrual-basic1}(4), and then replace $u$ by $u+1/2$. 
 This completes the proof. \end{proof}

\subsection{Step 3: Combine the two braidings}  The main result of this subsection is the computation of  the difference of
the two braidings   $$
$$
 appearing in the
inhomogeneous relation in Definition~\ref{def-ikmc1}. %To carry out this computation, we introduce the natural transformations  %$Z_9$ and $Z_{10}$, below. 
By Propositions~\ref{emm3} and~\ref{emm6},  we have
\begin{equation}
    \label{zuiz}
\begin{aligned}
& %
-Z_7-Z_8\right)\frac{1}{h_6}.
\end{aligned}\end{equation} 
Here \(Z_4\) and \(Z_5\) are the natural transformations defined in
Proposition~\ref{emm3}, and \(Z_7\) and \(Z_8\) are those defined in
Proposition~\ref{emm6}. The remaining two terms, \(Z_{10}\) and \(Z_{11}\),
are defined by
\begin{equation}
\label{Z300}
\begin{aligned} 
Z_{10} =& %
\Bigg]_{u^{-1}}\frac{1}{h_6}.
}\end{aligned} \end{equation} 
In the above formulas, we use the notation 
\begin{equation}\label{notationh3} (h_1')^{-1}=\ell_{--}^{23}, \  h_3^{-1}=\ell_{-+}^{13}, \    (h_3')^{-1}=\ell_{+-}^{13}.
\end{equation}
This notation will be used throughout the rest of this subsection.   The term  $Z_{10}$ will be simplified  in Lemma~\ref{parth}, while  $Z_{11}$ will be computed in Lemma~\ref{emm9}. 
%a step that will require the use of Lemmas~\ref{x1x3} and ~\ref{1x1x3}.
\begin{Lemma}
\label{parth}
When evaluated  on any finitely generated object $V\in \mathcal M_\lambda$, the natural transformation  $Z_{10}$ satisfies the following identity:
 {\small   \begin{equation}\label{z9}
\begin{aligned}
Z_{10}
= & A\frac{1} {1-x_2-x_3}+B\frac{1} {(1-x_1+x_3)(1-x_2-x_3)}\\
\\&+(1-x_2+x_3)C\frac{1} {(1-x_2-x_3)(1-x_1+x_3)} 
-\frac{1} {1+x_2-x_3} C\frac{1} {(1-x_2-x_3)(1-x_1+x_3)} \\
&-D\frac{x_1+x_2-1}{(1-x_2-x_3)(1-x_1+x_3)}+\frac{1} {1+x_2-x_3} D\frac{x_1+x_2-1} { (1-x_1+x_2)(1-x_2-x_3)(1-x_1+x_3)}\\  \end{aligned}\end{equation}}
where  \begin{itemize}\item [(1)] $ A= 
\mathord{
  = \operatorname{RHS} \text{ of  \eqref{z9}} . 
\end{equation}
By Lemma~\ref{AB relations-re}(3) and Definition~\ref{C defn}(4), the dots can be moved across  crossings.  In the first diagram appearing on the left-hand side of \eqref{Z10Z},  we first move the dot labeled by \(h_1h_3'\)
successively to the bottom of the diagram.  After this move, a unique term appears
which consists of a composition of three crossings and has a dot labeled by
\(h'h_1'h_3\) at the bottom.  To produce the term that  cancels the second diagram
on the left-hand side of  \eqref{Z10Z},   we must move the  bottom dot labeled by
\(h'\) to the top of the diagram.  The second dot-moving process produces precisely the diagram
needed to cancel that second term, together with several additional terms.
All extra  terms produced by these two   dot-moving processes are  exactly the terms appearing on the right-hand
side of \eqref{z9}.  Thus, using Definition~\ref{AB defn}(2) together with 
\eqref{Z300}, we obtain the required formula for \(Z_{10}\), and hence \eqref{z9} is verified.
\end{proof}

To compute the natural transformation $Z_{11}$, we introduce the natural transformations  $Z_{12}$ and $Z_{13}$ as follows:
\begin{equation} \label{Z400} \begin{aligned} Z_{12}=& h_5^{-1}
\Bigg[
\Bigg]_{u^{-1}}h_6^{-1}}.\end{aligned}
\end{equation}

\begin{Lemma}\label{x1x3}
When evaluated  on any finitely generated object $V\in \mathcal M_\lambda$, the natural transformations $Z_{12}$ and $Z_{13}$ satisfy the following relations: 
\begin{itemize} \item [(1)]    \label{lemfg}

$Z_{12}=%
\right]_{u^{-1}}h_6^{-1}}. $$
\end{itemize}
\end{Lemma}
\begin{proof} 
In the following, the expansion $t f_{31}^{-1}$ is used as a single piece of notation  for the rational function obtained after
canceling the factor $x_3-x_1$ from the numerator of $t$. Only the resulting function is evaluated on the local diagram; neither $f_{31}^{-1}$ nor the uncanceled expression is interpreted as an operator on that diagram. In
particular, no inverse of $x_3-x_1$ is involved.
We have the following computation: 
\begin{equation} 
\label{leftcaup1}
    \begin{aligned}
    Z_{12}&
=
 h_5^{-1}
%
\right)h_6^{-1}\\
&
= 
\text{RHS of (1),}\\
\end{aligned}
\end{equation}
where 
$ p_{13}=p(-x_1)-p(-x_3)$, $q_{13}=q(x_3)-q(x_1)$, $ p'=p(-x_2)$ and $ q=q(x_1)$. 
Here, the first equality follows by applying Lemma~\ref{usefuel-equa}(5) to the two $\ominus$’s in the first term, and to the two $\oplus$’s in the second term of $Z_{12}$ in \eqref{Z400}.  The second equality follows immediately from the definitions of $p_{13}$ and $q_{13}$ above. %Lemma~\ref{usefuel-equa}(5) together with Lemma~\ref{zero1}.

It remains to explain the last equality. The first and the third terms after the second equality in \eqref{leftcaup1} are exactly the first and  second terms on the right-hand side of (1). In the second  term after the second equality in \eqref{leftcaup1}, the dotted cap at the bottom is a red cap oriented  to the left.  Similarly, the dotted cup in the fourth term is in fact a cup oriented to the left. 
Both rewritings follow from Definition~\ref{natrual-basic}(3).  For the second term, we apply Lemma~\ref{ii21}(2) to deal with the rational function $p'$ on  the cup, thereby rewriting the black cup as a red one whose orientation is to the left.  This gives exactly the third term on the right-hand side of (1).  For the fourth term, we first apply Lemma~\ref{ii21}(1) to the cap carrying the label $q$, and then use Definition~\ref{natrual-basic}(3) to rewrite the black cap as a red one whose orientation is to the left.  This gives the fourth term on the right-hand side of (1). This completes the explanation of the last equality.

%Let   $q=q(x_1)$ and $p'=p(-x_2)$. Then 
Similarly, (2) follows from the following computation: 
\begin{equation}\label{leftcup111} 
   \begin{aligned} 
Z_{13}=&\mathord{-h_5^{-1}
\Bigg[
h_6^{-1}} +f_{21}Af_{32}
\\
=&\text{RHS of (2)}.
\end{aligned}
\end{equation}
The explanation is analogous. \end{proof}

In the following, $G/(x_3-x_1)$ denotes the rational function obtained after
cancelling the factor $x_3-x_1$ from the numerator of $G$; in
particular, no inverse of $x_3-x_1$ is involved. All remaining
denominator factors act invertibly on the relevant generalized
eigenspaces.
\begin{Lemma}\label{emm9} When evaluated on  any finitely generated object  $V\in \mathcal M_\lambda$, the natural transformation $Z_{11}$ satisfies 
$$Z_{11}=%
,
    \end{aligned}
    \end{equation}
    where $$H=\frac{(x_3-x_2)^2}{1+x_3-x_2}-\frac{(x_1-x_2)^2}{1+x_1-x_2}+\frac{1}{(1+x_1-x_2)}-\frac{1}{(1+x_3-x_2)}.$$ 
    A direct computation shows that $H=x_3-x_1$. Moreover, after
putting $G=G_1+G_2$ over a common denominator, its numerator is
divisible by $x_3-x_1$. Thus the quotient
$
\frac{G}{x_3-x_1}
$
is well defined. By (6.30) and (6.34), the left-hand side of (6.37) is
precisely \(Z_{12}+Z_{13}=Z_{11}\). Since \(H=x_3-x_1\),
canceling this factor in the first coefficient on the
right-hand side of \eqref{eat4}  gives \(1\). The asserted formula
for \(Z_{11}\) follows.
 \end{proof}

\begin{Lemma}\label{tttt}
When evaluated  on any finitely generated object  $V\in \mathcal M_\lambda$, we have:
\begin{itemize}\item [(1)] \label{tu1} $
$.
\end{itemize}
\end{Lemma}

\begin{proof} A direct computation using Lemma~\ref{ii} and Lemma~\ref{dotslidecrossing}(3) yields the desired result.
\end{proof} 
\begin{Theorem} \label{inhomo100} \label{Pi=1}%All strands in the diagram  below are labeled by $\frac{1}{2}$. Then  we have $$ 
When evaluated  on any finitely generated object $V\in \mathcal M_\lambda$, the inhomogeneous  relation in  Definition~\ref{def-ikmc1} holds.
\end{Theorem}
\begin{proof}  Combining \eqref{zuiz} with the explicit  expressions for $Z_{4}$, $Z_{5}$, $Z_{7}$, $Z_{8}$, $Z_{10}$, and $Z_{11}$ in  Propositions~\ref{emm3},~\ref{emm6}, and  Lemmas~\ref{parth}, ~\ref{emm9},~\ref{tttt}, we obtain \begin{equation}
    \label{zuizho}
\begin{aligned}
&
,
\end{aligned}\end{equation}
where the rational functions $T_i$, $1\leq i\leq5$, are defined
explicitly in Section~9.
By Lemma~\ref{lem:unprimed-rational-identities},
$T_1=T_5=1$ and $ T_2=T_3=T_4=0$.
\end{proof}

\begin{Theorem}\label{thm:mainthmaffineb}
Suppose that a locally Schurian category $\mathcal M$ admits an affine Brauer categorification.
If the spectrum of the dot endomorphism on the generating endofunctor is exactly  $\frac12+\mathbb Z$,
 then the following results hold: \begin{itemize} \item [(1)]  $\mathcal M$ gives rise
    to a generalized nilpotent $2$-representation of
    $\mathfrak U^\imath_+$.
\item[(2)] The full subcategories  of finitely generated objects form a nilpotent $2$-subrepresentation of $\mathfrak U^\imath_+$.
\item [(3)] When $\mathcal M$ is locally finite Abelian, $\mathcal M$ carries the structure of a nilpotent   $2$-representation of $\mathfrak U^\imath_+$. \end{itemize} 
\end{Theorem}

\begin{proof}By Lemmas~\ref{adj}--\ref{mixedup1}, the natural transformations satisfy all the relations in Definition~\ref{def-ikmc}. By Theorem~\ref{inhomo100},  they
satisfy the remaining inhomogeneous relation in Definition~\ref{def-ikmc1}. Hence the asserted $2$-representation is well-defined. 
The finiteness
assertions in (2) and (3) follow from the argument given at the end of Section 5.
\end{proof}
The theorem furnishes a recognition principle for $\imath$-Kac--Moody 2-representations. Instead of verifying the full
defining relations of $\mathfrak  U^\imath_+$ directly, it is enough to construct an affine Brauer action with the prescribed
half-integral spectrum. In Lie-theoretic examples this input is often canonical: the dot is induced by a tensor
Casimir operator, while the remaining affine Brauer generators arise from symmetry and adjunction.
\subsection{A non-defining  relation for $i\ne \frac{1}{2}$  }
The following relation for $i\ne \frac{1}{2}$ and $i\in I^\imath$ appears in \cite[\S3.6]{BWW25}:
\begin{equation}
    \label{equ:inhomotype}
\mathord{

\right)
}\end{equation}
% \old{It is a consequence of the defining relations. Hence, it is not necessary to verify it in the proof of
% Theorem~\ref{thm:mainthmaffineb}.} 
Since \eqref{equ:inhomotype}  already follows from the defining relations verified in Theorem~\ref{thm:mainthmaffineb}, the purpose of this subsection is to establish several identities that will be used repeatedly in the proof of the braid relation for the black morphisms in the affine Brauer category in Section 7.
%Nevertheless, it can be checked directly by the same kind of computation as
%in the case $i=\frac{1}{2}$ considered in the previous three subsections.
%Since some computations will be
%needed in the next section, 
%where we prove that a nilpotent 2-representation of $\mathfrak U^\imath_+$ gives
%rise to a module category over the affine Brauer category, 
We include, in the remainder of this section, a sketch
of the necessary computations for the case $i\neq \frac{1}{2}$, highlighting only the differences from
the case $i=\frac{1}{2}$. 

In this subsection, we take $$f_1(u)= n_{V, i}(u) \ \  \text{and} \ \  f_2(u)= (u-i)^k$$  in \eqref{key-coprim}  for a sufficiently large integer $k$. By \eqref{key-coprim}, there are $g(u), h(u)\in Z_{\lambda, V}[u]$ satisfying \eqref{key-coprim}. Similarly, we take 
$$f_1(u)= m_{E_{-i}V, i} (-u) (4u^2-1) \ \  \text{and}\ \  f_2(u)= (u+i)^kt_{E_{-i}V,i}(-u) $$  in \eqref{key-coprim}  for a sufficiently large integer $k$.  By \eqref{key-coprim}, there are $g'(u), h'(u)\in Z_{\lambda-\alpha_i, E_{-i}V}[u]$ satisfying \eqref{key-coprim}.

Throughout the remainder of  this subsection, we  use the notations  $h_1, h_2, h_3, h_5, h_6, h_2', h_3'$  defined as  in \eqref{notation100}, \eqref{notation200} and \eqref{notationh3}.
We continue to use $f_{ab}=(x_a-x_b)^{-1}$ and $g_{ab}=(x_a+x_b)^{-1}$, as in  \eqref{notationfij} and \eqref{notationgij}, for $1\le a,b\le 3$ with $a\neq b$. 
We also introduce   the following notation  to simplify the presentation.

 \begin{Defn}\label{rationalffj} Suppose that $x, x_1, x_2$ are indeterminates. Define the following  rational functions:  
 \begin{multicols}{2}
\item [(1)] $h(x_1,x_2)=\frac{(x_1+x_2-1)(x_1-x_2)^2}{(x_1-x_2)^2-1}$, 
\item [(2)]  $h_4(x_1,x_2)=\frac{(x_1-x_2)^2} {(x_2-x_1)^2-1}$,
%\item [(3)] $h'(x_1, x_2)=h(x_2,x_1)$, 
%\item [(4)] $h_4'(x_1, x_2)=h_4(x_2,x_1)$,
\item[(3)]   $c(x)=\frac{(4x^2-1)m_{E_{-i}V,i}(-x)}{4xn_{E_{-i}V,i}(-x)}$, \item [(4)]  $p(x)=4xn_{E_{-i}V,i}(-x)g'(x)$, \item [(5)] $d(x)=\frac{n_{V,i}(x)}{xm_{V,i}(x)}$, \item[(6)]  $q(x)=xm_{V,i}(x)g(x)$.
\end{multicols}\end{Defn} 

We always suppress the arguments. 
For example, we write
$h_4(x_1,x_2)$ simply as $h_4$,
and similarly for the other rational functions below. However, for rational functions defined in the variable $x$,
such as $c(x)$, we  make the argument explicit after substitution.
For instance, we write $c(x_1)$ for the value of $c(x)$ after substituting  $x=x_1$.
We define the auxiliary functions
\[
h'(x_1,x_2):=h(x_2,x_1),\qquad h'_4(x_1,x_2):=h_4(x_2,x_1).
\]
Indeed,   $h'=h$ and $h'_4=h_4$ since    the rational functions $h(x_1, x_2)$ and $h_4(x_1, x_2)$ are symmetric. We retain these symbols to keep the notation parallel to that used above  for $i=\frac12$.

%We shall use this notation throughout the remainder of this section. 

In the following,  the three vertices in the top row (and similarly in the bottom row) are labeled, from left to right, by $ -i, i, -i$ in the black diagram, whereas they are labeled by $ i, i, i $ in the red diagram.

Repeating the computations from  the previous subsections with the new notation gives  analogues of all the preceding lemmas. In the following, we record only the differences. First, by arguing exactly as in the proofs of Lemma~\ref{guoc} and Lemma~\ref{lem415}, but under the new notation, we obtain the following.
 
 %Mimicking all the computations in the previous subsections we have all the lemmas under the new notation. In the %following we only mention the differences. First of all, 
 %mimicking the proofs of Lemma~\ref{guoc} and Lemma~\ref{lem415},  under the new notation we have
\[\begin{aligned}
\mathord{

},\]
where   $c=c(x_1), d=d(x_1)$,  $p=p(x)$, and  $q=q(x)$.
Moreover, $p(u), q(u)$ are the polynomials obtained from $p(x), q(x)$, respectively, by substituting 
$u$ for $x$.

By the same argument as in the proof of Theorem~\ref{inhomo100}, we obtain  the following identity:   
\begin{equation}
    \label{zuizz}
\begin{aligned}
& %
-Z_7-Z_8\right)\frac{1}{h_6}, 
\end{aligned}\end{equation} 
where $Z_4$ and $ Z_5$ are defined as in   Proposition~\ref{emm3}, 
$ Z_7$ and $ Z_8$ as in Proposition~~\ref{emm6}, and  $Z_{10}$ and $Z_{11}$  as in   \eqref{Z300}. However, as noted  above, the three vertices in the top 
row (and similarly in the bottom row) are labeled, from left to right, by $-i,i,-i$. In addition,  the rational functions appearing there should  be replaced by their counterparts in Definition~\ref{rationalffj}.  For example, the function $c$ in Proposition~\ref{emm3} is $c(x)$ in Definition~\ref{rationalff}, whereas here $c$ denotes the function defined     in Definition~\ref{rationalffj}. The same applies to the other rational functions.

\begin{Lemma}\label{emm9j} When 
evaluated  on any finitely generated object $V\in \mathcal M_\lambda$, the natural transformation $Z_{11}$ satisfies the relation: 
$$Z_{11}=
\begin{tikzpicture}[baseline = 0]
\draw[-,thick] (-0.45,.6) to (0.45,-.6);
\draw[-,thick] (0,.6) to[out=-90,in=180] (.25,0.2);
\draw[-,thick] (0.25,0.2) to[out=0,in=-90] (.45,0.6);
\draw[-,thick] (-.45,-.6) to[out=90,in=180] (-.2,-.2);
\draw[-,thick] (-.2,-.2) to[out=0,in=90] (0,-0.6);
%\draw[-,densely dotted] (.44,.5) to (0,.5);
\draw[-,densely dotted] (0,-.5) to (0,.5);
%\node at (-.8,.5) {$\scriptstyle{dh}$};
%\node [cplus,scale=0.7] at (0,0) {$-$};
\node at (-.67,0.1) {$\scriptscriptstyle{\frac{G}{x_3-x_1}}$};
 %\node at (.8,-.5) {$\scriptstyle{h_1}$};
 %\node [cplus,scale=0.7] at (0,.5) {$-$};
%\node [cplus,scale=0.7] at (-0.45,-.5) {$+$};
 %\node at (-1.8,-.9) {$\scriptscriptstyle{ch'}$};
% \node at (.45,0) {$\scriptscriptstyle{q(u)}$};
\end{tikzpicture},$$ 
where $G=G_1+G_2$, and 
$$\begin{aligned} G_1=&-\frac{4x_1^2+ 4x_1^2 (x_1-x_2)(x_2-x_3)}{(1+x_2+x_3)(1-x_2-x_3)(1+2x_1)(x_2-x_3)(x_1-x_2)(2x_1-1)}\\
 G_2= & \frac{4x_3^2+4x_3^2(x_1-x_2)(x_2-x_3)}{(1+x_1+x_2)(1-x_1-x_2)(1+2x_3)(x_1-x_2)(x_2-x_3)(2x_3-1)}.\\
\end{aligned}$$
\end{Lemma}

\begin{proof}After putting $G=G_1+G_2$ over a common denominator, its
numerator is divisible by $x_3-x_1$. Hence
$G/(x_3-x_1)$ is well defined.
The proof follows the same pattern as that   of Lemma~\ref{emm9}. We therefore omit detailed justifications for the computation below, since  the same arguments apply. 
    \begin{equation}
        \label{lemfgj}\begin{aligned} 
         Z_{12}
=& h_5^{-1}

+f_{21}Af_{32},\\
\end{aligned}
\end{equation}   where $Z_{12}$ and  $Z_{13}$ are defined  in the present setting by the analogue of 
\eqref{Z400}, and $A$ is defined by the analogue of the  formula  in Lemma~\ref{x1x3}(2). 
From ~\eqref{lemfgj}, we deduce that 
   \begin{equation}
       \label{eat2}
    \begin{aligned}
  Z_{12} 
+
h_5^{-1}\left[
%
.
    \end{aligned}
     \end{equation}
     By the analogue of  \eqref{Z300} and the definitions of \(Z_{12}\) and
\(Z_{13}\) in the present setting, the left-hand side of \eqref{eat2}
is precisely \(Z_{12}+Z_{13}=Z_{11}\). The asserted formula
therefore follows.
\end{proof}

\begin{Lemma}\label{pathj} When 
evaluated on any finitely generated object $V\in \mathcal M_\lambda$, the natural transformation $Z_{10}$ satisfies:      
$$
\begin{aligned}
Z_{10}= & A\frac{1} {1-x_2-x_3}+B\frac{1} {(1-x_1+x_3)(1-x_2-x_3)} \\ 
&- C\frac{1} {(1-x_2-x_3)(1-x_1+x_3)} +\frac{1}{1-(x_2-x_3)^2}C\frac{1} {(1-x_2-x_3)(1-x_1+x_3)} \\
&+
\frac{1} {1+x_2-x_3}D\frac{1-x_1-x_2}{(1-(x_1-x_2)^2)(1-x_2-x_3)(1-x_1+x_3)}\\ &
{-\frac{1} {1-(x_2-x_3)^2} D\frac{1-x_1-x_2} { (1+x_1-x_2)(1-x_2-x_3)(1-x_1+x_3)}}
\\
\end{aligned}$$ where  $A, B, C$ and $ D$ are defined in the present setting by the same diagrammatic  formulas as  in Lemma~\ref{parth}, with the rational functions replaced by those appearing around Definition~\ref{rationalffj}.
\end{Lemma}

\begin{proof}
  The result is  proved by the same dot-moving  arguments as in the proof of Lemma~\ref{parth}, with the rational functions replaced by those in Definition \ref{rationalffj}.
\end{proof}

\section{From  $\imath$-Kac--Moody $2$-representations to affine Brauer actions}\label{ikmcmc}

Assume throughout this section that
\(\mathcal M=(\mathcal M_\lambda)_{\lambda\in X^\imath_+}\)
is a locally Schurian nilpotent \(2\)-representation of
\(\mathfrak U^\imath_+\).
As in Section~4.2, every dotted bubble in
\(\mathfrak U^\imath_+\) gives rise to an element of
\[
Z_V:=Z(\operatorname{End}_{\mathcal M}(V)),
\]
where \(V\) is a finitely generated object of \(\mathcal M\).
In this section, we prove that \(\mathcal M\) admits a module
structure over the affine Brauer category \(\mathcal{AB}\) for
which the dot spectrum is contained in
\(\frac12+\mathbb Z\).

\subsection{Reverse generators}
\begin{Defn}\label{gen-red1} Fix     $ \lambda\in X_\imath^+$ and a  finitely generated object $V\in \mathcal M_\lambda$. We use   the following notation:   set \begin{itemize} \item [(1)] $E = \bigoplus_{i \in I^\imath} (E_i\oplus F_i)$.
\item [(2)] $\varepsilon_i(V),   \varepsilon_{-i}(V)\in \mathbb N$ such that  $u^{\varepsilon_i(V)}$ and  $u^{\varepsilon_{-i}(V)}$  are the minimal polynomials of 
\(
u^{-r-1}$, where   $i\in I^\imath$.
\end{itemize}
\end{Defn}

Then $E$ is  a  well-defined endofunctor of \(\mathcal{M}\). Since   $\mathcal M$   carries a nilpotent $2$-representation of $\mathfrak U_+^\imath$, Definition~\ref{nil}(1) guarantees that, on each object, only finitely many summands are nonzero;
hence the adjunctions $(E_i,F_i)$ and \((F_i,E_i)\) induce a biadjunction $(E,E)$, and $E$ is sweet.
If   $E_iV\neq 0$ for infinitely  many  $i\in I^\imath$, then $E$ may fail to be  a sweet functor.
Definition~\ref{nil}(2) guarantees that   $\varepsilon_i(V)$ and $\varepsilon_{-i}(V) $ are  well-defined for any $i\in I^\imath$.  

\begin{Lemma}\label{bubb-inftyi} Let $i\in I^\imath, \lambda\in X^+_\imath$,  and let $R=\text{End}_{\mathfrak U_+^\imath}(1_\lambda)$. Then  %the following relations hold: % define 
\begin{itemize}\item [(1)] $ =2^{\delta_{i, 1/2}} u^{\delta_{i, 1/2}} 1_\lambda$.
\end{itemize}
\end{Lemma}

\begin{proof} Statements (1) and (2) follow immediately from Definition~\ref{def-ikmc}(4a)–(4d), together with  \eqref{scalar23}, whereas statement (3) follows directly from Definition~\ref{def-ikmc}(4e).
 \end{proof}

For later use, we label the corresponding red dots by powers of $y$. More explicitly, we identify 
 $%
}$. 
This allows  us to    
represent a $\mathbb C$-linear combination of monomials by  labeling red dots with   polynomials in 
$y$.

\begin{Lemma}\label{lem:curlrebubblrmove}
For any $\lambda\in X^+_\imath$, and any $i, j\in I^\imath$, we have
\begin{multicols}{2}
\item[(1)] \label{crul-relation1}
$\mathord{
%
}$.
\end{multicols}
\end{Lemma}
\begin{proof}
For the quasi-split type AIII with an even number of nodes, the isomorphism of 2-categories in \cite[Thm~3.8]{BWW25} identifies the conventions of \cite{BWW25} with those of \cite{BSWW-icate}. Under this isomorphism, relations (1), (3) correspond to \cite[(3.42)]{BWW25}; relations (5)--(6) to \cite[(3.29)]{BWW25}; and relations (2), (4) to \cite[(3.40)]{BWW25}. Hence the relations hold in the present notation.
\end{proof}

%In the following, we assume that $\mathcal M$ is a 2-representation of $\mathcal U^\imath_+$.
From here to the end of this section, unless otherwise stated, we assume that $\lambda\in X^\imath_+ $ and that  $V$ is a finitely generated object in $\mathcal M_\lambda$. 
 We  will omit $

$ from the notation, and  write $
%
 $.
 Similarly, in what follows we  omit the rightmost factor involving  $V$.

\begin{Lemma}\label{zeropol-red} Let $i\in I^\imath$, and  let $f(u)\in  Z_{V} [u]$ be  a monic polynomial. When the corresponding natural
transformations are evaluated  on $V$, the following statements  hold:  
\begin{itemize}
\item[(1)] Assume that $
%
\, f(u)$.
Then $g(u)\in Z_{V}[u]$ is a monic polynomial such that
\(
\deg g(u)
=
\deg f(u)+\langle{}^\theta\alpha_i^\vee,\lambda\rangle
+\delta_{i,1/2},
\)
and
\(
%
(u)\, f(u)
\).
Then $g(u)\in Z_{V}[u]$ is a polynomial with leading
coefficient $2^{\delta_{i,1/2}}$, and
\(
\deg g(u)
=
\deg f(u)-\langle{}^\theta\alpha_i^\vee,\lambda\rangle.
\)
Moreover,
\(
%
=0
\).
\end{itemize}
\end{Lemma}

\begin{proof}  
Suppose  $r\ge 0$.  When evaluating  on $V$,  we have
\[[g(u)]_{u^{-r-1}}=[
%
=0, \]
and the first assertion in  (1) follows.  
The second one follows from  Lemma~\ref{lem:curlrebubblrmove}(4) by mimicking the proof of Lemma \ref {lem:bublleyimeszeropoly}.
%\eqref{crul-relation2}. 
Since  (2) can be verified in the same way, we omit the details of the proof.  \end{proof} 

\begin{Defn} \label{evalu-bub} Suppose $i\in I^\imath$. For any $\lambda\in X_\imath^+$  and any finitely generated object \(V\in\mathcal M_\lambda\),   define
\begin{itemize} \item [(1)] $\circlearrowleft_{V,i}(u)=\phi_V(%
  )$,   where $\phi_V$ is the evaluation homomorphism defined in \eqref{evalu-na},
        \item [(2)]  $\circlearrowright_{V,i}(u)=\phi_V(%
(-u-j)\, (-1)^{\langle {}^\theta \alpha_j^\vee,\lambda\rangle}\, (1/2)^{\delta_{j,1/2}}$,
\item [(5)] 
   $n_V(u)=\widetilde{\mathcal O}_V(u)m_V(u)$,
where $\varepsilon_j(V)$ is given in Definition~\ref{gen-red1}(2).
 %where $\phi_V$ is the evaluation homomorphism  defined  in \eqref{evalu-na}. 
\end{itemize}
\end{Defn}

In (4), the product is understood after applying $\varphi_V$.
It is well defined since, by nilpotence and the finite support of
$\lambda$, all but finitely many factors act as the identity on $V$.

\begin{Lemma}\label{simplered} Let   $i\in I^\imath$, and let $\lambda\in X_+^\imath$.   
Then, for every simple object $L\in \mathcal M_\lambda$,  the following relations hold. Here,  $\delta_{a_L}=1$ if $a_L$ is odd and $0$ otherwise: 
\begin{multicols}{2} \item [(1)] $
\circlearrowleft_{L,i}(u) = u^{\varepsilon_{-i}(L)-\varepsilon_i(L)+\delta_{i,1/2}(1-\delta_{a_L})}$,
\item [(2)]  $
\varepsilon_{-i}(L)-\varepsilon_i(L) = \langle {}^\theta \alpha_i^\vee,\lambda\rangle + \delta_{i,1/2}\delta_{a_L}$,
\item [(3)] 
$
\widetilde{\mathcal O}_L(u) = ((-1)^{a_L}u - 1/2)\frac{m_L(-u)}{m_L(u)}$.\end{multicols}
\end{Lemma}

\begin{proof} 
Specializing  Lemma~\ref{zeropol-red}(1) to  $f(u)=u^{\varepsilon_i(L)}$, we see  that $\circlearrowleft_{L,i}(u)u^{\varepsilon_i(L)}\in \mathbb C[u]$ is monic with degree $\varepsilon_i(L)+\langle {}^\theta\alpha_i^\vee,\lambda\rangle+\delta_{i,1/2}$ and is divisible by $u^{\varepsilon_{-i}(L)}$. Similarly,   applying  Lemma~\ref{zeropol-red}(2) with   $f(u)=u^{\varepsilon_{-i}(L)}$ shows that  $\circlearrowright_{L,i}(u)u^{\varepsilon_{-i}(L)}/2^{\delta_{i,1/2}}\in \mathbb C[u]$ is  monic with degree $\varepsilon_{-i}(L)-\langle {}^\theta\alpha_i^\vee,\lambda\rangle$ and  is divisible by $u^{\varepsilon_i(L)}$. 
Consequently,
\begin{equation}\label{ep-inequ1}
\varepsilon_{-i}(L) \le \varepsilon_i(L)+\langle {}^\theta\alpha_i^\vee,\lambda\rangle+\delta_{i,1/2}, \quad 
\varepsilon_i(L) \le \varepsilon_{-i}(L)-\langle {}^\theta\alpha_i^\vee,\lambda\rangle.
\end{equation}
{Since  $\lambda$ is even,  $\sum_{j\in I^\imath}\langle{}^\theta\alpha_j^\vee,\lambda\rangle\equiv 0 \pmod 2$. By Definition~\ref{evalu-bub}(3), we have      $$\sum_{i\in I^\imath} \varepsilon_{-i}(L)-\varepsilon_i(L) \equiv a_L \pmod 2.$$   
Now,  the inequalities in  \eqref{ep-inequ1} imply relation~(2). } They  also yield 
relation~(1) when $a_L$ is odd. 

 Suppose $a_L$ is even. Then the same  argument using \eqref{ep-inequ1} shows that 
\[
\circlearrowright_{L,i}(u)
=
2^{\delta_{i,1/2}}
u^{\varepsilon_i(L)-\varepsilon_{-i}(L)}
.\]  By Lemma~\ref{bubb-inftyi}(3), we
 obtain relation~(1) in this case  as well. 
Finally, relation~(3) follows from relation~(1), Lemma~\ref{bubb-inftyi}(3),
and Definition~\ref{evalu-bub}(3)-(4).  
\end{proof}
\begin{Defn}
    For  any $\lambda\in X_{\imath}^+$, let $\tilde Z_\lambda$ be  the subalgebra of $\End(\text{Id}_{\mathcal M_\lambda})$ generated by all diagrams  $\begin{tikzpicture}[baseline = 0]
  \draw[->,thick,darkred] (0.2,0.2) to[out=90,in=0] (0,.4);
  \draw[-,thick,darkred] (0,0.4) to[out=180,in=90] (-.2,0.2);
\draw[-,thick,darkred] (-.2,0.2) to[out=-90,in=180] (0,0);
  \draw[-,thick,darkred] (0,0) to[out=0,in=-90] (0.2,0.2);
 \node at (0,-.1) {$\scriptstyle{i}$};
   \node at (-0.3,0.2) {$\scriptstyle{\lambda}$};
   %\node at (0,0.2) {$\scriptstyle{i}$};  
   \node at (0.2,0.2) {$\color{darkred}\scriptstyle\bullet$}; \node at (0.37,0.2) {$\color{black}\scriptstyle{k}$};
  
        \end{tikzpicture},  \begin{tikzpicture}[baseline = 0 ]
  \draw[<-,thick,darkred] (0,0.4) to[out=180,in=90] (-.2,0.2);
  \draw[-,thick,darkred] (0.2,0.2) to[out=90,in=0] (0,.4);
 \draw[-,thick,darkred] (-.2,0.2) to[out=-90,in=180] (0,0);
  \draw[-,thick,darkred] (0,0) to[out=0,in=-90] (0.2,0.2);
 \node at (0,-.1) {$\scriptstyle{i}$};
   \node at (0.3,0.2) {$\scriptstyle{\lambda}$};
   \node at (-0.2,0.2) {$\color{darkred}\scriptstyle\bullet$}; \node at (-0.37,0.2) {$\color{black}\scriptstyle{k}$};
   
\end{tikzpicture}  \in \End(\text{Id}_{\mathcal M_\lambda})$ for all $k\in \mathbb N$ and all $i\in I^{\imath}$, viewed as  natural transformations of the identity functor. For any  $V\in \mathcal M_\lambda$, let  $\tilde Z_{\lambda,V}=\phi_V(\tilde Z_\lambda)$, where $\phi_V$ is the evaluation homomorphism defined in \eqref{evalu-na}.
\end{Defn}
%\comment{We need  a different notation here since they are different from the ring in Section 4. }
\begin{Lemma}\label{local-red}
    For any finitely generated object $V\in \mathcal M_\lambda$, $\tilde Z_{\lambda,V}$ is a finite--dimensional local  $\mathbb C$-algebra.
\end{Lemma}
\begin{proof} We work in the  $\imath$-Kac--Moody $2$-category $\mathfrak U^\imath$ over $\mathbb C$. 
Since $V$ is finitely generated, and $\mathcal M_\lambda$ is locally Schurian, $\End_{\mathcal M_\lambda}(V)$ is a finite--dimensional $\mathbb C$-algebra, and hence its subalgebra $\tilde Z_{\lambda,V}$ is also finite--dimensional.  
    Let $L$ be any simple object in $\mathcal M_\lambda$.  
    By Lemma~\ref{simplered}(1)-(2), for any $i\in I^{\imath}$ we have   \[
    \begin{tikzpicture}[baseline=0]
        \draw[->,thick,darkred] (0.2,0.2) to[out=90,in=0] (0,.4);
        \draw[-,thick,darkred] (0,0.4) to[out=180,in=90] (-.2,0.2);
        \draw[-,thick,darkred] (-.2,0.2) to[out=-90,in=180] (0,0);
        \draw[-,thick,darkred] (0,0) to[out=0,in=-90] (0.2,0.2);
        \node at (0,-.1) {$\scriptstyle{i}$};
        \node at (-0.3,0.2) {$\scriptstyle{\lambda}$};
        \node at (0.2,0.2) {$\color{darkred}\scriptstyle\bullet$};
        \node at (0.37,0.2) {$\color{black}\scriptstyle{k}$};
    \end{tikzpicture}\begin{tikzpicture}[baseline = 0]
   \draw[-,darkg,thick] (0.65,-0.5) to (0.65,.5);
    \node at (0.65,-.7) {$\darkg\scriptstyle{L}$};
 \end{tikzpicture}
    =
    \begin{cases}
        1 & \text{if } k = 1 - \langle \thal^\vee_i, \lambda+\alpha_i\rangle,\\
        0 & \text{otherwise}.
    \end{cases}
    \]
    Then by Lemma~\ref{bubb-inftyi}(3),
\[
    \begin{tikzpicture}[baseline=0]
        \draw[<-,thick,darkred] (0,0.4) to[out=180,in=90] (-.2,0.2);
        \draw[-,thick,darkred] (0.2,0.2) to[out=90,in=0] (0,.4);
        \draw[-,thick,darkred] (-.2,0.2) to[out=-90,in=180] (0,0);
        \draw[-,thick,darkred] (0,0) to[out=0,in=-90] (0.2,0.2);
        \node at (0,-.1) {$\scriptstyle{i}$};
        \node at (0.3,0.2) {$\scriptstyle{\lambda}$};
        \node at (-0.2,0.2) {$\color{darkred}\scriptstyle\bullet$};
        \node at (-0.37,0.2) {$\color{black}\scriptstyle{k}$};
    \end{tikzpicture}\begin{tikzpicture}[baseline = 0]
   \draw[-,darkg,thick] (0.65,-0.5) to (0.65,.5);
    \node at (0.65,-.7) {$\darkg\scriptstyle{L}$};
 \end{tikzpicture}
    =
    \begin{cases}
        2^{\delta_{i,\frac{1}{2}}} & \text{if } k = \langle \thal^\vee_i, \lambda\rangle - 1,\\
        0 & \text{otherwise}.
    \end{cases}
    \]
Therefore, every red bubble, with or without dots, acts on a simple
object as a scalar; moreover, this scalar is independent of the
simple object in $\mathcal M_\lambda$. Define
\[
\widetilde J_\lambda
:=
\left\{
b\in\widetilde Z_\lambda
\ \middle|\
b\text{ acts as zero on every simple object of }\mathcal M_\lambda
\right\},
\]
and set
\(
\widetilde J_{\lambda,V}
:=
\phi_V(\widetilde J_\lambda)
\).
By the same argument used for $Z_{\lambda,V}$ after 
Definition~\ref{phiV}, we have
\[
\widetilde Z_{\lambda,V}/\widetilde J_{\lambda,V}
\cong \mathbb C,
\]
and every element of $\widetilde J_{\lambda,V}$ is nilpotent.
Since $\widetilde Z_{\lambda,V}$ is finite-dimensional,
$\widetilde J_{\lambda,V}$ is a nilpotent ideal. Hence it is contained
in every maximal ideal of $\widetilde Z_{\lambda,V}$. Since
$\widetilde J_{\lambda,V}$ is itself maximal, it is the unique maximal
ideal. Therefore $\widetilde Z_{\lambda,V}$ is a local ring.

\iffalse \old{Therefore, every red bubble, with or without dots, acts on a simple object as a scalar; moreover, this scalar
is independent of the simple object in $\mathcal M_\lambda$. }

\old{By the same argument as in the proof of Lemma~\ref{lem:localring},
we conclude that $\tilde Z_{\lambda}$ is a local ring. Consequently, its image $\tilde Z_{\lambda,V}$ under the evaluation homomorphism  is also  a local ring.}\fi

\end{proof}

For any $g\in\tilde Z_{\lambda,V}$,  we write $\bar g$ for its image  in $\mathbb C$.

\begin{rem} By Lemma~\ref{zeropol-red}, the polynomial $n_V(u)$ defined in
Definition~\ref{evalu-bub}(5) lies in $\tilde Z_{\lambda,V}[u]$ and is monic.
By Lemma~\ref{local-red}, we can apply Lemma~\ref{lem:moniclift}  to 
define the monic polynomials $t_{V,i}(u)$ and $n_{V,i}(u)$ as in Corollary~\ref{lift-n}.
To avoid confusion, we emphasize that the notation
$m_V(u)$, $t_{V,i}(u)$, $n_{V,i}(u)$, and so on, is used here for the
corresponding polynomials in the category $\mathfrak U_+^\imath$, even
though the same notation was used in Section~4. After proving the main
result of this section, we will see that these polynomials
coincide with their counterparts in Section~4.
\end{rem}  %调了下位置
\begin{Lemma}\label{bubb-red}
Let $i\in I$,  and let $V \in \mathcal M_\lambda$ be a  finitely generated object.
Then  
 $$t_{V,i}(u)=\begin{cases} (u-i)^{\varepsilon_i(V)}\circlearrowleft_{V,i}(u-i) & \text{if $i\in I^{\imath}$}\\
(1/2)^{\delta_{-i, 1/2}} (-1)^{\langle {}^\theta\alpha_{-i}^\vee,\lambda\rangle}(u-i)^{\varepsilon_{i}(V)}\circlearrowright_{V,-i}(-u+i) & \text{if $i\in -I^\imath$.}\\
\end{cases} 
$$
 \end{Lemma}

\begin{proof} Suppose first that  $i, j\in I^{\imath}$.
    By Lemma~\ref{zeropol-red}, we have 
   $$(1/2)^{\delta_{j,1/2}}(-1)^{\langle {}^\theta\alpha_j^\vee,\lambda\rangle}(u+j)^{\varepsilon_{-j}(V)}\circlearrowright_{V,j}(-u-j), 
  \ \  (u-i)^{\varepsilon_i(V)}\circlearrowleft_{V,i}(u-i)\in  \tilde{Z}_{\lambda,V}[u].$$
    By Lemma~\ref{simplered}(1),
    $$\overline{(u-i)^{\varepsilon_i(V)}\circlearrowleft_{V,i}(u-i)}  = (u-i)^{\varepsilon_i(V)}\circlearrowleft_{L,i}(u-i) = (u-i)^{\phi_i(V)}$$
    for some $\phi_i(V)\in\mathbb N$, where $L\in\mathcal M_\lambda$ is a simple object.
    Similarly, Lemma~\ref{simplered}(1), together with  Lemma~\ref{bubb-inftyi}(3), gives 
    $$\overline{(1/2)^{\delta_{j,1/2}}(-1)^{\langle {}^\theta\alpha_j^\vee,\lambda\rangle}(u+j)^{\varepsilon_{-j}(V)}\circlearrowright_{V,j}(-u-j)} = (u+j)^{\phi_{-j}(V)}$$
    for some $\phi_{-j}(V)\in\mathbb N$.
   Thus, for \(i\in I^\imath\) and \(-i\in I^\imath\), respectively,
\(
(u-i)^{\epsilon_i(V)}
\circlearrowleft_{V,i}(u-i)
\)
and
\(
\left(\frac{1}{2}\right)^{\delta_{-i,\frac{1}{2}}}
(-1)^{\langle {}^\theta\alpha^\vee_{-i},\lambda\rangle}
(u-i)^{\epsilon_i(V)}
\circlearrowright_{V,-i}(-u+i)
\)
are monic lifts of \((u-i)^{\phi_i(V)}\), and their product is \(n_V(u)\).    
   The result therefore follows from the  uniqueness of monic lifts  in Lemma~\ref{lem:moniclift}.
\end{proof}

%\noindent 

As in Section~4, we shall use the following standard argument freely: since $\mathcal M$ is locally Schurian, every object is a filtered colimit of finitely generated objects; moreover, all endofunctors involved are sweet and hence preserve filtered colimits. Consequently, any natural transformation defined on finitely generated objects extends uniquely to all objects. Therefore, throughout this section it suffices to verify all identities on finitely generated objects $V\in\mathcal M_\lambda$, and we will do so without further mention.
\begin{Defn}
\label{natrual-basic-red}
For any  $\lambda\in X^+_\imath$,  any 
$i\in I^\imath$, $j, k\in I$,  and  any finitely generated object $V\in \mathcal M_\lambda$, 
define 
\begin{multicols}{2} 
\item[(1)]
    $
=0 $  if $ k\neq j$.
\end{multicols}
\end{Defn}

By the same argument as in Lemma~\ref{tvi}, with
$Z_{\lambda,V}$ replaced by $\widetilde Z_{\lambda,V}$,
the assignments in (5)--(6) are well defined and natural in $V$.
Hence they define natural transformations on $\mathcal M$.

\begin{Lemma}\label{equa-dotte-generatetilde} Let $i\in I$. %\old{ and let $V$ be a  finitely generated object in $\mathcal M_\lambda$.} 
Then 
 \begin{multicols}{2}
\item[(1)]  $
    \widetilde\clock (u)\widetilde\clock ({-}u)=(\frac{1}{2}-u)(\frac{1}{2}+u)$,
    \item[(2)] $\mathord{
%
}$.
\end{multicols}
\end{Lemma}
\begin{proof} Statement (1) follows from Lemma~\ref{bubb-inftyi}(3). Applying Lemma~\ref{lem:curlrebubblrmove}(5)-(6) repeatedly yields (2). 
\end{proof}

%\newpage

\begin{Defn}\label{cross-red123} Suppose   $\lambda\in X^+_\imath$, % any finitely generated object $V\in \mathcal M_\lambda$, 
and $i, j\in I^\imath$. Define 
\begin{itemize} 
\item [(1)] % For any $i,j\in I_\imath$ we define 
%\begin{align*}
$
%
$.
\end{itemize}
\end{Defn}
We remark that, on each generalized-eigenspace summand, the definitions above are designed to invert the constructions in   Definitions~\ref{natrual-basic} and~\ref{natrual-basic1}.

Since $\mathcal{M}$ is a nilpotent $2$-representation of $\mathfrak{U}_+^\imath$, Definition~\ref{nil}(1) ensures that, after evaluation
on any object of M, only finitely many summands in Definition~\ref{cross-red123}(6)-(10) are nonzero. Hence these sums
are well-defined. In what follows, we omit $\lambda$ from the notation whenever no ambiguity arises.

\subsection{Adjunction and mixed relations}

\begin{Lemma}
\label{dot-caup-red}
When evaluated on any finitely generated object
$V\in \mathcal M_\lambda$ with $\lambda\in X_\imath^+$,
all equalities in Definition~\ref{C defn}(1), and (4) hold.
Moreover, the following three relations hold:
    \begin{multicols}{2}
        \item [(1)] $
$
      \end{multicols} 
\end{Lemma}

\begin{proof}
 To prove relations (1) and (2), we first apply Definition~\ref{natrual-basic-red}(2)--(6) to rewrite the black dots and black cups/caps as red ones. Then, applying  Definition~\ref{def-ikmc}(1) and Definition~\ref{def-ikmc}(2a) to the resulting red diagrams, we obtain (1) and (2).

To prove the equalities in Definition~\ref{C defn}(1), it suffices to
show that
\[
\mathord{
%
}
\qquad \text{for every $i\in I$.}
\]
The proof is identical to that of Lemma~\ref{adj}, except that
Definition~\ref{def-ikmc}(1) is used in place of
Definition~\ref{C defn}(1). We therefore omit the details.

 To prove the equalities in Definition~\ref{C defn}(4) and relation (3), it suffices to
verify the following identities for all $i,j,k,l\in I$:
\begin{equation}\label{dotcross-red}
%
\end{equation}
We verify \eqref{dotcross-red} and \eqref{dotcross-red2} case by case.

If either $i\neq k$ or $l\neq j$, then   Definition~\ref{cross-red123}(3)--(5) rewrites each crossing in \eqref{dotcross-red} (resp., \eqref{dotcross-red2}) as either a pair of dotted black vertical strands, a dotted cup-cap composite, or zero.
Applying Definition~\ref{C defn}(1) together with relations (1)--(2) then gives \eqref{dotcross-red} (resp., \eqref{dotcross-red2}). 

Now assume that $i=k$ and $l=j$.
If $i,j\in I^\imath$,
 then by Definition~\ref{cross-red123}(1) and Definition~\ref{def-ikmc}(3a), the left-hand side of \eqref{dotcross-red} (resp., \eqref{dotcross-red2}) equals
$
\delta_{i,j}\,
\begin{tikzpicture}[baseline=-1mm]
    \draw[-] (0.18,-.28) to (0.18,.28);
    \draw[-] (-0.18,-.28) to (-0.18,.28);
    \node at (-0.18,-0.43) {$\scriptstyle{j}$};
    \node at (0.18,-0.43) {$\scriptstyle{i}$};
\end{tikzpicture}
\begin{tikzpicture}[baseline=0]
    \node at (0,0) {$\scriptstyle{\lambda}$};
\end{tikzpicture}$. Similarly, Definition~\ref{natrual-basic-red}(7)-(9) shows that the right-hand side of \eqref{dotcross-red} (resp., \eqref{dotcross-red2}) reduces to the same expression. Therefore, \eqref{dotcross-red} (resp., \eqref{dotcross-red2}) holds.

If $i,-j\in I^\imath$ (respectively, $-i,j\in I^\imath$, respectively, $-i,-j\in I^\imath$,), then \eqref{dotcross-red} follows from Definition~\ref{C defn}(1), the case $i,j\in I^\imath$ in \eqref{dotcross-red2} together with the second (resp. first, resp. third) equation in Definition~\ref{cross-red123}(2) and relations (1)--(2). In those cases \eqref{dotcross-red2} can be proved similarly. The only difference is that one has to use \eqref{dotcross-red} in place of  \eqref{dotcross-red2}. For example, suppose $-j,i\in I^\imath$. Then
\begin{equation}\label{empty1}
\begin{tikzpicture}[baseline=10pt, scale=0.5]
    \draw[thick, name path=curve1] (0.5,1.5) to[out=down, in=up] (0.2,0);
    \draw[thick, name path=curve2] (0.7,0) to[out=up, in=down] (0,1.5);
    \node at (0.2,-0.53) {$\scriptstyle j$};
    \node at (0.7,2.03) {$\scriptstyle j$};
    \node at (0.7,-0.53) {$\scriptstyle i$};
    \node at (0,2.03) {$\scriptstyle i$};
    \path[name intersections={of=curve1 and curve2, by=A}];
    \node at (A) {$\diamond$}; \node at (0.05,1.35) {$\scriptstyle\bullet$};
\end{tikzpicture}
\overset{\text{Def.~\ref{cross-red123}(2)}}=
\begin{tikzpicture}[baseline = 0, yscale=-1, scale=0.8]
	\draw[-,thick] (0.3,-.5) to (-0.3,.5);
	\draw[-,thick] (-0.2,-.2) to (0.2,.3);
        \draw[-,thick] (0.2,.3) to[out=50,in=180] (0.5,.5);
        \node at  (0.15,-.2) {$\scriptstyle\bullet$};
        \draw[-,thick] (0.5,.5) to[out=0,in=90] (0.8,-.5);
        \draw[-,thick] (-0.2,-.2) to[out=230,in=0] (-0.5,-.5);
        \draw[-,thick] (-0.5,-.5) to[out=180,in=-90] (-0.8,.5);
        \node at (0,0.05) {$\diamond$};
  \node at (-0.8,.65) {$\scriptstyle{j}$};
   \node at (0.28,-.62) {$\scriptstyle{i}$};
    \node at (-0.1,.65) {$\scriptstyle{i}$};
   \node at (0.78,-.62) {$\scriptstyle{j}$};
   \node at (1.05,.05) {$\scriptstyle{\lambda}$};
\end{tikzpicture}.
\end{equation}

For the local crossing on the right-hand side  of \eqref{empty1}, we can apply the case $j,i\in I^{\imath}$ in \eqref{dotcross-red2} proved above. Then, using the second equation in Definition~\ref{cross-red123}(2) and Definition~\ref{C defn}(1), we obtain \eqref{dotcross-red}.

 Thus the equalities in Definition~\ref{C defn}(4) and relation (3) hold in all cases.
 \end{proof}
 
% \old{Finally, relation~(3) follows from Lemma~\ref{AB relations-re}. Indeed,
% Lemma~\ref{AB relations-re}(1)(2) are  precisely relation~(1)(2) above, and
% Lemma~\ref{AB relations-re} shows that relation~(3) follows from
% Definition~\ref{C defn}(1), (3), and (4) and Lemma~\ref{AB relations-re}(1)(2). }
% %Lemma 2.2(3)也许可以删掉
\begin{Cor}\label{def6}
When evaluated on any finitely generated object
$V\in \mathcal M_\lambda$ with $\lambda\in X_\imath^+$,
all equalities in Definition~\ref{C defn}(5) hold.
\end{Cor}
\begin{proof}
The result follows immediately from Lemma~\ref{dot-caup-red}(1) together
with Definition~\ref{C defn}(1), which has already been verified in
Lemma~\ref{dot-caup-red}.
\end{proof}
\begin{Lemma}
\label{lemmasred}
In the present setting, the equations in Lemmas~\ref{ii21}, \ref{caup3}, and \ref{lrcrossing}, as well as equation~\eqref{redijdown}, continue to hold.
\end{Lemma}

\begin{proof}
The proof of Lemma~\ref{ii21} uses only the relations for $\mathcal O_V(u)$ stated in Lemma~\ref{ou}(2). By Lemma~\ref{equa-dotte-generatetilde}(2), the corresponding relations hold for $\widetilde{\mathcal O}_V(u)$ in the present setting. Hence the proof of Lemma~\ref{ii21} applies verbatim.

The proof of Lemma~\ref{caup3} uses only Lemma~\ref{ii21}, Lemma~\ref{AB relations-re}(1)--(2), and Definition~\ref{natrual-basic}(3). By Lemma~\ref{dot-caup-red}(1)--(2), the relations in Lemma~\ref{AB relations-re}(1)--(2) have their required analogues in the present setting, while Definitions~\ref{natrual-basic-red}(5)--(6) provide the corresponding relations required from Definition~\ref{natrual-basic}(3). Therefore, the proof of Lemma~\ref{caup3} carries over unchanged.

The proof of Lemma~\ref{lrcrossing} uses Lemmas~\ref{dotslidecrossing}, \ref{AB relations-re}, \ref{ii21}, \ref{caup3}, and \ref{adj}, together with Definition~\ref{natrual-basic}(2), Definition~\ref{natrual-basic1}(1)--(3), and Definition~\ref{C defn}(3). In the present setting, the required analogues of these ingredients are provided by Lemma~\ref{dot-caup-red}, Definitions~\ref{natrual-basic-red}(3)--(4), \ref{cross-red123}(1), \ref{def-ikmc}(2b), (5), and (1), respectively, together with the second equality in Definition~\ref{cross-red123}(2). Since all of these relations have already been established, the proof of Lemma~\ref{lrcrossing} carries over without modification.

Finally, the proof of equation~\eqref{redijdown} uses only the relations established above. Consequently, the same argument proves \eqref{redijdown} in the present setting.
\end{proof}

We shall therefore use these results freely in what follows.
\begin{Cor}For any $i, j\in I^\imath$ and any $\lambda\in X_\imath^+$, we have
    \begin{equation}\label{eeeq}
       \mathord{

}
    \end{equation}
\end{Cor}
\begin{proof}
We prove the equality in \eqref{eeeq} by comparing the two sides. Acting on $V\in\mathcal M_\lambda$, we first compute the right-hand side. Applying Definitions~\ref{cross-red123}(1) and~\ref{natrual-basic-red}(3)--(4) transforms the black diagram into a red one; then, using Definition~\ref{def-ikmc}(2b), the right-hand side of \eqref{eeeq} reduces to 
\begin{equation}\label{downarrowred}
\begin{cases}
%
& \text{if } j\neq i,i-1.
\end{cases}
\end{equation}
On the other hand, applying \eqref{redijdown} and Definition~\ref{cross-red123}(2) to the left‑hand side of \eqref{eeeq} yields the same expression as in \eqref{downarrowred}. Hence the two sides of \eqref{eeeq} are equal, as required.
\end{proof}

\begin{Lemma}
\label{lem:C-equalities}
When evaluated on any finitely generated object $V\in \mathcal M_\lambda$ with $\lambda\in X_\imath^+$, the equalities in Definition~\ref{C defn}(3) hold.
\end{Lemma}
\begin{proof}
To prove the equalities in Definition~\ref{C defn}(3), it suffices to establish the following identities for all $i,j,k,l\in I$:
\begin{equation}\label{capcup-cross}
.
\end{equation}
By Definition~\ref{C defn}(1), it suffices to prove the first equality in \eqref{capcup-cross}.

Now assume first that $i\neq -j$ or $l\neq k$. Then Definitions~\ref{cross-red123}(3)--(5) rewrite each crossing in the first equality as either a pair of dotted black vertical strands, a dotted cup–cap composite, or zero. Hence the desired identity follows immediately from Definition~\ref{C defn}(1) and Lemma~\ref{dot-caup-red}(1)--(2).

It remains to treat the case $i=-j$ and $k=l$.

If $i,-k\in I^\imath$, then the second equation in Definition~\ref{cross-red123}(2) gives
\begin{equation}\label{eqqq}
\begin{tikzpicture}[baseline=0, scale=0.8]
    \draw[-] (0.38,-.4) to (-0.38,.5);
    \draw[-] (-0.38,-.4) to (0.38,.5);
    \node at (0,0.05) {$\diamond$};
    \node at (-.4,-.62) {$\scriptstyle{k}$};
    \node at (.4,-.6) {$\scriptstyle{i}$};
    \node at (-.4,.7) {$\scriptstyle{i}$};
    \node at (.4,.72) {$\scriptstyle{k}$};
    \node at (0.5,0.09) {$\scriptstyle{\lambda}$};
\end{tikzpicture}
=
\begin{tikzpicture}[baseline=0, yscale=-1, scale=0.8]
    \draw[-,thick] (0.3,-.5) to (-0.3,.5);
    \draw[-,thick] (-0.2,-.2) to (0.2,.3);
    \draw[-,thick] (0.2,.3) to[out=50,in=180] (0.5,.5);
    \draw[-,thick] (0.5,.5) to[out=0,in=90] (0.8,-.5);
    \draw[-,thick] (-0.2,-.2) to[out=230,in=0] (-0.5,-.5);
    \draw[-,thick] (-0.5,-.5) to[out=180,in=-90] (-0.8,.5);
    \node at (0,0.05) {$\diamond$};
    \node at (-0.8,.65) {$\scriptstyle{k}$};
    \node at (0.28,-.62) {$\scriptstyle{i}$};
    \node at (-0.1,.65) {$\scriptstyle{i}$};
    \node at (0.78,-.62) {$\scriptstyle{k}$};
    \node at (1.05,.05) {$\scriptstyle{\lambda}$};
\end{tikzpicture}
=
\begin{tikzpicture}[baseline=0, scale=0.7]
    \draw[-,thick] (0.3,-.5) to (-0.3,.5);
    \draw[-,thick] (-0.2,-.2) to (0.2,.3);
    \draw[-,thick] (0.2,.3) to[out=50,in=180] (0.5,.5);
    \draw[-,thick] (0.5,.5) to[out=0,in=90] (0.8,-.5);
    \draw[-,thick] (-0.2,-.2) to[out=230,in=0] (-0.5,-.5);
    \draw[-,thick] (-0.5,-.5) to[out=180,in=-90] (-0.8,.5);
    \node at (0,0.05) {$\diamond$};
    \node at (-0.8,.7) {$\scriptstyle{i}$};
    \node at (0,.7) {$\scriptstyle{k}$};
    \node at (0.28,-.69) {$\scriptstyle{k}$};
    \node at (0.68,-.69) {$\scriptstyle{i}$};
    \node at (1.05,.05) {$\scriptstyle{\lambda}$};
\end{tikzpicture}.
\end{equation}
The second equality in \eqref{eqqq} is obtained by applying \eqref{eeeq} to the crossing on the right‑hand side and then using Definition~\ref{C defn}(1). Hence the first equality in \eqref{capcup-cross} follows from \eqref{eqqq} and Definition~\ref{C defn}(1).

If $-i,-k\in I^\imath$, the same argument works, with the first equation in Definition~\ref{cross-red123}(2) used in place of the second.

Finally, if $i,k\in I^\imath$ (respectively, $-i,k\in I^\imath$), the first equality in \eqref{capcup-cross} follows directly from the second (respectively, first) equation in Definition~\ref{cross-red123}(2) together with Definition~\ref{C defn}(1).

This proves the first equality in \eqref{capcup-cross}, and therefore the lemma.
\end{proof}

\begin{Defn} \label{ouinver}
    Let 
   $\clock(u) = u-\frac{1}{2}+\sum_{r\ge 0}
\frac{1}{u^r}
\begin{tikzpicture}[baseline = 1.25mm]
		\draw[-] (0,0.4) to[out=180,in=90] (-.2,0.2);
		\draw[-] (0.2,0.2) to[out=90,in=0] (0,.4);
		\draw[-] (-.2,0.2) to[out=-90,in=180] (0,0);
		\draw[-] (0,0) to[out=0,in=-90] (0.2,0.2);
		\node at (0.2,0.2) {$\scriptstyle\bullet$};
		\node at (0.45,0.2) {$\scriptstyle{x^r}$};
\end{tikzpicture}$ and let  $\clock_V(u)=\phi_V(\clock(u))$, where $\phi_V$ is the evaluation homomorphism defined in \eqref{evalu-na}. 
\end{Defn}
By Definitions~\ref{natrual-basic-red}, ~\ref{cross-red123}, 
the element $\clock(u)$  is well defined in the present setting, and is the analogue of the corresponding element in the affine Brauer category. 
\iffalse\old{The proof of Lemma~\ref{caup3} depends only on the relations for  ${\mathcal O}_V(u)$ stated in Lemma~\ref{ou}(2). By  Lemma~\ref{equa-dotte-generatetilde}(2), 
the  same relations hold for  $\tilde{\mathcal O}_V(u)$ in the present setting. Therefore, the argument proving Lemma~\ref{caup3} applies without change, and Lemma~\ref{caup3} remains valid in the present setting. We shall therefore use it freely in what follows. }
\fi

\begin{Prop}\label{bub=bubred}
For any finitely generated object $V\in \mathcal M_\lambda$, one has
$ \tilde{\mathcal O}_V(u)=\clock_V(u)$.
\end{Prop}
\begin{proof} In all sums over $i\in I$ occurring in this proof,  only those $i$ with $\varepsilon_i(V)\neq 0$ are included. 
For any $i\in I$, take
$
f_1(u)=u m_{V,i}(u)$  and $
f_2(u)=(u-i)^k$
in \eqref{key-coprim} for a   sufficiently large integer $k$. Then there exist
polynomials $g_i(u),h_i(u)\in \tilde Z_{\lambda,V}[u]$ satisfying
\eqref{key-coprim}. By Definition~\ref{evalu-bub}(5) and
\eqref{key-coprim}, we have
\begin{equation}
\label{eq:decomp}
\tilde{\mathcal O }_V(u)
=
\frac{n_V(u)\,u\,g_i(u)}{(u-i)^{\varepsilon_i(V)}}
+
\frac{n_V(u)(u-i)^{k-\varepsilon_i(V)}h_i(u)}{m_{V,i}(u)}.
\end{equation}
Define
\begin{equation}\label{rv-red}
R_V(u)
:=
\tilde{\mathcal O}_V(u)
-
\sum_{i\in I}
\frac{n_V(u)\,u\,g_i(u)}{(u-i)^{\varepsilon_i(V)}}.
\end{equation}
Using \eqref{eq:decomp} for any fixed $i\in I$, we see that
$R_V(u)m_V(u)\in \tilde Z_{\lambda, V}[u]$ and that it is divisible by
$(u-i)^{\varepsilon_i(V)}$. Since only finitely many $i\in I$ satisfy
$\varepsilon_i(V)\neq 0$, and since the polynomials
$(u-i)^{\varepsilon_i(V)}$ are pairwise coprime for distinct $i$, it follows
that $R_V(u)m_V(u)$ is divisible by $m_V(u)$. As $m_V(u)$ is monic, it is
not a zero divisor, and hence
$R_V(u)\in \tilde Z_{\lambda, V}[u]$.

Moreover, since $\widetilde{\mathcal O}_V(u)=n_V(u)/m_V(u)$ and
$I=1/2+\mathbb Z$, we have $m_V(0)\neq 0$. Thus
 $\widetilde{\mathcal O}_V(0)$ is well-defined. On the other hand,
$[\widetilde{\mathcal O}_V(u)]_{u^0}$ denotes the coefficient of $u^0$ in
the expansion at $u=\infty$, and should not be confused with the value
at $u=0$.
We have 
\begin{equation}\label{eq:coeff}
\begin{aligned}
\bigl[\widetilde{\mathcal O}_V(u)\bigr]_{u^{-r}}
&=
\Biggl[
\sum_{i\in I}
\frac{n_V(u)g_i(u)}{(u-i)^{\varepsilon_i(V)}}
\Biggr]_{u^{-r-1}}
\qquad \text{for all } r\ge 1, \\
\widetilde{\mathcal O}_V(0)
&=
R_V(0)
=
\bigl[\widetilde{\mathcal O }_V(u)\bigr]_{u^0}
-
\Biggl[
\sum_{i\in I}
\frac{n_V(u)g_i(u)}{(u-i)^{\varepsilon_i(V)}}
\Biggr]_{u^{-1}} .
\end{aligned}
\end{equation}
Here the first equality involving $\widetilde{\mathcal O}_V(0)$ follows from \eqref{rv-red} by setting $u=0$.
On the other hand, Definition~\ref{ouinver} gives
\begin{equation}\label{eq:clock}
\bigl[\clock_V(u)\bigr]_{u^{-r}}
=
\sum_{i\in I}
\begin{tikzpicture}[baseline = 1.25mm]
	\draw[-] (0,0.4) to[out=180,in=90] (-.2,0.2);
	\draw[-] (0.2,0.2) to[out=90,in=0] (0,.4);
	\draw[-] (-.2,0.2) to[out=-90,in=180] (0,0);
	\draw[-] (0,0) to[out=0,in=-90] (0.2,0.2);
	\node at (0.2,0.2) {$\scriptstyle\bullet$};
    \node at (0.28,-0.05) {$\scriptstyle -i$};
    \node at (-0.35,0.2) {$\scriptstyle i$};
	\node at (0.45,0.2) {$\scriptstyle{x^r}$};
\end{tikzpicture}
-\delta_{r,0}\,\frac12
\qquad \text{for all $r\ge 0$}.
\end{equation}
Furthermore, the leading term of $\clock_V(u)$ is $u$. By
Lemma~\ref{equa-dotte-generatetilde}(1) and
Definition~\ref{evalu-bub}(5), the leading term of
$\widetilde{\mathcal O}_V(u)$ is also $u$. Therefore, to prove that
$
\widetilde{\mathcal O}_V(u)=\clock_V(u)$,
it suffices to prove that
\begin{equation}\label{evaou-ex}
\widetilde{\mathcal O }_V(0)=-\frac12,
\qquad\text{and}\qquad
\begin{tikzpicture}[baseline = 1.25mm]
	\draw[-] (0,0.4) to[out=180,in=90] (-.2,0.2);
	\draw[-] (0.2,0.2) to[out=90,in=0] (0,.4);
	\draw[-] (-.2,0.2) to[out=-90,in=180] (0,0);
	\draw[-] (0,0) to[out=0,in=-90] (0.2,0.2);
	\node at (0.2,0.2) {$\scriptstyle\bullet$};
    \node at (0.28,-0.05) {$\scriptstyle -i$};
    \node at (-0.35,0.2) {$\scriptstyle i$};
	\node at (0.45,0.2) {$\scriptstyle{x^r}$};
\end{tikzpicture}
=
\Biggl[
\frac{n_V(u)g_i(u)}{(u-i)^{\varepsilon_i(V)}}
\Biggr]_{u^{-r-1}}
\end{equation}
for every $r\ge 0$ and every $i\in I$.

Indeed, by Lemma~\ref{simplered}(3), we have
$
\overline{\widetilde{\mathcal O}_V(0)}
=
\widetilde{\mathcal O}_L(0)
=
-\frac12$ for every simple object
$L\in\mathcal M_\lambda$.
Thus $\widetilde{\mathcal O}_V(0)+\frac12$ lies in the maximal ideal of the local ring $\tilde Z_{\lambda, V}$, so $\widetilde{\mathcal O}_V(0)-\frac12$ is a unit. Evaluating Lemma~\ref{equa-dotte-generatetilde}(1) at $u=0$ gives
$
\widetilde{\mathcal O}_V(0)^2=\frac14$.
It follows that
$
\widetilde{\mathcal O}_V(0)=-\frac12$.

Now we prove the second equality in \eqref{evaou-ex}.  Suppose that $i\in I^\imath$. Then
\begin{equation}
\label{equ:bublleplus}
\begin{aligned}
\begin{tikzpicture}[baseline = 1.25mm]
	\draw[-] (0,0.4) to[out=180,in=90] (-.2,0.2);
	\draw[-] (0.2,0.2) to[out=90,in=0] (0,.4);
	\draw[-] (-.2,0.2) to[out=-90,in=180] (0,0);
	\draw[-] (0,0) to[out=0,in=-90] (0.2,0.2);
	\node at (0.2,0.2) {$\scriptstyle\bullet$};
    \node at (0.28,-0.05) {$\scriptstyle i$};
    \node at (-0.45,0.2) {$\scriptstyle -i$};
	\node at (0.45,0.2) {$\scriptstyle{x^r}$};
\end{tikzpicture}
&=
\begin{tikzpicture}[baseline = 0]
  \draw[->,thick,darkred] (0.2,0.2) to[out=90,in=0] (0,.4);
  \draw[-,thick,darkred] (0,0.4) to[out=180,in=90] (-.2,0.2);
  \draw[-,thick,darkred] (-.2,0.2) to[out=-90,in=180] (0,0);
  \draw[-,thick,darkred] (0,0) to[out=0,in=-90] (0.2,0.2);
  \node at (0,-.16) {$\scriptstyle{i}$};
  \node at (-0.3,0.2) {$\scriptstyle{\lambda}$};
  \node at (0.2,0.2) {$\scriptstyle\bullet$};
  \node at (1.2,0.2)
  {$\scriptstyle{\frac{n_{V,i}(x)x^r}{xm_{V,i}(x)}}$};
\end{tikzpicture}
=
\begin{tikzpicture}[baseline = 0]
  \draw[->,thick,darkred] (0.2,0.2) to[out=90,in=0] (0,.4);
  \draw[-,thick,darkred] (0,0.4) to[out=180,in=90] (-.2,0.2);
  \draw[-,thick,darkred] (-.2,0.2) to[out=-90,in=180] (0,0);
  \draw[-,thick,darkred] (0,0) to[out=0,in=-90] (0.2,0.2);
  \node at (0,-.16) {$\scriptstyle{i}$};
  \node at (-0.3,0.2) {$\scriptstyle{\lambda}$};
  \node at (0.2,0.2) {$\scriptstyle\bullet$};
  \node at (1.4,0.2) {$\scriptstyle{n_{V,i}(x)g_i(x)x^r}$};
\end{tikzpicture}
\\
&=
\left[
\begin{tikzpicture}[baseline = 0]
  \draw[->,thick,darkred] (0.2,0.2) to[out=90,in=0] (0,.4);
  \draw[-,thick,darkred] (0,0.4) to[out=180,in=90] (-.2,0.2);
  \draw[-,thick,darkred] (-.2,0.2) to[out=-90,in=180] (0,0);
  \draw[-,thick,darkred] (0,0) to[out=0,in=-90] (0.2,0.2);
  \node at (0,-.1) {$\scriptstyle{i}$};
  \node at (-0.3,0.2) {$\scriptstyle{\lambda}$};
  \node [cplus,scale=0.7] at(0.2,0.2) {$-$};
  \node at (1.5,0.2) {$\scriptstyle{n_{V,i}(u)g_i(u)}$};
\end{tikzpicture}
\right]_{u^{-r-1}}
=
\Biggl[
\frac{n_V(u)g_i(u)}{(u-i)^{\varepsilon_i(V)}}
\Biggr]_{u^{-r-1}},
\end{aligned}
\end{equation}
where the last equality follows from Lemma~\ref{bubb-red}.  Using Lemma~\ref{bubb-red} and Lemma~\ref{bubb-inftyi}(3), we obtain
\begin{equation}\label{plumin}
\frac{4(u+1/2)^{\delta_{i,1/2}}}{4u^2-1}\,
\frac{n_{V,i}(-u)}{m_{V,i}(-u)}
=
(-1)^{-\langle{}^\theta\alpha_i^\vee,\lambda\rangle
-\delta_{i,1/2}+1}
\frac{m_{V,-i}(u)}{n_{V,-i}(u)}.
\end{equation}
Then Lemma~\ref{caup3}(1), together with \eqref{plumin}, gives
\[
.
\]
Consequently,  after evaluating  on $V$, we have
\begin{equation}\label{equ:bubbleplus1}
\begin{aligned}
%
\right]_{u^{-r-1}}
\\
&\overset{Lem~\ref{bubb-red}}=
\Biggl[
\frac{n_V(u)g_{-i}(u)}{(u+i)^{\varepsilon_{-i}(V)}}
\Biggr]_{u^{-r-1}}.
\end{aligned}
\end{equation}
The second equality in \eqref{evaou-ex} now follows from
\eqref{equ:bublleplus} and \eqref{equ:bubbleplus1}. This completes the
proof.
\end{proof}

\begin{Lemma}\label{right cuaps - down cross-inverse}
When evaluated on any finitely generated object
$V\in \mathcal M_\lambda$ with $\lambda\in X_\imath^+$, the two
equalities in Definition~\ref{C defn}(2) hold.
\end{Lemma}

\begin{proof}
It suffices to verify that
\begin{equation} \label{caup-red-second}
\mathord{
 }
\qquad \text{for all $i,j\in I$.}
\end{equation}

Suppose first that $i\neq -j$. It follows from  Definition~\ref{natrual-basic-red}(7)--(8)  that the right-hand sides of the two equalities in~\eqref{caup-red-second} are both zero. Moreover, Definition~\ref{cross-red123}(5), (9), and~(10) imply that the corresponding  left-hand sides are also  zero. Hence~\eqref{caup-red-second} holds.

Now suppose that $i=-j$. By Definition~\ref{C defn}(1) and (3), it
suffices to prove \eqref{caup-red-second} for $i\in I^\imath$.
Set
\begin{equation}\label{extra-a}
A :=
%
,
\end{equation}
where $p(u)$ is the polynomial used in the computation leading to  \eqref{capcross-red}.

In the computation leading to \eqref{capcross-red}, we worked in the
affine Brauer category. This allowed us to use Lemma~\ref{usefuel-equa}(3)
and the first equation in Definition~\ref{natrual-basic1}(4). In the
present situation, however, we are working in $\mathfrak                        U_+ ^\imath$ rather
than in the affine Brauer category. Therefore, in deriving the fourth
equality in \eqref{capcross-red}, Lemma~\ref{usefuel-equa}(3) must be
replaced by Corollary~\ref{caup}(2) (see Lemma~\ref{dot-caup-red}), and the first equation in
Definition~\ref{natrual-basic1}(4) must be replaced by
Lemma~\ref{bubb-red}. This produces the extra term $A$, defined in
\eqref{extra-a}, and hence an extra term in each subsequent equality,
starting from the fourth   one. Thus we obtain
\begin{equation}\label{A-red}
\mathord{
} + A .
\end{equation}
Definition~\ref{def-ikmc}(5a) gives $A=0$. This proves the first equality
in \eqref{caup-red-second}, since
\(
%
\)
is invertible.
It remains to prove the second equality in \eqref{caup-red-second}. Define
\begin{equation}\label{extr-b}
B :=
%
,
\end{equation}
where $s(x)=\frac{4x}{b(x)}\frac{n_{V,i}(x)}{m_{V,i}(x)}$  given in \eqref{capcross-red1}. %

A similar argument applies to the computation leading to
\eqref{capcross-red1}. Namely, Lemma~\ref{usefuel-equa}(4) is replaced by
Corollary~\ref{caup}(3), and the second equation in
Definition~\ref{natrual-basic1}(4) is replaced by Lemma~\ref{bubb-red}.
This introduces the extra term $B$ in every step starting from the fourth 
equality, while the resulting identity remains valid. Thus we obtain
\begin{equation}\label{B-red}
    \mathord{
%
} + B .
\end{equation}
Definition~\ref{def-ikmc}(5b) gives $B=0$. This proves the second equality
in \eqref{caup-red-second}, since
$
%
$
is invertible. The proof is complete.\end{proof}

\begin{Lemma}\label{crosss-inverse}
When evaluated on any finitely generated object
$V\in \mathcal M_\lambda$ with $\lambda\in X_\imath^+$, the relation  in Definition~\ref{AB defn}(1) holds.
\end{Lemma}

\begin{proof}
It suffices to prove that
\begin{equation}
\label{crosss-inverse1}
%
\qquad \text{for all $i,j,k,l\in I$.}
\end{equation}

First suppose that $i\neq k$. Then $x-i$ is invertible on the strand
labeled by $k$. Let $b\gg 0$. Since the relations in
Definition~\ref{C defn} needed for Corollary~\ref{caup1} have been verified in the present category, we may apply
Corollary~\ref{caup1} to obtain
\[
%
=0.
\]
Hence \eqref{crosss-inverse1} holds in this case. The case $l\neq j$ is
proved in the same way.

It remains to consider the case $k=i$ and $l=j$.
Now suppose that $k\neq -l$. Definition~\ref{cross-red123}(5) gives
\[
%
.
\]
If $k,-l\in I^\imath$, then Lemma~\ref{lrcrossing} implies the first equality in 
\eqref{cross-2-black} as follows. 
\begin{equation}\label{cross-2-black} 
%
.
\end{equation}
For the second equality, we first use the defining relations of $\mathfrak U^\imath_+$ to slide  the red dots in the first red term of \eqref{cross-2-black} through the red crossings to the bottom of the diagram. We then apply Definition~\ref{def-ikmc}(7) to simplify the resulting composition of the two red crossings to two parallel vertical strands. This produces the first term after the second equality  in \eqref{cross-2-black}. 
{The remaining  cases, namely $k,l\in I^\imath$, $-k,l\in I^\imath$ and  $-k, -l\in I^\imath$ are
proved similarly}. 

It remains to treat the case $k=-l$. Suppose first that
$-k\in I^\imath{\setminus\{1/2\}}$,
and define
\begin{equation} \label{extra1}
A=
,
\end{equation}
where $c_V(x)$ is defined in the proof of
Lemma~\ref{bicross1 downred}. An argument analogous to the one used
above \eqref{A-red} shows that the extra term $A$ may be inserted starting 
with the second equality in the computation of \eqref{redd1}. This gives
\[
\mathord{
%
}
+A.
\]
Then Definition~\ref{def-ikmc}(6b) gives $A=0$, which is equivalent to
\eqref{crosss-inverse1} under the present assumption.

Next suppose that $k=-l$ and
$k\in I^\imath{\setminus\{\frac{1}{2}\}}$. Define
\[
B=
%
,
\]
where $d_V(x)$ is defined in the proof of
Lemma~\ref{bicross1 downred}. An argument analogous to the one used
above \eqref{B-red} shows that the extra term $B$ may be inserted from
the second  equality onward in the computation of \eqref{redd2}. This gives
\[
\mathord{
%
}
+B.
\]
Then Definition~\ref{def-ikmc}(6a) gives $B=0$, which is again equivalent
to \eqref{crosss-inverse1}. 

The remaining cases, namely, $k=-l$ with $k=\pm\frac12$, are checked in the same way. The only difference is that, in the case $k=-\frac12$, one uses~\eqref{redd1} with $i=\frac12$, whereas in the case $k=\frac12$, one uses~\eqref{redd2} with $i=\frac12$. This proves~\eqref{crosss-inverse1}, and hence the lemma.
\end{proof}

\subsection{Braid relations}
\begin{Lemma}\label{zero-localijk} Suppose $i,i',j,j',k, k'\in I$ and that at least one of the following conditions holds: $
i\neq k'$, $ j\neq j'$ or $ k\neq i'$. When evaluated on any finitely generated object
$V\in \mathcal M_\lambda$ with $\lambda\in X_\imath^+$, 
\tikzset{strand/.style={-,thick}}
 \[
.
	\]\end{Lemma}
  
	\begin{proof} The proof is analogous to that of Lemma~\ref{crosss-inverse},  except that
Corollary~\ref{slidecro}(1)--(3) is used in place of
Corollary~\ref{caup1}(1)--(2).   
   \end{proof}
    From this point  to the end of this section, our goal is   to prove the braiding relation \eqref{braid-local}  for arbitrary  $i,j, k\in I$.
    \tikzset{strand/.style={-,thick}}
     \begin{equation} \label{braid-local} %
. \end{equation}

\begin{Lemma}\label{braid-local-r1}
Suppose that $i,j,k\in I^\imath$. Then, when evaluated on any
finitely generated object $V\in \mathcal M_\lambda$, the relation
\eqref{braid-local} holds.
\end{Lemma}

\begin{proof}
We prove the result by treating separately the following three cases:
\begin{multicols}{2}
\item[(a)] $i=j=k$,
\item[(b)] exactly two of $i,j,k$ coincide,
\item[(c)] $i,j,k$ are pairwise distinct.
\end{multicols}

It follows from Definition~\ref{def-ikmc}(3a) and
Definition~\ref{cross-red123}(1) that
\begin{equation}
\label{equ:reddotmovecross}
  
\end{equation}
where $g=\ell_{-+}^{12}$. 

By \eqref{equ:reddotmovecross}, we may replace each occurrence of the black crossing by a linear combination of red
diagrams. Some of the resulting red diagrams
contain red horizontal strands labeled by $g$.  We then use the red sliding relations in Definition~\ref{def-ikmc}(3a)-(3b) to move these red horizontal strands upward
as far as possible. We have  %This gives the required identity in case~(a), as
%follows:
\begin{equation}\label{braid-iiired} \begin{aligned}
  %
, \end{equation}
$g'=\ell_{-+}^{13}$, and $g''=\ell_{-+}^{23}$. 
Here, the sixth equality follows from Definition~\ref{def-ikmc}(3a),(3b), and the seventh equality follows from Definition~\ref{def-ikmc}(3a).
Similarly, we obtain the following identity:

\begin{equation}\label{braid-iiired2}
%
+A  \end{equation}
     where $A$ was defined in \eqref{AB-red}. 
Combining
Definition~\ref{def-ikmc}(3c), \eqref{braid-iiired}, and
\eqref{braid-iiired2}, we obtain \eqref{braid-local} in case~(a).

The proofs of cases~(b) and~(c) are similar. The only difference is that
one uses the second or third formula defining
$
%
$
in Definition~\ref{cross-red123}(1), rather than the first formula used
in case~(a).
For example, in case~(c), we have
\[
%
.
\]
Here $g=g_{k,j}(x_1,x_2)$,
$g'=g_{k,i}(x_1,x_3)$, and 
$g''=g_{j,i}(x_2,x_3)$,
where
\[
g_{h,l}(x,y)=
\begin{cases}
(x-y)^{-1}, & \text{if } l=h+1,\\
1+(x-y)^{-1}, & \text{if } l\neq h+1.
\end{cases}
\]
We omit the details since the proofs are similar.
\end{proof}

  \begin{Cor}\label{braid-local-r3}  When evaluated  on any finitely generated object $V\in \mathcal M_\lambda$,  relation \eqref{braid-local} holds for each of the four triples $(-i, j, k), (j,i,-k), (-j,i,-k), (-j,-i,k)$, provided that all three entries of the triple under consideration lie in $I^{\imath}$.  
   \end{Cor}
    \begin{proof} 
 Suppose $-i,j, k\in I^{\imath}$. We have:
\begin{equation}
    \label{empty4}

   \end{equation}
For the local crossings on the right-hand side of  \eqref{empty4}, we can now apply the case
 $i, j,k\in I^{\imath}$ in \eqref{braid-local}. Then, using Definition~\ref{C defn}(1)-(3), we obtain \eqref{braid-local} for $-i, j,k\in I^{\imath}$. Finally, the remaining three  cases are analogous, and each case can also 
  be reduced to the case $i, j, k\in I^\imath$ as above. 
 \end{proof}

A similar argument (using the same technique as in \eqref{empty4}) shows that if \eqref{braid-local} holds for $-i, j, -k \in I^\imath$, then it also holds for $i, -j, k \in I^\imath$. Therefore, for the remainder of this section, we may assume that $-i, j, -k\in I^\imath$.
   
 \begin{Lemma}\label{braid-local-r21} Suppose that  $-i,j,-k\in I^\imath$. If at least one of  $j\neq -k$ and $j\neq -i$ holds, then,  when evaluated  on any finitely generated object $V\in\mathcal  M_\lambda$, the relation \eqref{braid-local} holds. 
    \end{Lemma}
    \begin{proof}
Since the proofs of the two cases $j\neq -k$ and $j\neq -i$ are completely analogous, we prove only the case $j\neq -i$.    
    Let $g=(x_1-x_2)^{-1}$. A direct computation shows that 
    % 在导言区添加以下定义（若已有则无需重复）

\begin{align*}
,
\end{align*}
where the first equality follows from Lemma~\ref{crosss-inverse}
together with Definition~\ref{cross-red123}(5).
Note that the above  equality is equivalent to the following identity:     
    \begin{equation}\label{braid-r4}  %
. \end{equation}
As explained in the proof of Corollary~\ref{caup1},
each identity in Corollary~\ref{caup1} is equivalent to the
corresponding identity obtained by replacing every occurrence
of $\ominus$ in the diagrams with $\bullet$.
These latter identities have already been verified in the
present category. Therefore, we use them to obtain    
     %other hand, by Lemma \ref{equ:dotmovecrosses}, we have 
      \begin{equation}\label{braid-r5}  %
\]
    where $h=1+g^{-2}$.  Since  we assume $-i, j\in I^\imath$, $h$ is invertible. Consequently,  the above identity is equivalent to \eqref{braid-local}.
    \end{proof}
    \begin{Lemma}\label{braid-local-r211} Suppose that $j= -k= -i=1/2$. Then,  when evaluated  on any finitely generated object $V\in \mathcal M_\lambda$, the relation \eqref{braid-local} holds. 
    \end{Lemma}
    \begin{proof} 
   By Lemma~\ref{dot-caup-red}, Corollary~\ref{def6}, and Lemmas~\ref{right cuaps - down cross-inverse}--\ref{zero-localijk}, all defining
relations of the affine Brauer category used so far hold in the present category, except for the braid
relation in Definition~\ref{AB defn}(2).   
   Therefore, the use of that  braiding relation in Section~6 is not yet justified here. Nevertheless, all computations for proving \eqref{zuizho} in Section~6 remain valid without this relation except for the computation of $Z_{10}$ in Lemma~\ref{parth}. More precisely, %In Lemma~\ref{parth}, 
    this term should be replaced by $\tilde Z_{10}$, defined as follows: %只有Lemma~\ref{parth}中的Z_9要替换
    \begin{equation}\label{Z9tilde} \tilde Z_{10}:=Z_{10} -
.   
    \end{equation}
    Here $h_1=(1-x_1-x_2)^{-1}$, $h'=-\frac{(x_1-x_2)^2}{(1+x_2-x_1) h_1}$, and $h_3'=(1-x_3+x_1)^{-1}$. We remark that the three vertices in the  top row (and similarly in the  bottom row)  are labeled from left to right by $-1/2, 1/2, -1/2$ in the black diagram, and by $1/2, 1/2, 1/2$ in the red diagram. Therefore, by the same argument as in the proof of \eqref{zuizho}, we obtain    
  \begin{equation}\label{Pi=2}  \begin{aligned}
 %
+\tilde Z_{10}-Z_{10}.
\end{aligned}  \end{equation}
Then, Definition~\ref{def-ikmc1} gives   $\tilde Z_{10}=Z_{10}$. Since $h_1, h'$ and $h_3'$ are invertible, 
\eqref{Z9tilde}  implies \eqref{braid-local}.    \end{proof}

\begin{Lemma}\label{braid-local-rneq} Suppose that $j= -k= -i\neq 1/2$ with $j\in I^\imath$. Then,  when evaluated  on any finitely  generated object $V\in \mathcal M_\lambda$, the relation \eqref{braid-local} holds. 
    \end{Lemma}
    \begin{proof}
 As in the proof of Lemma~\ref{braid-local-r211},  all computations in Section~6.4 remain valid without using the braid relation except for the computation of $Z_{10}$ in Lemma~\ref{pathj}. In the present setting, this term has to  be replaced by $\tilde Z_{10}$, where 
  \begin{equation}
\label{zuoT}\tilde Z_{10}=Z_{10}-
\end{equation}
 We remark that the three vertices in the  top row (and similarly in the  bottom row)  are labeled from left to right by $-j, j, -j$ in the black diagram, and by $j, j, j$ in the   red diagram below.
Combining \eqref{zuizz} with the explicit  expressions for $Z_{4}$, $Z_{5}$, $Z_{7}$, $Z_{8}$, $\tilde Z_{10}$, and $Z_{11}$ given in  Proposition~\ref{emm3}, Proposition~\ref{emm6}, Lemma~\ref{emm9j}, and Lemma~\ref{pathj},   we obtain \begin{equation}
    \label{zuizhoz}
\begin{aligned}
&%
-(\tilde Z_{10}-Z_{10}),
\end{aligned}\end{equation}
where the rational functions $T'_i$, $1\leq i\leq5$, are defined
in Section~9. By
Lemma~\ref{lem:primed-rational-identities}, 
$ T'_1=T'_2=T'_3=T'_4=T'_5=0$.
Therefore, \eqref{equ:inhomotype} gives $\widetilde Z_{10}=Z_{10}$. Since
$h_1$, $h'$, and $h'_3$ are invertible, \eqref{zuoT} implies
\eqref{braid-local}.
 \end{proof}
%\comment{ we have define the notion of affine Brauer categorification and maybe better to use this in the following?}
\begin{Theorem}\label{thm:mainthmaffineb1} Suppose that  a locally Schurian category  $\mathcal M$ 
 carries the structure of a nilpotent
$2$-representation of $\mathfrak U^\imath_+$. Then $\mathcal M$ admits
an affine Brauer categorification  for which
the dot spectrum is contained in $\frac12+\mathbb Z$.
\end{Theorem} 

\begin{proof} Suppose that $\mathcal M$ is a nilpotent $2$-representation  of $\mathfrak U^\imath_+$. By   Lemma~\ref{dot-caup-red}, Corollary~\ref{def6}, Lemma \ref{lem:C-equalities} and  Lemmas~\ref{right cuaps - down cross-inverse}--\ref{braid-local-rneq}, all defining relations for $\AB$ have been verified on  $\mathcal M$. Therefore, 
$\mathcal M$ admits the structure of a module category over the affine Brauer category. By Definition~\ref{gen-red1}, the generating endofunctor is 
$$E=\bigoplus_{i\in I^\imath} (E_i\bigoplus F_i),$$
and by Definition~\ref{natrual-basic-red} the dot has generalized eigenvalues $i$ and $-i$ on the summands $E_i$ and $F_i$, respectively. Hence its spectrum is contained in $\frac12+\mathbb Z$. 
\end{proof}

\subsection{Proof of Theorem~\ref{main111}} 
Part~(1) follows from Theorem~\ref{thm:mainthmaffineb}.
For part~(2), let $\mathcal M$ be a nilpotent
$2$-subrepresentation of the ambient locally Schurian
$2$-representation $\mathcal C$. We perform the constructions of Section~7 in the ambient locally
Schurian \(2\)-representation \(\mathcal C\) and restrict the
resulting functors and natural transformations to the nilpotent
\(2\)-subrepresentation \(\mathcal M\).
Since $\mathcal M$ is a
$2$-subrepresentation, all generating $1$-morphisms and
$2$-morphisms of $\mathfrak U^\imath_+$ restrict to
$\mathcal M$. The nilpotence assumption guarantees that
\[
E=\bigoplus_{i\in I^\imath}(E_i\oplus F_i)
\]
is well defined on every object of $\mathcal M$, and that all
rational-function expressions occurring in the constructions of
Section~7 are well defined on $\mathcal M$.

The relations established in Lemma~\ref{dot-caup-red},
Corollary~\ref{def6},  Lemma \ref{lem:C-equalities} and Lemmas~\ref{right cuaps - down cross-inverse}--\ref{braid-local-rneq}  hold in the ambient
locally Schurian $2$-representation $\mathcal C$, and hence
restrict to $\mathcal M$. Thus the construction in the proof of
Theorem~\ref{thm:mainthmaffineb1} endows $\mathcal M$ with an affine Brauer
categorification. By Definition~\ref{natrual-basic-red}, the dot has generalized
eigenvalue $i$ on $E_i$ and generalized eigenvalue $-i$ on
$F_i$, so its spectrum is contained in
$\frac12+\mathbb Z$.\qed

% Theorem~\ref{main111} follows from Theorems~\ref{thm:mainthmaffineb} and \ref{thm:mainthmaffineb1}.

\begin{rem}
    The two statements  in Theorem~\ref{main111} are
asymmetric. In part~(1), the affine Brauer action is required to have
dot spectrum exactly $\frac12+\mathbb Z$, whereas part~(2) produces an
affine Brauer action whose dot spectrum is only known to be contained in
$\frac12+\mathbb Z$. Thus Theorem~\ref{main111} does not assert that the two
constructions are mutually inverse, even in the locally finite Abelian
setting. In particular, we do not prove that, for every $i\in \frac12+\mathbb Z$, there is a simple object $L_i$  such that 
$E_iL_i\neq0$.\end{rem}

\section{The cyclotomic Brauer-$\imath$-Kac--Moody isomorphism}
In this section we apply the two constructions of Theorem A to cyclotomic quotients. Let $\mathcal{CB}_{\mathbf u}$ be a $\mathbf u$-admissible cyclotomic Brauer category with half-integral parameters. We associate to $\mathbf u$ a dominant $\imath$-weight $\kappa$ and a cyclotomic quotient $\mathcal Q(\kappa)$ of the principal 2-representation of $\mathfrak U^{\imath}_{+}$. Our goal is to identify the locally unital path algebra of $\mathcal{CB}_{\mathbf u}$ with the endomorphism algebra generated by the objects $E_{\mathbf i}1_{\kappa}$ in $\mathcal Q(\kappa)$. This identification is the main representation-theoretic consequence of the paper: it realizes cyclotomic Brauer algebras as endomorphism algebras in $\imath$-Kac--Moody categorification and transports the grading of the 2-category to them. In this sense it supplies the categorical and graded foundation for a Brauer--codieal counterpart of the type $A$ Brundan-Kleshchev-Ariki picture.

%The following elementary observation will be useful.

\begin{Prop}\label{gen123}
Let $\mathcal C$ be a strict $\Bbbk$-linear monoidal category,
let $\mathcal I$ be a left tensor ideal of $\mathcal C$, and set
$\mathcal B=\mathcal C/\mathcal I$. If the path algebra $A$ of
$\mathcal B$ is locally finite-dimensional, then
$\operatorname{lfdmod}\text{-}A$ is naturally a left module category
over $\mathcal C^{\operatorname{op},\operatorname{rev}}$.
Equivalently, it is a right module category over
$\mathcal C^{\operatorname{op}}$.
\end{Prop}

\begin{proof}
Since $\mathcal I$ is a left tensor ideal, left tensoring by
$X\in\mathcal C$ descends to an endofunctor
$
L_X\colon\mathcal B\longrightarrow\mathcal B$.
Under the standard equivalence
\[
\operatorname{lfdmod}\text{-}A
\simeq
\operatorname{Fun}_{\Bbbk}
\bigl(\mathcal B^{\operatorname{op}},
      \operatorname{Vect}^{\mathrm{fd}}_{\Bbbk}\bigr),
\]
the action of \(X\) is given by precomposition with
\(L_X^{\operatorname{op}}\). Since
$
L_{X\otimes Y}=L_X\circ L_Y$ and $
L_{\mathbf 1}=\operatorname{Id}_{\mathcal B}$,
these endofunctors define a left
\(\mathcal C^{\operatorname{op},\operatorname{rev}}\)-module
category structure.
\end{proof}

\subsection{Cyclotomic Brauer category} Fix  $\mathbf u=(u_1,\ldots, u_a)\in \Bbbk^a$ for some $a> 0$.   Choose 
$(\omega_i)_{i\in \mathbb N}\in \Bbbk^{\mathbb N}$ such that the $\mathbf u$-admissible condition in ~\cite[Proposition~3.1]{AMR-cyc} holds, namely 
\begin{equation}\label{uadamr} \sum_{i\ge 0}\frac{\omega_i}{u^i}+u-\frac{1}{2}= (u-\frac{1}{2}(-1)^a)\prod_{i=1}^a \frac{u+u_i}{u-u_i}.\end{equation}

\begin{Defn}\cite{RS-cyc}\label{822}
 The  cyclotomic Brauer category $\mathcal {CB}_{\mathbf u}$
associated with  $\mathbf u=(u_1,\ldots, u_a)$ is defined as the quotient category of $\mathcal {AB}$ by
the left tensor ideal generated by 
$\begin{tikzpicture}[baseline = -1mm]
 	\draw[-] (-0.18,-.4) to (-0.18,.4);
     \node at (-0.18,0) {$\scriptstyle\bullet$};
     \node at (0.35,0) {$\scriptstyle m(x)$};
 	%\draw[-,darkg,thick] (.9,.4) to (.9,-.4);
     %\node at (.9,-.55) {$\darkg\scriptstyle{V}$};
\end{tikzpicture}$
and
$ \begin{tikzpicture}[baseline = 1.25mm]
		\draw[-] (0,0.4) to[out=180,in=90] (-.2,0.2);
		\draw[-] (0.2,0.2) to[out=90,in=0] (0,.4);
		\draw[-] (-.2,0.2) to[out=-90,in=180] (0,0);
		\draw[-] (0,0) to[out=0,in=-90] (0.2,0.2);
		\node at (0.2,0.2) {$\scriptstyle\bullet$};
		\node at (0.45,0.2) {$\scriptstyle{x^r}$};
\end{tikzpicture}-\omega_r$ for $r\ge 0$, where 
$m(u)=\prod_{i=1}^a(u-u_i)$.
\end{Defn}

By the basis theorem given in \cite[Theorem~C]{RS-cyc} under the $\mathbf u$-admissible condition, $m(u)$ is the minimal polynomial of the generating morphism dot $x$.   Let $B_\mathbf u$ be the locally unital algebra associated to the cyclotomic Brauer category 
$\mathcal {CB}_{\mathbf u}$. 
That is,
$$B_\mathbf u=\bigoplus_{i,j\in \mathbb N}\Hom_{\mathcal {CB}_\mathbf u}(i,j).$$
 From now on, we write $B_\mathbf u$ simply as $B$. 
Since \(\mathcal{CB}_{\mathbf u}\) is locally finite-dimensional in the
setting considered here, and since it is defined as a quotient of
\(\mathcal{AB}\) by a left tensor ideal, Proposition 8.1 implies the
following.

\begin{Prop}
\label{prop:modues}
The category
$\mathrm{lfdmod}\text{-}B$ is naturally a left
module category over $\mathcal{AB}$. In other words,  $\text{lfdmod-}B$
 admits an affine Brauer categorification.  
 \end{Prop}
\begin{proof} More generally, suppose that
the category $\mathcal C$ in Proposition
8.1 is equipped with a strict anti-monoidal anti-involution,
equivalently a strict monoidal equivalence
\[
\iota:\mathcal C\longrightarrow \mathcal C^{op,\mathrm{rev}},
\]
where $\mathcal C^{op,\mathrm{rev}}$ denotes the category obtained from $\mathcal C$ by reversing both morphisms and the tensor product. Hence the left $\mathcal C^{op,\mathrm{rev}}$-module category structure on $\mathrm{lfdmod}\text{-}A$ constructed in Proposition~\ref{gen123} can be pulled back along $\iota$. More explicitly, if
\[
\Phi:\mathcal C^{op,\mathrm{rev}}\longrightarrow
\End(\mathrm{lfdmod}\text{-}A)
\]
is the action functor obtained above, then
\[
\Phi^\iota:=\Phi\circ\iota:
\mathcal C\longrightarrow
\End(\mathrm{lfdmod}\text{-}A)
\]
defines a left $\mathcal C$-module category structure on $\mathrm{lfdmod}\text{-}A$.
In particular, for the affine Brauer category, the path algebra $A$ in Proposition~\ref{gen123} is the algebra $B$ associated with the cyclotomic Brauer category.  
The diagrammatic anti-involution in Lemma~\ref{anti}, together with the left-right reversal of diagrams sending the dot to its negative, gives a strict monoidal equivalence  $\iota$ in this case.  Thus the left $\AB^{op,\mathrm{rev}}$-module category structure obtained from Proposition~\ref{gen123} is transported, via this anti-involution, to a left module category structure over the affine Brauer category $\AB$ itself.
\end{proof}

To relate $\mathrm{lfdmod}$-$B$ to module categories over $\mathfrak U^\imath_+$,
we now specialize to the half-integral case. Assume from now on that
$k=\mathbb C$ and that all roots of $m(u)$ lie in
$
I=\frac12+\mathbb Z$.
We rewrite $m(u)$ as
\begin{equation}
\label{equ:poly}
    m(u)=\prod_{i\in I}(u-i)^{\epsilon_i},
\end{equation}
where $\epsilon_i$ is the multiplicity of $i$ as a root of $m(u)$. This notation
will be used from this point to the end of this section.

Let  $O(u)$ denote  the left-hand side of \eqref{uadamr} and define $n(u)=O(u) m(u)$. Then \eqref{uadamr} implies that $n(u)$ is a polynomial. 
 %$$n(u)= ((-1)^au-\frac{1}{2})m(-u).$$ 
Write \begin{equation}
\label{nuexp} n(u)=\prod_{i\in I} (u-i)^{\phi_i}.\end{equation}
By \eqref{uadamr}, we have
\[
n(u)=\left(u-\frac12(-1)^a\right)
\prod_{r=1}^a(u+u_r),
\]
and hence 
\[
\varphi_i
=
\epsilon_{-i}
+
\delta_{i,\frac12(-1)^a}
\qquad\text{for all }i\in I.
\]
Then the integers
\(
c_i
:=
\varphi_i-\epsilon_i-\delta_{i,\frac12},
i\in I^\imath,
\)
have finite support and even total sum (see the proof of Corollary \ref{cor:even}). They therefore determine
an element $\kappa\in X^\imath_+$. Equivalently, $\kappa$ is
determined by
\begin{equation}
\label{equ:phiminusepslon1}
   \phi_i-\epsilon_i= \begin{cases} \langle ^\theta\alpha_i^\vee,\kappa \rangle, & \text{ if } i\neq  \frac{1}{2},\\
    \langle ^\theta\alpha_i^\vee,\kappa \rangle+1, & \text{ if } i=  \frac{1}{2}.
   \end{cases}
\end{equation} 
Let $E$ denote the generating object of $\AB$, represented by a vertical strand $\begin{tikzpicture}[baseline = -1mm, scale=0.7, transform shape]
 	\draw[-] (-0.18,-.4) to (-0.18,.4);
     %\node at (-0.18,0) {$\scriptstyle\bullet$};
     %\node at (0.35,0) {$\scriptstyle m(x)$};
 	%\draw[-,darkg,thick] (.9,.4) to (.9,-.4);
     %\node at (.9,-.55) {$\darkg\scriptstyle{V}$};
\end{tikzpicture}$. Then  
\[B = \bigoplus_{m,n\in \mathbb N} \Hom_{\mathcal{CB}_\mathbf u}(E^m,E^n).\] 
Thus $B$ is the path algebra of $\mathcal {CB}_{\mathbf u}$. 
By Proposition~\ref{prop:modues}, 
 $\text{lfdmod-}B$ becomes a  module category over $\mathcal {AB}$. 
The endofunctor $E$ on $\text{lfdmod-}B$ admits the  decomposition $E=\bigoplus_{i\in I}E_i$. By  the arguments in Section~\ref{Brauer-cate},  we have a weight decomposition 
\[ \text{lfdmod-}B= \prod_{\lambda\in {X_\imath^+}} \text{lfdmod-}B_\lambda,\]
 given in Lemma \ref{BLOCK} with $\lambda$ as in Definition \ref{expre}. 
Under the natural identification \(E(1_0B)\cong 1_1B\), the dot
has minimal polynomial \(m(u)\), while the dotted bubbles act on
\(1_0B\) by the prescribed scalars \(\omega_r\). Hence
\[
m_{1_0B}(u)=m(u)
\qquad\text{and}\qquad
\overline{n}_{1_0B}(u)=n(u).
\]
It follows from \eqref{equ:phiminusepslon1} and  \eqref{equ:phiminusepslon} that
\begin{equation}\label{objP}
P:=1_0B\in \operatorname{lfdmod}\text{-}B_\kappa.
\end{equation}
Moreover, there is a  natural isomorphism   
$E^dP\cong1_d B$, for $d\in\mathbb N$. Since $E=\bigoplus_{i\in I}E_i$, we have 
\[ E^d=\bigoplus_{\mathbf i\in I^d} E_{\mathbf i}\]
where $E_{\mathbf i}=E_{i_d}\ldots E_{i_1}$ for $\mathbf i=(i_1,\ldots,i_d)\in I^d$.
Define $\text{wt}(E_\mathbf i)=\alpha_{i_1}+\ldots+\alpha_{i_d}$. 
Then we have $E_{\mathbf i}P\in \text{lfdmod-}B_\lambda $ with $\lambda=\kappa+ \text{wt}(E_\mathbf i)$. 
Since $E^dP\cong 1_d B$, the idempotent $1_d$ decomposes as $1_d=\sum_{\mathbf i\in I^d} 1_{E_{\mathbf i}}$, with  
$1_{E_{\mathbf i}}B\cong E_{\mathbf i} P $.
In particular, we have the following  decomposition: % as locally unital algebra of $B$ 
\begin{equation}
\label{equ:decomofB}
B=\bigoplus_{m=0}^\infty \bigoplus_{n=0}^\infty\bigoplus_{\mathbf i\in I^m,\mathbf j\in I^n} 1_{E_\mathbf i} B1_{E_{\mathbf j}}\cong\bigoplus_{m=0}^\infty \bigoplus_{n=0}^\infty \bigoplus_{ \mathbf i\in I^m,\mathbf j\in I^n}\Hom_{B}(E_{\mathbf j}P, E_{\mathbf i}P).
\end{equation}

\subsection{Cyclotomic quotient category of $\mathfrak U^\imath_+$}
Fix $m(u)$, $n(u)$, and $\kappa\in X^+_\imath$ as in the previous subsection.
For each $i\in I$, let $\epsilon_i$ and $\phi_i$ be the nonnegative integers
determined by $m(u)$ and $n(u)$ as  in \eqref{equ:poly}-\eqref{nuexp}.

\begin{Defn}\label{844}
For each $i\in I^\imath$, define two families of parameters
\(\{\omega_i^{(r)} \mid r\in \mathbb Z\}\) and
\(\{\tilde{\omega}_i^{(r)} \mid r\in \mathbb Z\}\) by the following identities:
\begin{equation}\label{localpara}
O_i(u) := \sum_{r\in \mathbb Z} \omega_i^{(r)} u^{-r-1}
= u^{\phi_i-\epsilon_i},
\qquad
\tilde O_i(u) := \sum_{r\in \mathbb Z} \tilde\omega_i^{(r)} u^{-r-1}
=
2^{\delta_{i,\frac12}}
u^{\delta_{i,\frac12}+\epsilon_i-\phi_i}.
\end{equation}
\end{Defn}

Equivalently, for each $i\in I^\imath$, each of the two Laurent series
$O_i(u)$ and $\tilde O_i(u)$ consists of   a single monomial. Hence exactly one
of the parameters $\omega_i^{(r)}$ is nonzero, and exactly one of the parameters
$\tilde\omega_i^{(r)}$ is nonzero.

Following \cite[Section~2.6]{MM}, a 
\(\mathfrak U^\imath_+\)-invariant ideal of a \(2\)-representation
\(
  \mathcal M
  =
  (\mathcal M_\lambda)_{\lambda\in X^\imath_+}
\)
is a family
\(
  \mathcal I
  =
  (\mathcal I_\lambda)_{\lambda\in X^\imath_+}
\)
such that each \(\mathcal I_\lambda\) is an ideal of the
\(\mathbb C\)-linear category \(\mathcal M_\lambda\), and, for every
\(1\)-morphism \(F\colon\lambda\to\mu\),
\[
  F\bigl(\mathcal I_\lambda(X,Y)\bigr)
  \subseteq
  \mathcal I_\mu(FX,FY)
\]
for all \(X,Y\in\mathcal M_\lambda\).
The invariant ideal generated by a family of morphisms is the smallest
invariant ideal containing those morphisms.

Associated with $\kappa$, there is a universal $2$-representation
\((\mathcal R(\kappa)_\lambda)_{\lambda\in X_\imath^+}\) 
of $\mathfrak U^\imath_+$ defined by
\[
\mathcal R(\kappa)_\lambda
=
\Hom_{\mathfrak U^\imath_+}(\kappa,\lambda),
\]
where the actions of $1$- and $2$-morphisms are given by left horizontal
composition.
Let
\(
\mathcal I(\kappa)
=
\bigl(\mathcal I(\kappa)_\lambda\bigr)_
{\lambda\in X^\imath_+}
\)
be the $\mathfrak  U^\imath_+$-invariant ideal of the principal
$2$-representation
\(
\mathcal R(\kappa)
=
\bigl(\mathcal R(\kappa)_\lambda\bigr)_
{\lambda\in X^\imath_+}
\)
generated by
\begin{equation}\label{equ:genofgenealizedcyc}
\begin{tikzpicture}[baseline = -1mm,darkred, scale=0.7, transform shape ]
	\draw[->,thick] (0.08,-.4) to (0.08,.4);
    \node at (.08,0) {$\bullet$};
    \node at (.08,-.6) {$\scriptstyle{i}$};
    \node at (0.5,0.05) {$\color{black}\scriptstyle{\epsilon_i}$};
    \node at (0.4,-.3) {$\color{black}\scriptstyle{\kappa}$};
\end{tikzpicture},
\quad
\begin{tikzpicture}[baseline = -1mm,darkred, scale=0.7, transform shape ]
	\draw[<-,thick] (0.08,-.4) to (0.08,.4);
    \node at (.08,0) {$\bullet$};
    \node at (.08,-.6) {$\scriptstyle{i}$};
    \node at (0.5,0.05) {$\color{black}\scriptstyle{\phi_i}$};
    \node at (0.4,-.3) {$\color{black}\scriptstyle{\kappa}$};
\end{tikzpicture},
\quad
\begin{tikzpicture}[baseline = 0, scale=0.7, transform shape ]
  \draw[->,thick,darkred] (0.2,0.2) to[out=90,in=0] (0,.4);
  \draw[-,thick,darkred] (0,0.4) to[out=180,in=90] (-.2,0.2);
  \draw[-,thick,darkred] (-.2,0.2) to[out=-90,in=180] (0,0);
  \draw[-,thick,darkred] (0,0) to[out=0,in=-90] (0.2,0.2);
  \node at (0,-.1) {$\scriptstyle{i}$};
  \node at (-0.3,0.2) {$\scriptstyle{\kappa}$};
  \node at (0.2,0.2) {$\color{darkred}\scriptstyle\bullet$};
  \node at (0.4,0.2) {$\color{darkred}\scriptstyle{r}$};
\end{tikzpicture}
-
\omega_i^{(r)},
\quad
\begin{tikzpicture}[baseline = 0, scale=0.7, transform shape ]
  \draw[<-,thick,darkred] (0,0.4) to[out=180,in=90] (-.2,0.2);
  \draw[-,thick,darkred] (0.2,0.2) to[out=90,in=0] (0,.4);
  \draw[-,thick,darkred] (-.2,0.2) to[out=-90,in=180] (0,0);
  \draw[-,thick,darkred] (0,0) to[out=0,in=-90] (0.2,0.2);
  \node at (0,-.1) {$\scriptstyle{i}$};
  \node at (0.3,0.2) {$\scriptstyle{\kappa}$};
  \node at (-0.2,0.2) {$\color{darkred}\scriptstyle\bullet$};
  \node at (-0.4,0.2) {$\color{darkred}\scriptstyle{r}$};
\end{tikzpicture}
-
\tilde\omega_i^{(r)},
\end{equation}
for all $r\in\mathbb Z$.
The cyclotomic quotient is then defined by
\[
\mathcal Q(\kappa)_\lambda
:=
\mathcal R(\kappa)_\lambda/
\mathcal I(\kappa)_\lambda,
\qquad
\mathcal Q(\kappa)
:=
\prod_{\lambda\in X^\imath_+}
\mathcal Q(\kappa)_\lambda.
\]
By construction, $\mathcal I(\kappa)$ is stable under the action
of $\mathfrak U^\imath_+$; hence the universal action on
$\mathcal R(\kappa)$ descends to a strict $2$-action on
$\mathcal Q(\kappa)$.

Recall that, for each \(\lambda\in X^\imath_+\),
the category \(\mathcal Q(\kappa)_\lambda\) is additively
generated by the objects
\(\{
E_{\mathbf i}1_\kappa\mid 
\mathbf i=(i_1,\ldots,i_d),
\lambda=\kappa+\operatorname{wt}(E_{\mathbf i})\}\),
where
\[
\operatorname{wt}(E_{\mathbf i})
=
\sum_{j=1}^{d}\operatorname{wt}(E_{i_j}),
\qquad
\operatorname{wt}(E_h)=\alpha_h.
\]
We define the locally unital algebra
\[
Q:=
\bigoplus_{
  \substack{
    \mathbf i,\mathbf j,\\
    \kappa+\operatorname{wt}(E_{\mathbf i})
    =
    \kappa+\operatorname{wt}(E_{\mathbf j})
    \in X^\imath_+
  }}
\operatorname{Hom}_{\mathcal Q(\kappa)}
\bigl(E_{\mathbf i}1_\kappa,E_{\mathbf j}1_\kappa\bigr).
\]
Equivalently, \(Q\) is the path algebra of the full subcategory
of \(\mathcal Q(\kappa)\) on these additive generators.
The weight decomposition of $\mathcal Q(\kappa)$ induces a decomposition
\[
Q=\bigoplus_{\lambda\in X^\imath_+}Q_\lambda,
\qquad
Q_\lambda:=
\bigoplus_{\substack{\mathbf i,\mathbf j\\
\kappa+\operatorname{wt}(E_{\mathbf i})
=\kappa+\operatorname{wt}(E_{\mathbf j})=\lambda}}
\operatorname{Hom}_{\mathcal Q(\kappa)_\lambda}
(E_{\mathbf i}1_\kappa,E_{\mathbf j}1_\kappa).
\]
so that $Q$ is a locally unital algebra with distinguished idempotents $1_{E_\mathbf i}$.

\begin{Prop}  $\mathcal Q(\kappa)$ is locally
finite-dimensional. \end{Prop}

\begin{proof} Fix $\lambda\in X^\imath_+$ and sequences
$\mathbf i$ and $\mathbf j$ such that
$
E_{\mathbf i}1_\kappa,\ E_{\mathbf j}1_\kappa
   \in \mathcal R(\kappa)_\lambda$. By \cite[Definition~5.8 and Theorem~5.10]{BWW25}, the 
$2$-category $\mathfrak U^\imath$ is non-degenerate. Hence the
dotted reduced matching diagrams described in
\cite[Definition~5.8]{BWW25}, together with bubble monomials in the
rightmost region, form a diagrammatic basis of
$
\operatorname{Hom}_{\mathfrak U^\imath}
   (E_{\mathbf i}1_\kappa,E_{\mathbf j}1_\kappa)$.
Consequently, the images of these diagrams span
$
\operatorname{Hom}_{\mathcal Q(\kappa)_\lambda}
   (E_{\mathbf i},E_{\mathbf j})$.
We use only this spanning property; their images need not remain
linearly independent in the cyclotomic quotient.

The zig-zag relations give a vector-space isomorphism between the
morphism space above and a morphism space
$
\operatorname{Hom}_{\mathcal Q(\kappa)}
   (1_\kappa,E_{\mathbf k}1_\kappa)$,
where $\mathbf k$ is a sequence of length $2r$ for some $r\geq 0$.
Fix a reduced matching diagram in this latter morphism space. We may
choose its representative in the form
$
D_{\mathrm{cr}}\circ D_{\mathrm{cup}}$,
where $D_{\mathrm{cr}}$ is a reduced crossing diagram corresponding
to a permutation in $S_{2r}$, while $D_{\mathrm{cup}}$ is obtained
by successively joining adjacent positions by cups.

In the affine diagrammatic basis, the distinguished positions
carrying the dot polynomials may be chosen on the cup strands below
all crossings. Thus $D_{\mathrm{cup}}$ may be taken to be an
iterated composite of dotted cups, with no dot passing through a
crossing. By the interchange law, these dotted cups may be regarded
as being applied successively in the rightmost position, adjacent
to the exterior region labelled by $\kappa$. Therefore, at the
stage at which a given dotted cup is introduced, its dot polynomial
acts on the rightmost strand. The corresponding cyclotomic relation
in \eqref{equ:genofgenealizedcyc} then reduces the number of dots on that cup below the degree
of the relevant cyclotomic polynomial.

Applying this argument successively to all the cups in
$D_{\mathrm{cup}}$, we obtain a uniform bound on the number of dots
on every cup. Moreover, the bubble relations in \eqref{equ:genofgenealizedcyc} send all
bubble monomials in the rightmost region to their prescribed
scalars. For fixed boundary data, there are only finitely many
reduced matching diagrams, and each of them contains only finitely
many cups. Since the number of dots on every cup is bounded, only
finitely many dotted diagrams are required to span
$
\operatorname{Hom}_{\mathcal Q(\kappa)_\lambda}
   (E_{\mathbf i},E_{\mathbf j})$.
Therefore this morphism space is finite-dimensional. It follows
that $\mathcal Q(\kappa)$, and hence its path algebra $Q$, is
locally finite-dimensional.
\end{proof}

%Hence $\Hom_{\mathcal Q(\kappa)_\lambda}(E_\mathbf i, E_\mathbf j)$ is  finite-dimensional and $Q$ is locally finite--dimensional. 

By \cite[Construction~4.26]{BD-cate}, the strict $2$-action of 
$\mathfrak U^\imath_+$ on $\mathcal Q(\kappa)$ induces a categorical action
on $\mathrm{lfdmod}\text{-}Q$. Concretely, for each $i\in I^\imath$, left 
horizontal composition with $E_i$ defines algebra homomorphisms
\[
e_i: Q_\lambda \to Q_{\lambda+\alpha_i},
\]
whenever $\lambda, \lambda+\alpha_i \in X_\imath^+$. Restriction along $e_i$
 makes $Q_{\lambda+\alpha_i}$ into  a $(Q_\lambda, Q_{\lambda+\alpha_i})$-bimodule. Tensoring on the
right with this bimodule defines a functor on module categories. The functor induced similarly by $E_{-i}$ is biadjoint to it. Hence by \cite[Theorem~2.11]{BD-cate}, it is exact and preserves locally finite-dimensional modules, and therefore restricts an exact functor
\[
E_i: \mathrm{lfdmod}\text{-}Q_\lambda \longrightarrow \mathrm{lfdmod}\text{-}Q_{\lambda+\alpha_i}.
\]
Moreover, the generating $2$-morphisms of $\mathfrak U^\imath_+$ induce 
bimodule homomorphisms, hence natural transformations among these functors,
which satisfy the same defining relations as in $\mathfrak U^\imath_+$. 
Thus $\mathrm{lfdmod}\text{-}Q$ naturally becomes a $2$-representation of 
$\mathfrak U^\imath_+$. Let 
\[
P' := 1_\emptyset Q.
\]
Then $P'$ is the projective right $Q$-module corresponding to $1_\emptyset$, 
and for every sequence $\mathbf i$ there is a natural isomorphism
\begin{equation}\label{proj}
E_\mathbf i P' \cong 1_{E_\mathbf i} Q.
\end{equation}
Therefore,
\begin{equation}\label{equ:decomofQ}
Q = \bigoplus_{m=0}^\infty \bigoplus_{n=0}^\infty 
\bigoplus_{\mathbf i \in I^m,\, \mathbf j \in I^n} 
1_{E_\mathbf i} Q 1_{E_\mathbf j} 
\cong 
\bigoplus_{m=0}^\infty \bigoplus_{n=0}^\infty 
\bigoplus_{\mathbf i \in I^m,\, \mathbf j \in I^n} 
\Hom_Q(E_\mathbf j P', E_\mathbf i P').
\end{equation}
In particular, by \eqref{proj}, the projective modules obtained from $P'$ under 
the action of the functors $E_\mathbf i$ recover the canonical projective generators
of $\mathrm{lfdmod}\text{-}Q$, namely the modules 
\(\{ 1_{E_\mathbf i} Q \mid \mathbf i \in I^n, n\in \mathbb N \}\).

\subsection{Isomorphism between two cyclotomic quotient categories}
We now apply the two constructions established in Sections~5--7 to the
cyclotomic setting. The argument is parallel in spirit to the
Brundan--Kleshchev proof of the isomorphism between cyclotomic Hecke algebras
and cyclotomic quiver Hecke algebras: one first compares the corresponding
block module categories and then glues the resulting blockwise identifications
to obtain an isomorphism of the whole locally unital algebra.

\begin{Theorem}\label{cyc-iso}
There is an isomorphism
\(
B\cong Q
\)
of locally unital $\mathbb C$-algebras.
\end{Theorem}
\begin{proof}
By the result on dot spectra for the $\mathbf u$-admissible
cyclotomic Brauer categories in~\cite[Theorem~6.5]{GRS-tri},
the dot spectrum of the generating endofunctor on
$\operatorname{lfdmod}\text{-}B$ is exactly
$
I=\frac12+\mathbb Z$. Hence Theorem~\ref{thm:mainthmaffineb} applies and yields a
generalized nilpotent $2$-representation of
$\mathfrak U^\imath_+$ on
$
\bigl(
\operatorname{lfdmod}\text{-}B_\lambda
\bigr)_{\lambda\in X_\imath^+}
$. The universal property of the principal $2$-representation
\(\bigl(
\mathcal R(\kappa)_\lambda
\bigr)_{\lambda\in X_\imath^+}
\)
gives a unique morphism of $2$-representations
\[
\mathcal R(\kappa)
\longrightarrow
\operatorname{lfdmod}\text{-}B
\]
sending $1_\kappa$ to $P$, where
$P\in\operatorname{lfdmod}\text{-}B_\kappa$ is defined
in~\eqref{objP}.

Moreover, the constructions used in the proof of
Theorem~\ref{thm:mainthmaffineb}, more precisely,
Definitions~\ref{natrual-basic}(1)
and~\ref{natrual-basic1}(4), together with the cyclotomic
relation
$
m(x)=\prod_{r=1}^{a}(x-u_r)=0
$
in $B$ and the choice of parameters in
Definition~\ref{844}, show that the generators of the
$2$-ideal $\mathcal I(\kappa)$ listed
in~\eqref{equ:genofgenealizedcyc} are sent to zero under this
morphism. Hence the morphism factors through
$\mathcal Q(\kappa)$.
Thus, we obtain a $\mathbb C$-linear functor
\[
\Phi:
\mathcal Q(\kappa)
\longrightarrow
\operatorname{lfdmod}\text{-}B
\]
sending $E_{\mathbf i}$ to $E_{\mathbf i}P$.
Using the decomposition of $B$ in~\eqref{equ:decomofB},
this functor induces a $\mathbb C$-algebra homomorphism
\[
\Phi:Q\longrightarrow B.
\]
It sends each distinguished idempotent
$1_{E_{\mathbf i}}$ to the corresponding idempotent
$1_{E_{\mathbf i}}$ in $B$, and its values on the
diagrammatic generators are given by the constructions in
Section~\ref{modue5}, in particular, Definition~\ref{natrual-basic}.

Conversely, consider the locally Schurian $2$-representation on
$\operatorname{lfdmod}\text{-}Q$ induced by the cyclotomic
quotient $\mathcal Q(\kappa)$. The full subcategory
$\operatorname{pmod}\text{-}Q$ of finitely generated projective
modules is stable under the generating functors $E_i$ and $F_i$,
since these functors are sweet and hence preserve finitely
generated projective objects. Thus,
$\operatorname{pmod}\text{-}Q$ is a $2$-subrepresentation of
$\operatorname{lfdmod}\text{-}Q$.

Recall $
P':=1_{\emptyset}Q$. By \eqref{proj}, 
every distinguished projective module is of the form
$
1_{E_{\mathbf i}}Q\cong E_{\mathbf i}P'
$
for some sequence $\mathbf i$. The cyclotomic relations
in~\eqref{equ:genofgenealizedcyc}, together with their
horizontal translates, imply the  nilpotence conditions in
Definition~\ref{nil} on every distinguished projective module.
Since every finitely generated projective $Q$-module is a
direct summand of a finite direct sum of distinguished
projective modules, these two conditions hold on every object
of $\operatorname{pmod}\text{-}Q$. Hence
$\operatorname{pmod}\text{-}Q$ is a nilpotent
$2$-subrepresentation of $\operatorname{lfdmod}\text{-}Q$.

Applying part~(2) of Theorem~\ref{main111}, we obtain an action
of the affine Brauer category on
$\operatorname{pmod}\text{-}Q$. Evaluating this action at $P'$
gives a $\mathbb C$-linear functor
\[
\mathcal{AB}
\longrightarrow
\operatorname{pmod}\text{-}Q
\]
sending $1_0$ to $P'$.
By the constructions in
Definitions~\ref{natrual-basic-red}
and~\ref{cross-red123},
Proposition~\ref{bub=bubred}, and the parameter identities in
Definition~\ref{844}, the relation
$
m(x)=0
$
and the prescribed bubble relations in
Definition~\ref{822} act trivially on $P'$. Their horizontal
translates therefore act trivially on every $E^dP'$. Hence the
resulting $\mathcal{AB}$-action factors through
$\mathcal{CB}_{\mathbf u}$.
Consequently, we obtain a $\mathbb C$-linear functor
\[
\Psi:
\mathcal{CB}_{\mathbf u}
\longrightarrow
\operatorname{pmod}\text{-}Q
\]
sending the object $E^d$ to $E^dP'$ for every $d\geq0$.
The generalized-eigenspace decomposition
\[
E^dP'
=
\bigoplus_{\mathbf i\in I^d}E_{\mathbf i}P'
\]
then yields maps on the corresponding idempotent summands.
Using the decomposition~\eqref{equ:decomofQ}, these maps
assemble to a $\mathbb C$-algebra homomorphism
\[
\Psi:B\longrightarrow Q
\]
sending $1_{E_{\mathbf i}}$ to the corresponding idempotent
$1_{E_{\mathbf i}}$. Its values on the diagrammatic generators
are those given by
Definitions~\ref{natrual-basic-red}
and~\ref{cross-red123}.

By construction, the formulas in Definitions 7.11 and 7.13 are obtained by solving the formulas in Definitions 5.7 and 5.9 for the affine Brauer generators. Thus, on each generalized-eigenspace summand, the two sets of formulas are mutually inverse. Consequently, the composites \(\Psi\circ\Phi\) and \(\Phi\circ\Psi\) fix the distinguished idempotents and all the corresponding generators: the generating \(2\)-morphisms of \(Q\), and the dot, crossing, cup, and cap generators of \(B\), respectively. Since these elements generate \(Q\) and \(B\), it follows that
\[
 \Psi\circ\Phi=\operatorname{Id}_{Q},
 \qquad
 \Phi\circ\Psi=\operatorname{Id}_{B}.
 \]
This does not contradict Remark 7.30: that remark concerns whether the two constructions recover an arbitrary \(2\)-representation globally when not all spectral summands are present, whereas the inverse identities above hold on each summand appearing in the cyclotomic constructions.

% \old{By the explicit formulas in Sections~ \ref{modue5} and~\ref{ikmcmc}, i.e., Definitions~\ref{natrual-basic} and \ref{natrual-basic1}, Definition~\ref{natrual-basic-red} and~\ref{cross-red123}, the assignments
% defining $\Phi$ and $\Psi$ are inverse to each other on the
% distinguished idempotents and on all the corresponding
% generators: the dot, crossing, cup, and cap in $B$, and the
% generating $2$-morphisms of $\mathfrak U^\imath_+$ in $Q$.
% Since $B$ and $Q$ are generated by these distinguished
% idempotents and generators, respectively, it follows that
% \[
% \Psi\circ\Phi=\operatorname{Id}_Q,
% \qquad
% \Phi\circ\Psi=\operatorname{Id}_B.
% \]}
Thus, $\Phi$ and $\Psi$ are mutually inverse isomorphisms of
locally unital $\mathbb C$-algebras. Therefore,
\(
B\cong Q
\).
\end{proof}

The natural grading on the cyclotomic quotient of
$\mathfrak U^\imath_+$ induces, via Theorem~\ref{cyc-iso}, a natural
grading on the corresponding cyclotomic Brauer category. For the explicit degree of the generators, see the table after Remark 3.2 of \cite{BSWW-icate}. The same
theorem also yields the following result.

\begin{Cor}
\label{cor:cycb}
Let \(B_{m,n}\) be the \(u\)-admissible cyclotomic Brauer (or Nazarov–Wenzl) algebra associated with the parameters fixed above, and
assume that all roots of its cyclotomic polynomial lie in \(\frac12+\mathbb Z\).
 Then there is an isomorphism of
$\mathbb C$-algebras
$
\mathcal B_{m,n}
\cong
\End_{ Q}(X_n)$, where $
X_n=\bigoplus_{\mathbf i\in I^n}E_{\mathbf i}P'$. For each fixed $n$, only finitely many summands in the definition of $X_n$ is non-zero. 
In particular, $\mathcal B_{m,n}$ carries a natural
$\mathbb Z$-grading.
\end{Cor}

\begin{proof}
It was proved by two of the present authors in \cite[Theorem~C]{RS-cyc} that
there is an isomorphism of $\mathbb C$-algebras
$
\mathcal B_{m,n}
\cong
\End_{\CB_{\mathbf u}}(n)$.
The required isomorphism, and hence the induced
$\mathbb Z$-grading, now follow immediately from
Theorem~\ref{cyc-iso}.
\end{proof}

Theorem 8.6 and Corollary 8.7 complete the categorical identification pursued in this paper. They show that the ungraded cyclotomic Brauer algebra is the underlying algebra of a naturally graded endomorphism algebra arising from a cyclotomic $\imath$-Kac--Moody 2-representation (noting that the ideal $I_\kappa$ is homogeneous). A further problem is to describe this grading intrinsically by generators and relations. Such a KLR-type presentation requires a separate basis theorem and is not needed for the isomorphism proved here. Together with graded triangular basis methods, that presentation is expected to provide the next step toward categorifying the bar involution and relating decomposition numbers to $\imath$-canonical bases.

We expect that a parallel isomorphism between  a cyclotomic Kauffman category with quantum parameter $q$ not a  root of unity and the corresponding cyclotomic $\imath$-Kac--Moody $2$-category should  provide a closer counterpart of the Brundan--Kleshchev isomorphism, namely an $\imath$-analogue relating cyclotomic Hecke algebras to quiver Hecke algebras. Details will appear in \cite{GHRS}.

\section{Rational-function identities}
The purpose of this section is to record and verify the
rational-function identities used in the proofs of
Theorem~\ref{inhomo100} and Theorem~\ref{thm:mainthmaffineb1}. All identities below are identities 
in  $\mathbb C(x_1,x_2,x_3,x_4)$.

\begin{Defn}For $1\leq a,b\leq4$ and $\sigma,\tau\in\{1,-1\}$, set
\(
 d^{ab}_{\sigma\tau}:=\sigma x_a+\tau x_b\) and 
 \(
 \ell^{ab}_{\sigma\tau}:=1+d^{ab}_{\sigma\tau}\).
We also set
$
\Delta_{123}:=1+d^{12}_{+-}d^{23}_{+-}$.\end{Defn}
The supplementary
Maple file \texttt{verify\_section9\_maple.mpl} carries out these
symbolic simplifications in Lemmas~\ref{lem:unprimed-rational-identities} and \ref{lem:primed-rational-identities}.

\subsection{The identities arising in the forward construction}
The following rational functions occur in the verification of the inhomogeneous relation in Section~6. 
\begingroup
\small
\allowdisplaybreaks[4]
\begin{align*}
T_1
=&\frac{1}{\ell^{13}_{+-}d^{23}_{+-}\ell^{12}_{-+}}
+\frac{\ell^{12}_{+-}}{\ell^{13}_{+-}}
+\frac{\ell^{23}_{-+}}{\ell^{13}_{-+}}
-\frac{1}{\ell^{13}_{-+}d^{12}_{+-}\ell^{23}_{+-}}
+\frac{\ell^{12}_{--}}
{d^{12}_{+-}d^{23}_{+-}\ell^{23}_{--}\ell^{13}_{-+}}\\
&+\frac{\ell^{12}_{--}}
{\ell^{23}_{--}\ell^{13}_{-+}}
-\frac{\ell^{12}_{--}}
{\ell^{23}_{+-}\ell^{12}_{-+}\ell^{23}_{--}\ell^{13}_{-+}}
-\frac{\ell^{12}_{--}}
{d^{12}_{+-}d^{23}_{+-}\ell^{23}_{+-}\ell^{12}_{-+}
\ell^{23}_{--}\ell^{13}_{-+}}
-\frac{\ell^{23}_{-+}}
{\ell^{13}_{-+}\ell^{23}_{--}d^{12}_{+-}d^{23}_{+-}}\\
& +\frac{1}{\ell^{13}_{-+}\ell^{23}_{--}\ell^{23}_{+-}}
-\frac{\ell^{23}_{-+}}
{\ell^{13}_{-+}\ell^{23}_{--}}
+\frac{1}
{\ell^{13}_{-+}\ell^{23}_{--}\ell^{23}_{+-}d^{12}_{+-}d^{23}_{+-}}
-\frac{d^{12}_{+-}+(d^{12}_{+-})^2d^{23}_{+-}}
{\ell^{23}_{--}\ell^{13}_{+-}d^{23}_{+-}\ell^{12}_{-+}}\\
&+\frac{2\ell^{12}_{--}(d^{12}_{+-})^2}
{\ell^{13}_{-+}\ell^{23}_{--}\ell^{13}_{+-}\ell^{12}_{-+}}
+\frac{2\ell^{12}_{--}(d^{12}_{+-})^2}
{d^{12}_{+-}d^{23}_{+-}\ell^{13}_{-+}\ell^{23}_{--}
\ell^{13}_{+-}\ell^{12}_{-+}},
\\[2mm]
%\[\begin{aligned} 
T_2
=&\frac{(d^{11}_{++})^2}
{\ell^{11}_{++}\ell^{24}_{+-}d^{12}_{-+}d^{14}_{+-}d^{24}_{++}}
-\frac{\ell^{24}_{+-}d^{24}_{++}}
{\ell^{24}_{--}d^{12}_{-+}\ell^{14}_{--}d^{14}_{+-}\ell^{14}_{++}}
-\frac{1}
{d^{24}_{++}d^{12}_{-+}\ell^{14}_{--}\ell^{14}_{++}}\\
&+\frac{1}
{\ell^{24}_{--}\ell^{11}_{++}d^{24}_{++}d^{12}_{-+}
\ell^{14}_{--}\ell^{14}_{++}}
-\frac{2}
{d^{12}_{-+}d^{14}_{+-}\ell^{14}_{--}\ell^{14}_{++}}
+\frac{2}
{\ell^{24}_{--}\ell^{11}_{++}d^{12}_{-+}d^{14}_{+-}
\ell^{14}_{--}\ell^{14}_{++}}\\
&+\frac{\ell^{24}_{--}\ell^{11}_{++}-1}
{\ell^{24}_{--}\ell^{11}_{++}d^{24}_{++}d^{14}_{+-}
\ell^{14}_{--}\ell^{14}_{++}}
-\frac{2(d^{11}_{++})^2}
{\ell^{11}_{++}\ell^{24}_{+-}d^{24}_{--}d^{14}_{+-}
\ell^{14}_{++}\ell^{14}_{--}}\\
&+\frac{(d^{11}_{++})^2(d^{14}_{+-}-d^{24}_{++})}
{\ell^{11}_{++}\ell^{24}_{+-}d^{24}_{++}d^{14}_{+-}
d^{12}_{-+}\ell^{14}_{++}\ell^{14}_{--}}
-\frac{(d^{11}_{++})^2d^{14}_{++}}
{\ell^{11}_{++}\ell^{24}_{+-}d^{24}_{++}d^{12}_{+-}
d^{14}_{+-}\ell^{14}_{--}}\\
&-\frac{(\ell^{24}_{++}\ell^{24}_{--}-1)d^{24}_{+-}}
{\ell^{24}_{--}d^{24}_{++}d^{14}_{+-}d^{12}_{-+}
\ell^{14}_{--}\ell^{14}_{++}},
\\[2mm]
%\[\begin{aligned} 
T_3
=&\frac{1}
{\ell^{44}_{++}d^{14}_{-+}d^{12}_{+-}\ell^{12}_{--}}
+\frac{d^{12}_{+-}(d^{12}_{+-}+d^{14}_{+-})}
{d^{14}_{-+}d^{24}_{++}\ell^{14}_{++}\ell^{12}_{-+}}
-\frac{\ell^{12}_{--}d^{12}_{+-}}
{\ell^{14}_{--}\ell^{14}_{++}\ell^{12}_{-+}d^{14}_{-+}}\\
&-\frac{2-d^{11}_{++}}
{\ell^{44}_{++}d^{24}_{++}d^{12}_{+-}\ell^{12}_{--}\ell^{12}_{-+}}
-\frac{\ell^{12}_{--}d^{12}_{-+}}
{\ell^{14}_{++}\ell^{14}_{--}d^{14}_{-+}d^{24}_{++}\ell^{12}_{-+}}
-\frac{(d^{44}_{++})^2}
{d^{14}_{-+}\ell^{44}_{++}d^{24}_{++}\ell^{12}_{--}d^{12}_{+-}}\\
&-\frac{\ell^{12}_{--}(d^{12}_{+-})^2}
{\ell^{14}_{--}\ell^{14}_{++}d^{14}_{-+}d^{24}_{++}\ell^{12}_{-+}}
+\frac{\ell^{44}_{--}}
{d^{24}_{++}d^{12}_{+-}\ell^{12}_{--}}
+\frac{d^{12}_{+-}+d^{44}_{++}}
{d^{24}_{++}d^{14}_{-+}d^{12}_{+-}}\\
&-\frac{2}
{d^{24}_{++}\ell^{44}_{++}d^{14}_{-+}\ell^{12}_{-+}}
+\frac{-\ell^{44}_{--}}
{d^{14}_{-+}d^{12}_{+-}\ell^{12}_{--}}
-\frac{1}
{\ell^{44}_{++}d^{14}_{-+}d^{12}_{+-}\ell^{12}_{-+}}
+\frac{2d^{12}_{+-}}
{d^{14}_{-+}\ell^{14}_{++}\ell^{12}_{-+}},
\\[2mm]
%\[\begin{aligned} 
T_4
=&-\frac{(d^{11}_{++})^2}
{d^{13}_{-+}\ell^{23}_{++}\ell^{23}_{--}\ell^{11}_{++}}
-\frac{2}
{d^{12}_{-+}d^{23}_{-+}\ell^{12}_{--}\ell^{12}_{++}}
+\frac{(d^{11}_{++})^2}
{\ell^{23}_{--}\ell^{23}_{++}\ell^{11}_{++}\ell^{12}_{--}}\\
&+\frac{(d^{33}_{++})^2}
{d^{13}_{-+}\ell^{12}_{++}\ell^{12}_{--}\ell^{33}_{++}}
+\frac{2}
{\ell^{33}_{++}\ell^{11}_{++}\ell^{12}_{--}\ell^{12}_{++}}
-\frac{(d^{11}_{++})^2}
{d^{13}_{-+}d^{23}_{-+}d^{12}_{-+}
\ell^{23}_{++}\ell^{23}_{--}\ell^{11}_{++}}\\
&+\frac{(d^{33}_{++})^2}
{d^{13}_{-+}d^{23}_{-+}d^{12}_{-+}
\ell^{12}_{++}\ell^{12}_{--}\ell^{33}_{++}}
+\frac{2}
{\ell^{33}_{++}d^{23}_{-+}d^{12}_{-+}
\ell^{11}_{++}\ell^{12}_{--}\ell^{12}_{++}}\\
&-\frac{(d^{11}_{++})^2}
{\ell^{11}_{++}\ell^{23}_{++}d^{23}_{-+}d^{12}_{-+}
\ell^{12}_{++}\ell^{12}_{--}}
+\frac{(d^{11}_{++})^2}
{\ell^{11}_{++}\ell^{23}_{--}\ell^{23}_{++}d^{23}_{-+}
d^{12}_{-+}\ell^{12}_{--}}\\
&-\frac{(d^{11}_{++})^2}
{\ell^{23}_{++}\ell^{11}_{++}\ell^{12}_{++}\ell^{12}_{--}}
-\frac{2}
{\ell^{12}_{--}\ell^{12}_{++}},
\\[2mm]
%\begin{aligned} 
T_5
=&1+\frac{\ell^{23}_{--}}
{\ell^{12}_{--}d^{12}_{+-}\ell^{23}_{-+}}
-\frac{1}
{\ell^{12}_{--}\ell^{23}_{-+}}
+\frac{(d^{23}_{+-})^2}
{\ell^{12}_{--}\ell^{23}_{-+}}
-\frac{\ell^{12}_{--}}
{\ell^{12}_{+-}d^{23}_{+-}\ell^{23}_{--}}
-\frac{1}
{\ell^{12}_{+-}\ell^{23}_{--}}\\ 
& +\frac{(d^{12}_{+-})^2}
{\ell^{12}_{+-}\ell^{23}_{--}}
+\frac{\ell^{12}_{-+}}
{\ell^{23}_{--}}-\frac{(d^{23}_{+-})^2}
{\ell^{12}_{--}\ell^{23}_{-+}}
-\frac{d^{23}_{+-}}
{d^{12}_{+-}\ell^{12}_{--}\ell^{23}_{-+}}
+\frac{-\ell^{12}_{-+}}
{d^{12}_{+-}d^{23}_{-+}\ell^{23}_{--}}\\
&-\frac{1}
{\ell^{12}_{+-}\ell^{23}_{-+}}
+\frac{1}
{d^{12}_{+-}d^{23}_{-+}\ell^{12}_{+-}\ell^{23}_{--}}
+\frac{1}
{d^{12}_{+-}d^{23}_{+-}}\\
&+\frac{1}
{d^{12}_{+-}d^{23}_{-+}\ell^{12}_{+-}\ell^{23}_{-+}}
-\frac{1}
{\ell^{12}_{+-}\ell^{23}_{--}} ,
\end{align*}

\begin{Lemma}~\label{lem:unprimed-rational-identities}
As elements of $\mathbb C(x_1,x_2,x_3,x_4)$, the rational functions
above satisfy $
 T_1=T_5=1$ and $
 T_2=T_3=T_4=0$.
\end{Lemma}
\begin{proof}
After bringing each of
$
 T_1-1, T_2,  T_3,  T_4, T_5-1
$
to a common denominator, direct expansion shows that its numerator
vanishes identically. 
\end{proof}

\subsection{Identities for the reverse construction} The following rational functions occur in the verification of the
affine Brauer relations in Section~7.
\begingroup
\small
\allowdisplaybreaks[4]
\begin{align*}
T'_1={}&
-\frac{(d^{11}_{++})^2\Delta_{123}}
{\ell^{23}_{++}d^{12}_{+-}d^{23}_{+-}
 \bigl(1-(d^{12}_{++})^2\bigr)\bigl((d^{11}_{++})^2-1\bigr)}
+\frac{2\Delta_{123}}
{\ell^{33}_{++}d^{12}_{+-}d^{23}_{+-}
 \bigl((d^{11}_{++})^2-1\bigr)\bigl(1-(d^{12}_{++})^2\bigr)}
\\
&+\frac{2\Delta_{123}}
{d^{12}_{+-}d^{23}_{+-}\bigl((d^{33}_{++})^2-1\bigr)
 (d^{11}_{++}-1)\bigl(1-(d^{12}_{++})^2\bigr)}
+\frac{(d^{11}_{++})^2\Delta_{123}}
{\bigl(1-(d^{23}_{++})^2\bigr)d^{12}_{+-}d^{23}_{+-}
 \ell^{12}_{--}\bigl((d^{11}_{++})^2-1\bigr)}
\\
&+\frac{(d^{33}_{++})^2\Delta_{123}}
{d^{13}_{-+}\ell^{12}_{++}\ell^{12}_{--}\ell^{33}_{++}
 d^{12}_{+-}d^{23}_{+-}(d^{33}_{++}-1)}
-\frac{(d^{11}_{++})^2\Delta_{123}}
{\ell^{23}_{++}\ell^{23}_{--}\ell^{11}_{++}
 d^{23}_{+-}d^{12}_{+-}(d^{11}_{++}-1)d^{13}_{-+}},\\[2mm] 
%\begin{align*}
T'_2={}&
-\frac{(d^{11}_{++})^2d^{14}_{++}}
{\ell^{24}_{+-}d^{14}_{+-}d^{12}_{-+}d^{24}_{++}
 \bigl(1-(d^{11}_{++})^2\bigr)\ell^{14}_{--}}
+\frac{d^{24}_{++}d^{24}_{+-}}
{\bigl(1-(d^{14}_{++})^2\bigr)\bigl((d^{24}_{++})^2-1\bigr)
 d^{12}_{-+}d^{14}_{+-}}
\\
&+\frac{(d^{11}_{++})^2(d^{12}_{+-}+d^{44}_{++})}
{\ell^{24}_{+-}d^{14}_{+-}d^{12}_{-+}d^{24}_{++}
 \bigl(1-(d^{11}_{++})^2\bigr)\bigl(1-(d^{14}_{++})^2\bigr)}
+\frac{d^{11}_{++}+d^{24}_{++}}
{\ell^{24}_{--}d^{12}_{+-}d^{24}_{++}d^{14}_{+-}
 \bigl(1-(d^{11}_{++})^2\bigr)\bigl(1-(d^{14}_{++})^2\bigr)}
\\
&-\frac{d^{11}_{++}+d^{24}_{++}}
{\bigl(1-(d^{24}_{++})^2\bigr)d^{12}_{+-}d^{24}_{++}d^{14}_{+-}
 \ell^{11}_{--}\bigl(1-(d^{14}_{++})^2\bigr)}
-\frac{d^{11}_{++}d^{12}_{++}}
{\ell^{24}_{+-}d^{12}_{-+}\bigl(1-(d^{11}_{++})^2\bigr)
 d^{14}_{+-}d^{24}_{++}}
\\
&+\frac{d^{11}_{++}}
{\ell^{24}_{+-}d^{24}_{++}\bigl(1-(d^{11}_{++})^2\bigr)d^{14}_{+-}}
+\frac{(d^{24}_{-+}-1)d^{24}_{++}}
{\bigl(1-(d^{24}_{++})^2\bigr)d^{12}_{+-}
 \bigl(1-(d^{14}_{++})^2\bigr)d^{14}_{+-}},\\[2mm] \\
%\end{align*}
%\begin{align*}
T'_3={}&
-\frac{(d^{44}_{++})^2d^{12}_{++}}
{\bigl((d^{44}_{++})^2-1\bigr)d^{14}_{+-}d^{12}_{+-}
 d^{24}_{++}\ell^{12}_{--}}
-\frac{(d^{11}_{++}+d^{24}_{++})d^{12}_{+-}}
{\ell^{14}_{++}d^{14}_{+-}d^{24}_{++}
 \bigl((d^{12}_{+-})^2-1\bigr)}
\\
&-\frac{d^{12}_{-+}-d^{44}_{++}}
{\ell^{44}_{++}d^{12}_{+-}d^{24}_{++}d^{14}_{+-}
 \bigl((d^{12}_{+-})^2-1\bigr)}
-\frac{d^{12}_{-+}-d^{44}_{++}}
{\bigl((d^{44}_{++})^2-1\bigr)d^{12}_{+-}d^{14}_{+-}
 d^{24}_{++}(d^{12}_{-+}-1)}
\\
&+\frac{(d^{12}_{++}-1)d^{12}_{+-}}
{d^{14}_{+-}d^{24}_{++}\bigl(1-(d^{14}_{++})^2\bigr)
 \bigl((d^{12}_{+-})^2-1\bigr)}
+\frac{(d^{44}_{++})^2}
{\bigl((d^{44}_{++})^2-1\bigr)d^{24}_{++}d^{12}_{+-}
 d^{14}_{+-}\ell^{12}_{--}}\\
&-\frac{d^{14}_{++}(d^{12}_{++}-1)(d^{12}_{+-})^2}
{\bigl(1-(d^{14}_{++})^2\bigr)d^{14}_{+-}d^{12}_{+-}d^{24}_{++}
 \bigl((d^{12}_{+-})^2-1\bigr)},\\[2mm]
%\begin{align*}
T'_4={}&
-\frac{1}{\ell^{13}_{+-}}
+\frac{1}{\ell^{13}_{+-}\bigl((d^{12}_{+-})^2-1\bigr)d^{23}_{+-}}
-\frac{1}{\ell^{13}_{-+}}
-\frac{1}{\ell^{13}_{-+}d^{12}_{+-}\bigl((d^{23}_{+-})^2-1\bigr)}\\
&-\frac{(d^{12}_{+-})^2\Delta_{123}}
{\ell^{13}_{+-}\bigl((d^{12}_{+-})^2-1\bigr)
 \ell^{23}_{--}d^{12}_{+-}d^{23}_{+-}}
+\frac{\Delta_{123}(d^{23}_{+-})^2}
{d^{12}_{+-}d^{23}_{+-}\bigl((d^{23}_{+-})^2-1\bigr)
 \ell^{23}_{--}\ell^{13}_{-+}}\\
&+\frac{\Delta_{123}(d^{12}_{++}-1)}
{d^{12}_{+-}d^{23}_{+-}\ell^{23}_{--}\ell^{13}_{-+}
 \ell^{23}_{+-}\bigl((d^{12}_{+-})^2-1\bigr)}
+\frac{\Delta_{123}(d^{12}_{++}-1)}
{d^{12}_{+-}d^{23}_{+-}\ell^{23}_{--}\ell^{13}_{-+}
 \bigl((d^{23}_{+-})^2-1\bigr)(d^{12}_{-+}-1)}
\\
&-\frac{2(d^{12}_{++}-1)(d^{12}_{+-})^2\Delta_{123}}
{\ell^{13}_{+-}\bigl((d^{12}_{+-})^2-1\bigr)d^{12}_{+-}d^{23}_{+-}
 \ell^{23}_{--}\ell^{13}_{-+}},\\[2mm]
%\end{align*}
%\begin{align*}
T'_5={}&
\frac{(d^{12}_{+-})^2d^{23}_{-+}-(d^{13}_{++}-1)}
{\ell^{23}_{--}\bigl((d^{12}_{+-})^2-1\bigr)d^{23}_{-+}}
-\frac{(d^{23}_{+-})^2\Delta_{123}}
{\bigl((d^{23}_{+-})^2-1\bigr)d^{23}_{+-}d^{12}_{+-}\ell^{12}_{--}}\\
&-\frac{\Delta_{123}}
{\bigl((d^{23}_{+-})^2-1\bigr)(d^{12}_{+-}-1)
 d^{23}_{+-}d^{12}_{+-}}
+\frac{1}{\ell^{12}_{--}}
+\frac{d^{23}_{++}-1}
{\bigl((d^{23}_{+-})^2-1\bigr)d^{12}_{-+}\ell^{12}_{--}}
\\
&+\frac{\Delta_{123}(d^{12}_{++}-1)(d^{12}_{+-})^2}
{\ell^{23}_{--}d^{23}_{+-}d^{12}_{+-}
 \bigl((d^{12}_{+-})^2-1\bigr)\ell^{12}_{--}}
+\frac{(d^{12}_{++}-1)\Delta_{123}}
{\ell^{12}_{--}\bigl((d^{12}_{+-})^2-1\bigr)
 d^{23}_{+-}d^{12}_{+-}\ell^{23}_{-+}}.
\end{align*}
\endgroup
\begin{Lemma}
\label{lem:primed-rational-identities}
As elements of $\mathbb C(x_1,x_2,x_3,x_4)$, the rational functions
above satisfy $
 T'_1=T'_2=T'_3=T'_4=T'_5=0$.
\end{Lemma}

\begin{proof}
After bringing each $T'_i$ to a common denominator, direct expansion
shows that its numerator vanishes identically. 
\end{proof}

\end{document}